\documentclass[
    11pt,		
    dvipsnames, 
    oneside,
    reqno,  %right sided equation number
]{amsart}
\usepackage[foot]{amsaddr}

\usepackage[english]{babel}
\usepackage[utf8]{inputenc}
\usepackage[T1]{fontenc}
\usepackage{amsmath,amsthm,amssymb}
\usepackage[a4paper, margin = 2.5cm]{geometry}
\usepackage{appendix}
\usepackage{enumitem}
\usepackage{cite}
\usepackage{tabularx}
\usepackage{float} % allows precise tabular positioning
\usepackage{graphicx}
\usepackage{esint} %averaged integrals
\usepackage{mathrsfs} %script fonts \mathscr
\usepackage{multicol}
\usepackage{subcaption}[labelformat=empty]
\usepackage[
        colorlinks=true, 
        linkcolor=black, 
        citecolor=black, 
        urlcolor = black]{hyperref}
\usepackage{xurl}
\usepackage{tikz}
\usetikzlibrary{calc,decorations.pathreplacing,shadows,arrows.meta,patterns}
\usepackage{macros} % My macros

\definecolor{myorange}{HTML}{E69F00}   % replaces "orange"
\definecolor{myblue}{HTML}{0072B2}     % replaces "OliveGreen" (better contrast than a second warm/green tone)

\theoremstyle{plain}
\newtheorem{theorem}{Theorem}
\newtheorem{proposition}[theorem]{Proposition}
\newtheorem{lemma}[theorem]{Lemma}
\newtheorem{corollary}[theorem]{Corollary}

\theoremstyle{definition}
\newtheorem{definition}[theorem]{Definition}
\newtheorem{assumption}{Assumption}

\theoremstyle{remark}
\newtheorem{remark}[theorem]{Remark}
\theoremstyle{remark}

\newtheorem{example}[theorem]{Example}

\numberwithin{theorem}{section}
\numberwithin{equation}{section}  

\hypersetup{
    pdftitle={Structure theory of BV functions and finite perimeter sets on Riemannian manifolds},
    pdfauthor={Péter Koltai, Kathrin Völkner},
    }

\title{Structure theory of \textit{BV} functions and
finite perimeter sets on Riemannian manifolds 
}

\author{P\'eter Koltai$^1$}
\address{$^1$Universit\"at Bayreuth, Bayreuth, Germany}
\author{Kathrin V\"olkner$^2$}
\address{$^2$Freie Universit\"at Berlin, Berlin, Germany. \upshape{Corresponding author. kathrin.voelkner@fu-berlin.de}}

\thanks{This research was funded by an Einstein Visiting Fellowship of the Einstein Foundation Berlin. K.V. was supported by the Deutsche Forschungsgemeinschaft under Germany's Excellence Strategy – The Berlin Mathematics Research Center MATH+ (EXC-2046/2, project ID: 390685689).\\
K.V. would like to thank Batu Güneysu for an insightful exchange regarding covector measures and the effect of curvature conditions on $BV$ functions, and Giona Veronelli for helpful remarks on Example~\ref{ex: decomposition of Lipschitz domains}. We further thank Gary Froyland and Elena Mäder-Baumdicker for pointing us to valuable references on related works.}
\begin{document}

\begin{abstract}
We develop the theory of functions of bounded variation and the structure theory of finite-perimeter sets on arbitrary Riemannian manifolds without relying on global curvature bounds or completeness of the manifold. To this end, we build a localization framework  that permits a synthesis of techniques from Euclidean geometric measure theory and analysis on metric measure spaces while preserving genuinely Riemannian features, such as polar and normal vector fields, reduced boundaries, and approximate tangent spaces. 
As a consequence, we recover key results of the Euclidean theory, such as a differentiation theorem for a Riemannian generalization of vector-valued measures and the structure theorems by De~Giorgi and Federer, formulated intrinsically on Riemannian manifolds. This makes a large portion of the classical Euclidean \textit{BV} theory available to the Riemannian setting. 
 
We demonstrate this by providing the theoretical background for boundary value problems on domains in manifolds, including trace and Gauss--Green theorems. Finally, we prove an approximation result for finite-perimeter sets in a strict sense that respects a prescribed Dirichlet boundary portion of a given ambient domain and apply this to prove Gamma-convergence of a family of energy functionals for capillarity problems with mixed boundary conditions.
\end{abstract}
\maketitle

\tableofcontents

\section{Introduction}

Functions of bounded variation ($BV$ \textit{functions}) and the structure theory of sets of finite perimeter in $\R^m$ are central to geometric analysis, the calculus of variations, and geometric measure theory. 
Classical results such as De~Giorgi's structure theorem and Federer's characterization of sets of finite perimeter are essential tools
in minimal surface theory and the study of isoperimetric problems. Generalizations of such Euclidean results to Riemannian manifolds have gained increasing attention over the past years. For instance, quantitative isoperimetric problems were studied first in the special cases of the sphere \cite{BogDuzSch15_SharpQuantitativeIsoperimetric} and of hyperbolic space \cite{BogDuzFus17_QuantitativeIsoperimetricInequality} before it was shown in \cite{ChoEngSpo22_RiemannianQuantitativeIsoperimetric} that an analog of the classical Euclidean result fails on general Riemannian manifolds. On the other hand, applications of $BV$ functions in areas such as free boundary problems or shape optimization naturally occur on curved spaces, e.g.~\cite{NarChaTro24_MatchingProblemFunctional, BroBurGil24_SpectralTotalvariationProcessing}. 

The study of this type of problems requires an in-depth knowledge of $BV$ functions and the structure theory of finite
perimeter sets on Riemannian manifolds. To the best of our knowledge, an account of this is not available in the existing literature.  Yet, the Riemannian generalizations are far from trivial: For instance, a naive definition of the reduced boundary of finite perimeter sets would assume a globally trivial tangent bundle, something that can be done in hyperbolic space but is clearly false in general. It is therefore a natural question to what extent the  theory can be formulated intrinsically on Riemannian manifolds without imposing restrictive global assumptions or relying on a higher dimensional Euclidean ambient space.\\

In this context, two main lines of work have emerged in recent years.
On the one hand, treatments of \textit{$BV$ functions on Riemannian manifolds} are mainly concerned with the heat semigroup approximation of the total variation. These works rely on global curvature bounds or completeness of the manifold~\cite{GunPal15_FunctionsBoundedVariation,  MirPalParPre07_HeatSemigroupFunctions,  CarMau07_NoteBoundedVariation, Gen26_FinitePerimeterRiemannian}, or they consider merely compact manifolds~\cite{AloBau23_AnotherHeatSemigroup, KreMor19_FractionalSobolevNorms}. 
On the other hand, the theory of \textit{$BV$ functions on metric measure spaces}, which  can be viewed as a further generalization, has
been developed extensively under restrictive global assumptions:   here,  one typically assumes that the space is complete,  volume doubling,  and satisfies a $1$-Poincar\'e inequality,
yielding powerful analytic tools in a very general framework. We only mention  \cite{Mir03_FunctionsBoundedVariation,  AmbMirPal04_SpecialFunctionsBounded,  AmbDi14_EquivalentDefinitionsBV,  LahSha18_TraceTheoremsFunctions} and underline that this list is far from being exhaustive.
However, these global structural hypotheses  
are not satisfied on arbitrary (noncompact) Riemannian manifolds.
Moreover, several geometric objects from the Euclidean theory that can be given an intrinsic meaning on smooth manifolds,
such as a notion of vector measures, normal vector fields, reduced boundaries and approximate tangent spaces to rectifiable sets, do not arise naturally in the general metric measure space setting.\\

We also mention a related but separate line of work on weakly differentiable functions in the setting of rectifiable varifolds, see \cite{Men16_WeaklyDifferentiableFunctions, Men16_SobolevFunctionsVarifolds}. Rectifiable varifolds generalize rectifiable subsets of Euclidean space and, in particular, smoothly embedded submanifolds. However, the theory is formulated in an ambient Euclidean space and is therefore not intrinsic to the manifold. While some results have parallel aspects with the structure theory developed here (compare, e.g.,~\cite[Corollary~12.2]{Men16_WeaklyDifferentiableFunctions} with Theorem~\ref{thm: De~Giorgi}), they concern a genuinely different class of functions, which does not contain functions of bounded variation (see~\cite[Remark 8.10]{Men16_WeaklyDifferentiableFunctions}); moreover, the full-dimensional setting of the present paper is structurally excluded by the standing assumption of positive codimension.\\

The purpose of this work is to bridge the gap between the Euclidean setting and the metric measure space framework by developing  localization
techniques for functions of bounded variation and sets of finite perimeter on
smooth Riemannian manifolds.
Our approach is based on the characterization of $BV$ functions via a Riemannian analog of vector-valued measures from \cite{GunPal15_FunctionsBoundedVariation} and exploits the fact that every relatively compact domain of a smooth manifold satisfies local
doubling and local Poincar\'e inequalities.
This allows us to combine methods from Euclidean geometric measure theory
with results from analysis on metric measure spaces while retaining
genuinely Riemannian information encoded in
covector measures.
In particular, throughout the paper we avoid global curvature assumptions and do not require  completeness of the manifold. A distinctive feature of our approach is that we give intrinsic (coordinate-free) Riemannian  definitions while simultaneously developing a localization machinery (cf. Lemma~\ref{lem: differentiable points in charts}, Lemma~\ref{lem: approximate tangent space independent of chart}, Lemma~\ref{lem: Riemannian approximate tangent space chart wise} and Corollary~\ref{cor: density of of sets in normal coordinates}). In this way, the present work provides a unified framework that connects both
perspectives and makes a large portion of the Euclidean and metric measure
space literature available on arbitrary smooth Riemannian manifolds.\\

After introducing some notation and preliminaries in Section \ref{sec: preliminaries}, the paper is organized as follows: \\

\textit{BV functions and covector measures on Riemannian manifolds.}
As mentioned above, a central ingredient of our framework is the characterization of $BV$ functions via  distributional
derivatives that can be represented in terms of a Riemannian analog of vector-valued measures. In Section~\ref{sec: BV functions and covector measures} we base our theory on~\cite{GunPal15_FunctionsBoundedVariation} and the representation theorem for such measures proved therein, and further develop the theory of spaces of (locally) finite covector measures on Riemannian manifolds and their dependence on the choice of the Riemannian metric. We go on to address convergence and compactness properties of
covector measures and $BV$ functions, including weak$^\ast$ compactness
(Lemma~\ref{lem: properties of weak and weak$^*$ convergence}), strict approximation by smooth functions
(Proposition~\ref{prop: approximation by smooth functions}), and compactness in $BV_{\mathrm{loc}}$
(Corollary~\ref{Cor: compactness in BV_loc}).
\\

\textit{Differentiation theorems for Radon and covector measures.}
A key analytic tool in our arguments is a differentiation theorem for
covector measures.
In Section~\ref{sec: differentiation theorems} we first establish a Riemannian version of the
Lebesgue--Besicovitch--Federer differentiation theorem for Radon measures
(Theorem~\ref{thm: Lebesgue--Besicovitch--Federer}), based on Vitali relations generated by closed Riemannian balls.
As a consequence, Theorem~\ref{thm: Besicovitch Theorem for vector measures} yields a differentiation theorem for covector
measures by recovering the density vector field of a given covector measure with respect to an arbitrary Radon measure from the weak$^*$-limit of a sequence of covector measures.  
These results provide the main localization mechanism of the paper:
they allow us to recover pointwise geometric information from measure-theoretic data without invoking global curvature assumptions.
A chartwise formulation (Lemma~\ref{lem: differentiable points in charts}) later shows that the differentiation process is
compatible with local coordinates, which is essential for transferring
Euclidean results to the Riemannian setting.
\\

\textit{Structure theory for sets of finite perimeter.}
Section~\ref{sec: structure theory} develops the structure theory of finite perimeter sets. Using the differentiation results from Section \ref{sec: differentiation theorems}, we introduce a Riemannian notion of
reduced boundary  based on the existence of a
measure-theoretic inner normal vector.
This definition is consistent with the classical Euclidean one and retains
the geometric information encoded in the polar vector field of the perimeter
measure. We proceed with a definition of approximate tangent spaces to rectifiable sets via blow-ups in Riemannian manifolds leading up to the proofs of the main structure theorems:   Theorem~\ref{thm: De~Giorgi} is a Riemannian version of De~Giorgi's theorem and states that the reduced boundary $\partial^*E$ of  a set $E$ of (locally) finite perimeter is (locally) rectifiable and the perimeter measure is given by the restriction of the codimension-$1$ Hausdorff measure to $\partial^*E$. In addition, it establishes the connections between the measure-theoretic normal vectors and approximate tangent spaces at $\partial^*E$.    Theorem~\ref{thm: Federer}  generalizes  a Euclidean result by Federer, stating that the measure-theoretic boundary of a finite perimeter set is contained in the set of points of density $1/2$ up to sets of codimension-$1$ Hausdorff measure zero and establishes  the equivalence in Hausdorff measure of the reduced and measure-theoretic boundary.
\\

{The last section is dedicated to applications of the previously established theory:}\\

\textit{Theory for  boundary value problems and a mixed boundary capillarity problem.} In Section \ref{sec: BVP and capillarity} we present Riemannian versions of the central tools for the treatment of boundary value problems on domains in Riemannian manifolds. This includes conditions for the existence and continuity of extension and trace operators, as well as local compactness and smooth approximation up to the boundary. We employ the localization principles established earlier to deduce a Riemannian Gauss--Green 
formula (Corollary~\ref{cor: Gauss--Green}), complementing existing
approaches that rely on global geometric assumptions.

In Section \ref{sec: mixed bvp and strict approximation} we illustrate the flexibility of the developed theory  by deriving a framework for the treatment of variational problems with \textit{mixed boundary conditions on rough domains} in Riemannian manifolds.
Such problems arise in various   contexts, ranging from generalized mixed boundary Cheeger problems used to model landslides \cite{IonLac05_GeneralizedCheegerSets,Dwe25_WeightedTotalVariation} to the study of coherent sets in dynamical systems \cite{FroKol23_DetectingBirthDeath,atnip2024inflated}. Here,  the subspace of functions with zero boundary conditions on the given Dirichlet boundary portion (or the subspace of sets with  a `no touch' condition thereon) is not closed under weak$^\ast$  convergence, and more generally, traces are not lower semicontinuous with respect to  weak$^\ast$ convergence. Thus, one cannot expect that minimizers exist in the constrained function spaces. 
We employ the previously developed structure theory to prove a  strict interior approximation result (Theorem~\ref{thm: strict Dirichlet approximation of finite perimeter sets}) on the level of relative perimeters  that respects a Dirichlet boundary condition on a fixed boundary portion under very general regularity assumptions on the domain.
In variational problems such as those mentioned above, this result has two main applications: first, it permits the relaxation of boundary-constrained minimization problems to well-posed minimization problems on the full space of finite perimeter sets, by adding a penalization term to  a given energy functional and provides a general tool to rigorously prove equivalence of the constrained and unconstrained minimization problem. Such techniques are applied,
e.g. in \cite{Dwe25_WeightedTotalVariation, IonLac05_GeneralizedCheegerSets} using  related results for Lipschitz domains in Euclidean space. Second, the approximation  serves as a means to construct recovery sequences for Gamma-convergence problems with mixed boundary conditions, as we shall demonstrate in the last section.  

In Section~\ref{sec: mixed boundary capillarity} we formulate a capillarity problem in a rough domain with mixed boundary conditions in a manifold. Capillarity problems in domains model equilibrium states of liquids in solid containers and the formation of droplets on surfaces and have been studied extensively in recent years, e.g. \cite{FusJulMorPra25_IsoperimetricInequalityCapillary,PasPoz24_QuantitativeIsoperimetricInequalities,PhiMag14_RegularityFreeBoundaries, LiZhoZhu25_MinmaxTheoryCapillary, HonSat23_CapillarySurfacesStability}. The mixed boundary conditions in the problem under consideration correspond to an inhomogeneous container whose walls consist of two parts with different adhesion properties, for instance where one part is treated with a hydrophobic coating while the other part is \textit{wetting}. After proving existence of minimizers using the theory developed in Sections \ref{sec: gauss green} and \ref{sec: mixed bvp and strict approximation}, we consider a family of mixed boundary capillarity problems depending on a varying adhesion parameter and show that the energy functionals Gamma-converge under suitable convergence of the adhesion parameter by employing Theorem~\ref{thm: strict Dirichlet approximation of finite perimeter sets} to construct a recovery sequence.

% %%%%%%%%%%%%%%%%%%%%%%%%%%%%%%%%%%%%%%%%%%%%%%%%%%%%%%%%%%%%
\section{Preliminaries}\label{sec: preliminaries}
%%%%%%%%%%%%%%%%%%%%%%%%%%%%%%%%%%%%%%%%%%%%%%%%%%%%%%%%%%%%
Throughout this paper $M\equiv (M,g)$ is a smooth, connected and oriented $m$-dimensional Riemannian manifold. In the sequel we assume all our chart domains to be relatively compact.
If $\calC$ is a regularity class, we let $\Gamma_{\calC}(M;TM)$ denote the space of sections in the tangent  bundle  of class $\calC$,  and we use the same convention for open subsets of $M$, interpreting them as open submanifolds in the usual way. If  the regularity $\calC$ is not specified, then  $\Gamma(M;TM)$  is the space of rough sections in the sense of \cite[Chapter 10]{Lee12_IntroductionSmoothManifolds}. We use similar notation for sections in the cotangent bundle $T^*M$.

We write $(X,Y)(x):=(X,Y)_g (x):=g_x(X(x),Y(x))$ for $x\in M$ and  sections $X,Y$ in $TM$  where  we will often omit the pointwise dependence of sections in situations where this causes no confusion.  Since $g$ is fixed throughout,  we will  indicate the dependence of Riemannian objects on $g$ by a subscript only when necessary, for instance when other Riemannian metrics aside from $g$ appear. In particular, the Riemannian ball around $x\in M$ of radius $r>0$ will be denoted by $B(x,r) = B_g(x,r)$.

Let $\hat{g}:TM\to T^*M$ be the canonical Riemannian bundle isomorphism given by the (pointwise) action $\hat{g}(X)(Y)=(X,Y)_g$ and its inverse $\hat{g}^{-1}: T^*M \to TM$ satisfying $(\hat{g}^{-1}(\alpha),X)_g = \alpha(X)$. Then the canonical bundle metric on the cotangent bundle $T^*M$ is denoted by $g\inv$ and satisfies $(\alpha,\beta)_{g\inv} = \alpha(\hat{g}\inv(\beta))$.

The spaces of Borel measurable and bounded Borel measurable functions on $U$ will be denoted by $ {B}(U)$ and $ {B}_b(U)$, and unless specified otherwise, by \textit{measurable} we will always mean Borel measurable. 
For open sets $U\subset M$ we denote the space of $p$-integrable functions with respect to the Riemannian  volume measure $\vol = \vol_g$ by  $L^p(U) =L^p(U,g) $ and write $ \norm[p,U]{\cdot} = \norm[p,U,g]{\cdot}$ for the corresponding norm. If $E\subset M$ is a measurable set, we will sometimes write $ |E|_g =|E|:=\vol(E) $.  

The  divergence of a $C^1$-vector field $X$ on $M$  with respect to the volume measure is uniquely defined as the  function $\ddiv(X) = \ddiv[g](X)\in C(M)$ satisfying
\begin{equation*}
    \int_M \phi \ddiv(X) \d \vol = -\int_M d\phi(X) \d \vol \tforall \phi\in C_c^\infty(M),
\end{equation*}
where $d\phi$ is the exterior derivative of $\phi$.
If $ X=\sum_i X^i \delx{i}$ in local coordinates, then 
\begin{equation}
    \ddiv (X) = \sum_i \frac{1}{\sqrt{\det g}} {\delx{i}} \sqrt{\det g} X^i.
\end{equation}
For open sets $U$ the equality 
\begin{equation*}
        \int_U \phi \ddiv(X) \d \vol = -\int_U (\nabla \phi,X)  \d \vol 
\end{equation*}
is satisfied for $\phi\in C^1(U)$ and $X\in \Gamma_{C^1}(U;TU)$ if $\phi$ or $X$ is compactly supported in $U$. 
\begin{remark} 
    By the previous considerations, $-\ddiv $ is the formal adjoint of the gradient operator with respect to the $L^2$-inner product of vector fields on $M$. Thus, given a distribution $T:C_c^\infty(U )\rar \R$, we can define its distributional gradient $(\nabla  T)(X) = -T(\ddiv (X))$ and for a regular distribution $T_u$ induced by $u\in L^1_{\loc}(U )$, we can view $Du = D_g{u} $ as a distributional derivative defined by
    \begin{equation*}
        Du[X] :=  \nabla T_u(X) = - \int_U  u\ddiv (X) \d \vol ,  \quad X\in \Gamma_{C_c^\infty}(U ;TU ),
    \end{equation*}
     and $Du[X]$ is well-defined for $X\in \Gamma_{C_c^1}(U ;TU )$. 
\end{remark}

We will repeatedly rely on results from  $BV$-theory on metric measure spaces which often relies on the validity of a $1$-Poincar\'e inequality and a   doubling condition on the measure. It is well-known that both assumptions are valid \textit{globally} on complete Riemannian manifolds with  {nonnegative} Ricci curvature  (see e.g. \cite[Chapter 5.6.3]{Sal02_AspectsSobolevtypeInequalities}). For our purposes, it will often be sufficient to rely on these metric measure space assumptions \textit{locally} on arbitrary Riemannian manifolds. This is justified by the following:

\begin{remark}
\label{rmk: local doubling condition and Poincare inequality   in arbitrary Riemannian manifolds}
    If $\Omega\subset M$ is relatively compact, then we can always view it as a domain in a closed Riemannian manifold:  take some smoothly bounded relatively compact domain $M_0$  in $M$ such that $\Omega$ is compactly contained in $M_0$. Then its closure $\bar{M}_0$ is a compact manifold with smooth boundary and  can be extended to a closed manifold  $\hat{M}$, its so-called \textit{double}, see~\cite[Example~9.32]{Lee12_IntroductionSmoothManifolds}. Since $\hat{M}$ is closed, the Riemannian metric can be extended from the closure of $\Omega$ to all of $\hat{M}$ by \cite[Lemma 10.12]{Lee12_IntroductionSmoothManifolds}. 
    In particular, as every closed manifold has bounded  curvature,  \cite[Theorems 5.6.4, 5.6.5]{Sal02_AspectsSobolevtypeInequalities} can be applied to $\Omega$. This means that
    \begin{itemize}
        \item the volume measure is  doubling on every relatively compact domain $\Omega$ in $M$,
        \item every relatively compact domain $\Omega$ satisfies a $1$-Poincar\'e inequality for arbitrary balls $B \subset \Omega$,
    \end{itemize} 
    with  constants depending on $\Omega$. If $M$ is  complete, then the same statements apply to arbitrary bounded domains.
\end{remark}

%%%%%%%%%%%%%%%%%%%%%%%%%%%%%%%%%%%%%%%%%%%%%%%%%%%%%%%%%%%%
\section{BV functions and covector measures on Riemannian manifolds}
\label{sec: BV functions and covector measures}
%%%%%%%%%%%%%%%%%%%%%%%%%%%%%%%%%%%%%%%%%%%%%%%%%%%%%%%%%%%%

% For an open subset $\Omega\subset M$, 
The \emph{total variation} of a function $u\in L^1_{\loc}(\Omega)$ on an open subset $\Omega\subset M$ is defined by
\begin{equation}\label{def: variation}
   \mathrm{Var}_\Omega(u) 
   % = \mathrm{Var}_{g,\Omega}(u)  := \sup \set{\int_\Omega u\ddiv[g] (X) \d {\vol_g}: X\in \Gamma_{C^1_c}(\Omega;T\Omega),\, |X|_g\leq 1 }
   := \sup \set{\int_\Omega u\ddiv (X) \d {\vol}: X\in \Gamma_{C^1_c}(\Omega;T\Omega),\, |X|\leq 1 }\in [0,\infty].
\end{equation}
The total variation is lower semicontinuous with respect to~$L^1_\loc$. That is, if a sequence $(u_n)$  converges to  $u$ in  $L^1_\loc(\Omega)$, then for all open subsets $U\subset\Omega$ one has
    $$\liminf_n \mathrm{Var}_\Omega(u_n)(U) \geq \mathrm{Var}_\Omega(u)(U).$$ 
 We write $U\subset\subset \Omega$ to express that a subset $U$ is relatively compact in~$\Omega$. If  $\mathrm{Var}_{U}(u) < \infty$  holds for all open subsets $U\subset\subset\Omega$, we say that $u$ has \textit{locally bounded variation on $\Omega$} and write~$u\in BV_{\loc}(\Omega)$.
 If $u\in L^1(\Omega)$ and $\mathrm{Var}_\Omega (u)<\infty$, we say that $u$ is a \emph{function of bounded variation on $\Omega$}. The space  $BV(\Omega)$  of functions of bounded variation on $\Omega$ becomes a Banach space with the norm $\norm[BV(\Omega)]{u }:= \norm[1,\Omega]{u }+\mathrm{Var}_{\Omega}(u)$.

In the Euclidean setting, functions of bounded variation can be characterized as those integrable functions whose distributional derivative can be represented by a vector-valued Radon measure. We adapt the Riemannian analog introduced in \cite{GunPal15_FunctionsBoundedVariation} in terms of so-called \textit{generalized vector measures} that act on $1$-forms to measures that act on vector fields. We denote by $\calM(\Omega)$ and  $\calM_\loc(\Omega)$ the spaces of finite  Radon measures and  Radon measures on~$\Omega$, respectively. Note that Radon measures are always locally finite.

\begin{definition}\label{def: covector measure} 
   A   \textit{(locally) finite covector measure}  on an open set $\Omega\subset M$  is defined as a pair   $\nu  = (\mu,\sigma) $ consisting of a  (locally) finite  Radon measure $\mu$ on $\Omega$     and a (locally) bounded  Borel measurable section $\sigma :\Omega\to  T^*\Omega$. For $\nu=(\mu,\sigma)$ we define 
   \begin{equation*}
       \sigma^\nu(x):= \begin{cases}
           \begin{aligned}
               &\frac{\sigma(x)}{|\sigma(x)|}, &\text{ if }\sigma(x)\neq 0 \\
           &0, & \text{otherwise}
           \end{aligned}
       \end{cases}
       \qand |\nu| := |\sigma| \mu,
   \end{equation*}
   and identify two pairs $\nu = (|\nu|  , \sigma^\nu )$ and $\tilde{\nu} = (|\tilde\nu|, \sigma^{\tilde{\nu}} )$   if $|\nu|  = |\tilde\nu| $ as Borel measures and $\sigma^\nu (x) =\sigma^{\tilde{\nu}} (x)$ for $|\nu| $-almost every $x\in\Omega$.
   We denote the spaces of finite  and locally finite covector measures  on $\Omega$ by $ \calM(\Omega;T^*\Omega)$ and $\calM_\loc (\Omega; T^*\Omega)$, respectively, and remark that clearly one always has $ \calM(\Omega;T^*\Omega)\subset\calM_\loc (\Omega; T^*\Omega)$. Often, we simply refer to $\nu\in\calM_\loc(\Omega;T^*\Omega)$ as a covector measure. 
\end{definition}

By~\cite[Theorem 1]{GunPal15_FunctionsBoundedVariation} the classical Riesz--Markov representation theorem for measures extends to  covector measures in the following way: 

\begin{theorem}[Riesz--Markov representation theorem for finite covector measures]\label{thm: Riesz Markov}
    The map
    \begin{equation}\label{eqn: generalized Riesz isomorphism}
    \begin{aligned}
        \calM(\Omega;T^*\Omega) &\to (\Gamma_{C_0}(\Omega;T\Omega), \norm[\infty]{\cdot})^*  \\
        \nu &\mapsto T_\nu \\ 
    \end{aligned}
        \text{ with } \quad
        T_\nu[X]
        :=\int_\Omega \sigma^\nu (X)  \d |\nu| 
    \end{equation}
is a well-defined bijection, and $|\nu| (\Omega)$ equals the  operator norm of $T_\nu$, that is,
\begin{equation}\label{eqn: tv equals operator norm}
     |\nu| (\Omega) = \norm[\Omega,\infty,*]{T_\nu} =  \sup \set{\int_\Omega \sigma^\nu(X) \,  d|\nu|  :\, X\in \Gamma_{C_0}(\Omega;T\Omega),\, |X| \leq 1}.
\end{equation}
\end{theorem}
Here, $C_0$ stands for the space of continuous functions {vanishing at infinity}.
We deduce the Riesz--Markov representation theorem for locally finite covector measures: 
\begin{corollary}\label{cor: riesz markov for loc fin covector measures}
    For every linear map $T:\Gamma_{C_c}(\Omega;T\Omega)\to \R$ satisfying
    \begin{equation}\label{eqn: locally bounded linear form on vector fields}
        \norm[U,\infty, *]{T} <\infty\text{ for every open set } U\subset\subset \Omega,
    \end{equation}
      there exists a unique $\nu =(\nu, \sigma)\in \calM_\loc (\Omega; T^*\Omega)$ such that 
    \begin{equation*}
        T[X] = \int_\Omega \sigma(X)\d\mu \tforall X\in \Gamma_{C_c}(\Omega;T\Omega).
    \end{equation*}
    Where $T\mapsto \nu$ is a bijection and one has $\norm[U,\infty, *]{T} = |\nu|(U)$.
\end{corollary}
\begin{proof}
    To see this, note that by density of $C_c(U)$ in $C_0(U)$,  condition \eqref{eqn: locally bounded linear form on vector fields} is satisfied if and only if the restriction of $T$  to $\Gamma_{C_0}(U;TU)$ is  bounded   with respect to  uniform convergence  for all   open $U\subset\subset\Omega$. On the other hand,
    $\nu =(\mu,\sigma)$ lies in $\calM_\loc (\Omega; T^*\Omega)$  if and only if the restriction $\nu\llcorner U = (\mu\llcorner U, \sigma_{|U})$ (see also the definition at the end of this section) is a finite covector measure on each relatively compact open set $U\subset\subset \Omega$. 
    Now one can apply  the Riesz--Markov theorem for finite covector measures on $U$ to the respective restrictions of $T$ and $\nu$.
\end{proof}

    From now on, we identify $\calM(\Omega;T^*\Omega)$ with $(\Gamma_{C_0}(\Omega;T\Omega))^*$ and $\calM_\loc (\Omega; T^*\Omega)$ with $(\Gamma_{C_c}(\Omega;T\Omega))^*$.  If $\nu$ is a covector measure on $\Omega$, the measure $|\nu|$ is called its \textit{total variation measure}, and we define the \textit{polar  vector field} $\Sigma^\nu:\Omega\to T\Omega$ of $\nu$  as the Borel measurable vector field given  by 
    \begin{equation}\label{eqn: polar  vector field wrt g}
        \Sigma^\nu:=  \hat{g}\inv(\sigma^\nu).
    \end{equation}
    Then, by Definition~\ref{def: covector measure}, one has that $|\Sigma^\nu(x)| = 1$ for $|\nu|$-almost every $x\in \Omega$, and we write 
    \begin{equation*}
        \nu[X] 
        :=  \int_\Omega X\cdot d \nu
        :=  \int_\Omega \sigma^\nu(X) \d |\nu|
        =  \int_\Omega (\Sigma^\nu,X) \d |\nu|,
    \end{equation*}
    where the action $\nu[X]\in [-\infty,\infty]$ is well-defined for Borel sections $X:\Omega\to T\Omega$.
     We identify $\nu = (|\nu|,\sigma^\nu) = (|\nu|,\Sigma^\nu)$ and refer to the last representation as the \textit{polar  decomposition} of $\nu$. It follows immediately that on $\R^m$, the space of covector measures coincides with the usual space of vector-valued measures, since the cotangent and tangent bundles are globally trivial.
    Finally, we mention that the total variation measure of an open subset $U\subset\Omega$  is given by 
    \begin{equation*} 
     |\nu| (U) =\sup \set{\int_\Omega (\Sigma^\nu,X) \,  d|\nu|  :\, X\in \Gamma_{\calC}(U;TU),\, |X| \leq 1}.
    \end{equation*}
    Here, the supremum can be taken over vector fields of class $\calC$, with $\calC$ being any dense subspace of $C_0(U)$, such as $C_0^1(U)$, $C_c^1(U)$ or $C_c^\infty(U)$, and by separability of $C_0(U)$, one can pick a countable dense subset as well. As a result of these considerations, we immediately obtain the announced characterization of functions of bounded variation:

\begin{theorem}[Characterization of BV] \label{thm: characterization of BV}
       A function $u\in L^1_{\loc}(\Omega)$ has (locally) bounded variation  if and only if there exists a (locally) finite covector measure $D u$ with  polar  decomposition $(|D  u|,\Sigma^u)$ on $\Omega$  such that
        \begin{equation*} 
           \int_\Omega (X,\Sigma^u) \d |D  u| = -\int_\Omega u\ddiv(X) \d \vol \tforall X\in \Gamma_{C^1_c}(\Omega, T\Omega).
        \end{equation*}
       The covector measure $D u$  and $|D u|(\Omega)$ coincides with the total variation of $u$ as defined in \eqref{def: variation}. 
\end{theorem}

\begin{example}
    If $u\in W^{1,1}(\Omega)$ (respectively, $ W^{1,1}_\loc(\Omega)$), and $\nabla u$ denotes the weak gradient of $u$, then $-\int_\Omega u\ddiv(X) \d \vol =  \int_\Omega (X, \nabla u) \d \vol$  implies that $u\in BV(\Omega)$ (respectively, $u\in BV_\loc(\Omega)$) with
    \begin{equation*}
      \TV{u}(A) = \int_A|\nabla u| \d \vol 
     \qand \Sigma^u = \nabla u(|\nabla u|)^{-1} \text{ on }\sppt{\nabla u}.
  \end{equation*} 
  In particular, for $u\in W^{1,1}(\Omega)$ one has $\norm[W^{1,1}(\Omega)]{u} = \norm[BV(\Omega)]{u}$.
\end{example}

\begin{remark}\label{rmk: properties of tv NEW}
    The principal advantage of our formulations of Definition~\ref{def: covector measure} and Corollary~\ref{cor: riesz markov for loc fin covector measures} is that the space $\calM_\loc(\Omega;T^*\Omega)$ is well-defined on arbitrary smooth manifolds and the Riesz--Markov theorem remains valid without a Riemannian metric\footnote{In other works on Riemannian $BV$ functions, generalizations of Euclidean vector-valued  measures are typically defined directly in terms of their  polar  decomposition, see~\cite{GunPal15_FunctionsBoundedVariation}.}.
    This is due to the fact that 
     boundedness of a vector field on relatively compact sets $U$ is well-defined by boundedness in the Euclidean sense for every finite cover of $U$ by local charts. The same is true for property~\eqref{eqn: locally bounded linear form on vector fields} and we define the action of $\nu$ as a continuous linear functional on $\Gamma_{C_c}(\Omega;T\Omega)$ via the dual pairing rather than the Riemannian metric.

    By contrast, the definitions of finite covector measures and of functions of bounded variation on noncompact open subsets of $\Omega$ do require a norm on tangent vectors. From this perspective, working with the  polar  decomposition $\nu=(|\nu|,\Sigma^\nu)$ provides a more natural framework, closely paralleling the classical  theory of vector-valued measures and  $BV$ functions on $\R^m$. Since several results developed later in this paper are inherently Riemannian, from now on we shall employ the  polar  decomposition with respect to the fixed metric $g$.
\end{remark}

We conclude this section by defining some operations on  covector measures.  We denote by $B_b(\Omega)$ the space of bounded Borel measurable functions on $\Omega$. Let $\nu$ be a (locally) finite covector measure on $\Omega$. Then its multiplication with $f\in B_b(\Omega)$ produces a  (locally) finite covector measure $f\nu$ by setting 
\begin{equation*}
    (f\nu)[X]:= \nu[fX] = \int_\Omega (X,\Sigma^\nu) f \d|\nu|
\end{equation*}
for  $|\nu|$-integrable vector fields $X$. For a section in the endomorphism bundle $\calA \in\Gamma_{B_b}(\Omega,\End(T\Omega))$, we define the (locally) finite covector measure $\calA\nu$ by 
\begin{equation*}
    (\calA \nu)[X] := \nu[\calA X] = \int_\Omega (X,\calA ^*\Sigma^\nu)\d |\nu|.
\end{equation*}
The \textit{restriction} of $\nu$ to a Borel measurable set $E$ is defined by
\begin{equation*}
     (\nu\llcorner E)[X]  :=
     % (\mathbf{1}_E\nu)[X] =
    \int_E X\cdot\d\nu.
\end{equation*}
For a fixed vector field $X\in \Gamma_{B_b}(\Omega;T\Omega)$, the \textit{contraction $\nu_X$ of $\nu$ with $X$} defines a signed  measure on $\Omega$ by
\begin{equation*}
    \nu_X(E):= \int_E X\cdot\d\nu
\end{equation*}
for Borel sets $E\subset\Omega$. 
    If $ f :M\to N$  is a diffeomorphism, then the \emph{pushforward}  of   $\calM_\loc(M; T^*M)$   by $ f $ is the locally finite covector measure $ f _*\nu$ on $N$ defined by 
    \begin{equation*}
         f _*\nu \, [Y] = \nu\,[( f \inv)_*Y]
    \end{equation*}
    for $Y\in \Gamma_{C_0^1}(N,TN)$.  
    Similarly, the \textit{pullback} of $\eta\in \calM_\loc(N;T^*N)$ under $f$ is the locally finite covector measure on $M$ defined by 
    \begin{equation*}
         f ^*\eta\, [X]:= \eta\,[ f _*X],
    \end{equation*}
    and we obtain $ f _*\nu\,[Y] = ( f \inv)^*\nu\, [Y]$.

%%%%%%%%%%%%%%%%%%%%%%%%%%%%%%%%%%%%%%%%%%%%%%%%%%%%%%%%%%%%%%%%%%%%%%%%%%%%%%%%%%%%%%%%
%%%%%%%%%%%%%%%%%%%%%%%%%% Dependence on g %%%%%%%%%%%%%%%%%%%%%%%%%%%%%%%%%%%%%%%%%%%%%
%%%%%%%%%%%%%%%%%%%%%%%%%%%%%%%%%%%%%%%%%%%%%%%%%%%%%%%%%%%%%%%%%%%%%%%%%%%%%%%%%%%%%%%%
\subsection{Dependence of \texorpdfstring{$\calM$}{\textit{M}} and \texorpdfstring{$BV$}{\textit{BV}} on the choice of Riemannian metric}
%%%%%%%%%%%%%%%%%%%%%%%%%%%%%%%%%%%%%%%%%%%%%%%%%%%%%%%%%%%%%%%%%%%%%%%%%%%%%%%%%%%%%%%%

We examine the dependence of the spaces of covector measures and functions of bounded variation  on the choice of  Riemannian metric. 
Since the  polar  decomposition of a covector measure $\nu$ depends on the Riemannian metric, and therefore the operator seminorms of $\nu$ do as well,   we shall indicate the dependence when necessary by writing $(|\nu|_g,\Sigma^\nu_g)$. As we have seen earlier, the distributional gradient $D_gu =Du$  of $\smash{ u\in L^1_\loc(\Omega) }$ depends on the Riemannian volume measure.  Hence, the total variation measure of $u$ depends on $g$ twofold, and we shall write $ |Du| = |Du|_g= |D_gu|_g$, unless two Riemannian metrics are involved. 
Let $h$ be another Riemannian metric on~$M$. Then  
\begin{equation}\label{eqn: calA endomorphism}
    \calA= \calA_{g,h} :=\hat{g}^{-1}\circ \hat{h}
    \qand \calA\inv= \calA_{h,g} =\hat{h}\inv\circ \hat{g}
\end{equation}

 are the unique endomorphisms of $TM$ satisfying
\begin{equation}\label{eqn: transformation endomorphism for riemannian metrics}
    (X,Y)_h = (\calA\, X ,Y)_g \qand (X, \calA\inv Y)_h = (X,Y)_g .
\end{equation}
Then $g\leq C h$ holds on $\Omega\subset M$  in the sense of bilinear forms for some constant $C>0$ if and only if $\norm[\infty,\Omega]{\calA\inv} := \sup_{x\in\Omega} |\calA\inv(x)|$ is finite, and in that case, the optimal constant $C$ equals $ \norm[\infty,\Omega]{\calA\inv}$.
 Then $  \sqrt{\det(\calA)}$
is the Jacobian of the identity map from $(M,g)$ to $(M,h)$ and satisfies
$
    \det(\calA) = {\det(h)}/{\det(g)}.
$
In view of the local formula for the Riemannian volume measure in coordinates one has~$\smash{ \vol_h = \sqrt{\det(\calA)}\, \vol_g }$.

\begin{lemma}\label{lem: equivalence of gvm spaces for equivalent norms}
The  polar  decompositions of a covector measure $\nu$ on $\Omega$ transform via
    \begin{equation}\label{eqn: radon nikodym densities of tv measures}
             \frac{d|\nu|_h}{d|\nu|_g} =   (\Sigma^\nu_h,\Sigma^\nu_g)_g = |\calA\inv \,\Sigma^\nu_g|_h = \sqrt{(\Sigma^\nu_g,\calA\inv\,\Sigma^\nu_g)_g}
    \end{equation}
    and 
    \begin{equation}\label{eqn: transformation of polar  vector fields w.r.t metrics}
        \Sigma^\nu_h = \frac{\calA\inv\, \Sigma^\nu_g}{|\calA\inv \,\Sigma^\nu_g|_h} .
    \end{equation}
    If for some $C>0$ one has  $g\leq C h$ on $\Omega$, then $\calM(\Omega;T^*\Omega,g) \subset \calM(\Omega;T^*\Omega,h)$ and for every $\nu\in \calM(\Omega;T^*\Omega,g)$ the total variation measures  satisfy $ |\nu|_h(E) \leq C|\nu|_g(E)$  for all Borel sets $E\subset\Omega$.
\end{lemma}
\begin{proof}
    If $g\leq Ch$ on $\Omega$, then    $(\Gamma_{C_0}(\Omega, T\Omega), \norm[g,\infty]{\cdot})$ is  embedded into $(\Gamma_{C_0}(\Omega, T\Omega), \norm[h,\infty]{\cdot})$ with embedding constant $C$. Using the  Riesz--Markov isomorphism, it follows that  $\calM(\Omega; T^*\Omega,g) \subset \calM(\Omega; T^*\Omega,h)$ and  every $\nu\in \calM(\Omega;T^*\Omega,g)$ admits a unique  polar  decomposition $(|\nu|_h,\Sigma^\nu_h)$ with respect to $h$.   For open subsets $U\subset\Omega$,  the  estimate  $|\nu|_h(U) \leq C|\nu|_g(U)$ follows directly from the characterization of the total variation on open subsets in terms of the dual space norm  and extends to arbitrary Borel sets by outer regularity of finite Radon measures.  
    Finally, Definition~\ref{def: covector measure} and \eqref{eqn: polar  vector field wrt g} immediately imply \eqref{eqn: radon nikodym densities of tv measures} and \eqref{eqn: transformation of polar  vector fields w.r.t metrics}.
\end{proof}
We summarize some consequences of the previous lemma: 

\begin{remark}\label{rmk: pushfoward of vector measures}
\begin{enumerate}[leftmargin = *,   labelindent = 0em,]
    \item If $h$ and $g$ are bi-Lipschitz equivalent on $\Omega$, meaning that there exists a constant  $C\geq1$ such that $\tfrac{1}{C} g\leq h \leq C g$ holds on $\Omega$, then the spaces  $L^1(\Omega)$, $BV(\Omega)$, and $\calM(\Omega;T^*\Omega)$ defined with respect to $g$ coincide, as sets, with the corresponding spaces defined with respect to  $h$. Moreover, the norms induced by  $h$ and $g$ on each of these spaces are equivalent.  In particular,  since  two Riemannian metrics on $M$ are always locally  bi-Lipschitz equivalent, the spaces $BV_\loc(\Omega)$ and $\calM_\loc(\Omega;T^*\Omega,g)$ are independent of the choice of metric, as we observed for $\calM_\loc$ in Remark \ref{rmk: properties of tv NEW}  by the very construction.  On each of these spaces, the family of seminorms given by the norms on the restrictions to relatively compact open subsets of $\Omega$ induces a topology  which is independent of the Riemannian metric.  
    \item For $u\in L^1_\loc(\Omega)$, one can deduce from the coordinate representation of $\ddiv_g$ and $\ddiv_h$  that
    \begin{equation*}
      D_h u[X] = \sqrt{\det (\calA)} D_g u [ X],
    \end{equation*}
    and with Lemma~\ref{lem: equivalence of gvm spaces for equivalent norms} we conclude that one has
    $BV(\Omega,g) = BV(\Omega,h)$ if and only if both $\norm[\infty,\Omega]{\calA}$ and $ \norm[\infty,\Omega]{\calA\inv}$ are finite. Then one gets 
    \begin{equation} \label{eqn: g to h transformation of unweighted variation}
        |D u|_h=  \sqrt{\det(\calA)(\Sigma^u_g,\calA\inv\Sigma^u_g)_g}\, |D u|_g.
    \end{equation}
    \item If $\nu$ is a covector measure on $M$ represented by $(\Sigma_g,|\nu|_g)$ with respect to $g$ and $f:M\to N$ is a diffeomorphism, then $ f _*\nu$ is represented by $({ f }_*(\Sigma_g), f _*|\nu|_g)$ with respect to $\tilde{g}:={( f \inv)}^*g$. 
    Moreover, if $u\in BV_\loc(M)$, then $\tilde{u}:=( f \inv)^*u=u\circ f \inv\in BV_\loc(N)$  and one has
    \begin{equation*}
         D_{\tilde{g}}\tilde{u}[Y] = D_g u[( f \inv)_*Y]=  f _*(D_g u)[Y]
    \end{equation*}
      and   $|\Var[\tilde{g}]{\tilde{u}}|_{\tilde{g}} =  f _*(|D_g{u}|_g)$.
\end{enumerate}
    
\end{remark}

%%%%%%%%%%%%%%%%%%%%%%%%%%%%%%%%%%%%%%%%%%%%%%%%%%%%%%%%%%%%%%%%%%%%%%%%%%%%%%%%%%%%%%%%
%%%%%%%%%%%%%%%%%%%%%%%%%% Dependence on g %%%%%%%%%%%%%%%%%%%%%%%%%%%%%%%%%%%%%%%%%%%%%
%%%%%%%%%%%%%%%%%%%%%%%%%%%%%%%%%%%%%%%%%%%%%%%%%%%%%%%%%%%%%%%%%%%%%%%%%%%%%%%%%%%%%%%%
\subsection{Convergence of covector measures and convergence in \texorpdfstring{$BV$}{\textit{BV}}}
\label{sec: convergence of covector measures and in BV}

Recall that a sequence of (signed) Radon measures $\mu_n$ on $\Omega$ is said to converge in the \emph{weak$^*$} or \emph{vague} sense of measures to $\mu$ if for all $\psi\in C_0(\Omega)$ one has
\begin{equation}\label{eqn: weak$^*$/weak convergence of radon measures}
    \int_\Omega \psi \d \mu_n \to  \int_\Omega \psi  \d \mu \text{ as }n\to \infty,
\end{equation}
and  we then write $\mu_n\weakstarto \mu$. Note that the space of test functions can be reduced to the dense subspace $C_c(\Omega)$. If \eqref{eqn: weak$^*$/weak convergence of radon measures} holds for all bounded, continuous functions $\psi\in C_b(\Omega)$, then  we say that $\mu_n$ converges to $\mu$  \emph{weakly}, in which case we write  $\mu_n\rightharpoonup \mu$. Weak and weak$^*$ limits are unique, and every bounded sequence of (signed) Radon measures admits a weakly$^*$ convergent subsequence. Similarly, we define \textit{weak$^*$ convergence of  covector measures} $\nu_n$  to $\nu \in \calM_\loc(\Omega;T^*\Omega)$ via the condition that
    \begin{equation}\label{eqn: vage/weak convergence of generalized covector measures}
    \nu_n[X]   \to \nu[X]  \text{ as }n\to \infty
\end{equation}
for all $C_0$ or equivalently $C_c$ vector fields $X$. For finite covector measures we say that \textit{$\nu_n$ converges to $\nu$ weakly} if \eqref{eqn: vage/weak convergence of generalized covector measures} holds for all $C_b$-vector fields $X$ and we use the same notation as above to indicate weak and weak$^*$ convergence. Finally, we say that a sequence of covector measures on $\Omega$  $(\nu_n)$ is \textit{bounded}  if  $\sup_n|\nu_n|(\Omega)<\infty$ and \textit{locally bounded} if it is bounded on all relatively compact subsets of $\Omega$.

\begin{lemma}\label{lem: properties of weak and weak$^*$ convergence}
 Let $\nu$  and  $\nu_n$, $n\in\N$,  be  finite covector measures on $\Omega$.    
 \begin{enumerate}[label = (\roman*)]
        \item Weak$^*$ and weak limits of $(\nu_n)$ are unique.
        \item If $(\nu_n)$ is bounded, then it admits a weakly$^*$ convergent subsequence. 
        \item Assume  that $\nu_n\weakstarto \nu$. Then one has $\liminf_n |\nu_n|(U)\geq |\nu|(U)$ for all open subsets $U\subset\Omega$. If in addition, $\limsup|\nu_n|(\Omega)\leq |\nu|(\Omega)$, then one has $|\nu_n| \rightharpoonup  |\nu|$.
    \end{enumerate}
\end{lemma}

\begin{proof} 
    (i) The identification of $\calM(\Omega,T^*\Omega)$ with $(\Gamma_{C_0}(\Omega;T\Omega))^*$ via the Riesz--Markov representation theorem for covector measures implies that sequential weak$^*$ convergence in  $(\Gamma_{C_0}(\Omega;T\Omega))^*$ is equivalent to weak$^*$ convergence of the corresponding sequence in $\calM(\Omega,T^*\Omega)$. In particular, weak$^*$ limits of finite covector measures are unique and since every weak limit is  a weak$^*$ limit as well, this implies uniqueness of weak limits.
        
         (ii) Since the space $\Gamma_{C_0}(\Omega;T\Omega)$ is a separable Banach space,  bounded subsets of its dual space are relatively sequentially weakly$^*$-compact by the Banach-Alaoglu theorem, and we conclude that every bounded sequence of finite covector measures has a weakly$^*$ convergent subsequence.
        
         (iii) Let $U\subset\Omega$ be open and fix $\varepsilon>0$. By \eqref{eqn: tv equals operator norm} and a density argument, we can find $X\in \Gamma_{C_c}(U, TU)$ with $|X|\leq 1$  such that $|\nu|(U)\leq \int_U (X,\Sigma^\nu)  \,d|\nu| +\varepsilon$. Hence by weak$^*$ convergence, 
        \begin{equation*}
        \begin{gathered}
            |\nu|(U)-\varepsilon 
            \leq \int_\Omega (X,\Sigma^\nu)  \,d|\nu| 
            = \lim_n \int_\Omega (X,\Sigma^\nu_n)  \,d|\nu_n| 
            \leq \liminf_n\int_\Omega |(X,\Sigma^\nu_n) | \,d|\nu_n| \\
            \leq \liminf_n\int_\Omega 1 \,d|\nu_n| 
            = \liminf_n |\nu_n|(U).
            \end{gathered}
        \end{equation*} 
        This proves the first assertion. Now the second assertion is an immediate application of a version of the Portmanteau theorem for finite Borel measures which states that $|\nu_n| $ converges weakly to $ |\nu|$ if and only if $\lim|\nu_n|(\Omega) =  |\nu|(\Omega)$ and $\liminf_n |\nu_n|(U)\geq |\nu|$ for open subsets $U\subset\Omega$, see \cite[Theorem 4.10]{Els18_MassUndIntegrationstheorie}. 
\end{proof}

\begin{remark}\label{rem: remark on convergence notions for covector measures and BV}
     By definition, a sequence $(\nu_n)$ of  covector measures is bounded if and only if the sequence of  total variation measures $(|\nu_n|)$  is bounded. Just like $(\nu_n)$, the associated measures $|\nu_n|$ admit a weakly$^*$ convergent subsequence with a weak$^*$ limit $\mu$  in the sense of \eqref{eqn: weak$^*$/weak convergence of radon measures}, but in general this does \emph{not} imply $\mu = |\nu|$.  The third part of Lemma~\ref{lem: properties of weak and weak$^*$ convergence} provides a sufficient condition for equality of $\mu$ and $|\nu|$.
\end{remark}

In analogy to the Euclidean case, we define the notion of strict convergence of functions of bounded variation:
\begin{definition}
    A sequence $u_n\in BV(\Omega)$ converges \textit{strictly} to $u\in BV(\Omega)$ if $u_n\to u$ in $L^1(\Omega)$ and $|D u_n|(\Omega)\to \TV{u}(\Omega)$. 
\end{definition}

Clearly, this notion of convergence is  weaker than norm convergence in $BV(\Omega)$. However, it turns out to be useful since, unlike the norm topology, it admits approximations by smooth functions, as we shall see in Proposition~\ref{prop: approximation by smooth functions}. We first state some connections between notions of convergence for $BV$ functions.

\begin{corollary}\label{cor: weak$^*$ and weak convergence of (total) variation measure}
Let  $u_n, u\in BV(\Omega)$.
\begin{enumerate}[label=(\roman*)]
    \item If $\norm[1,\Omega]{u_n-u}\rar0$ and $\sup_n \TV{u_n}(\Omega)<\infty$, then  $Du_n \weakstarto  Du$.
    \item  If $u_n\to u$ strictly, then $Du_n\rightharpoonup Du$ and $\TV{u_n} \rightharpoonup \TV{u}$.
\end{enumerate} 
\end{corollary}
 
\begin{proof}
    (i)  By part (ii) of Lemma~\ref{lem: properties of weak and weak$^*$ convergence} one can extract a weakly$^*$ convergent subsequence of $Du_n$ with weak$^*$ limit $\nu = (|\nu|,\Sigma^\nu)$. It remains to prove that $\nu = Du$. Indeed, using that $u_n\to  u$ in $L^1(\Omega)$,  one gets for arbitrary $X\in \Gamma_{C^1_c}(\Omega, T\Omega)$, 
    \begin{equation*}
       - \int_\Omega u\ddiv (X) \d \vol 
       = - \lim_n \int_\Omega u_n \ddiv (X) \d \vol 
       = \lim_n \int_\Omega (X,\Sigma^{u_n}) \d \TV{u_n} 
        = \int_\Omega (X,\Sigma^\nu) \d |\nu|,
    \end{equation*}
    hence by Theorem~\ref{thm: characterization of BV}, we get $(|\nu|,\Sigma^\nu) = (\TV{u},\Sigma^u)$. Since this equality is independent of the choice of subsequence, the whole sequence  $(Du_n)$ converges weakly$^*$ to $Du$.

    (ii)   If $(U,\varphi)$ is a local  chart for $M$, then an application of Remark \ref{rmk: pushfoward of vector measures} implies that $u_n\to u$ strictly in $BV(U)$ if and only if $\tilde{u}_n = u_n\circ\varphi\inv \to u\circ \varphi\inv = \tilde{u} $ in $BV(\varphi(U))$. From \cite[Proposition 3.15]{AmbFusPal00_FunctionsBoundedVariation} it follows that $D\tilde{u}_n\to D\tilde{u}$  weakly on the domain $\varphi(U)\subset\R^m$.  Therefore, the pushforward measures under the inverse chart $ \varphi\inv_* (D\tilde{u}_n) = \varphi_*(\varphi\inv_* (Du_n)) = Du_n$ converge weakly to $ \varphi\inv_* (D\tilde{u} )= Du$ on $U$. 
    Now let $\set{U_\alpha}_\alpha$ be a locally finite and countable atlas for  $\Omega$ and let $\set{\zeta_\alpha}_\alpha\subset C^\infty(M)$ be a   partition of unity subordinate to $\set{U_\alpha}_\alpha$. Then for all $X\in \Gamma_{C_b}(\Omega;T\Omega)$ we get
  \begin{equation}
      Du_n[X] = \sum_\alpha \int_{U_\alpha}\zeta_\alpha X\cdot \d Du_n \to \sum_\alpha \int_{U_\alpha}\zeta_\alpha X\cdot \d Du = Du[X].
  \end{equation}
  Now $Du_n$ converges to $ Du$ weakly and by  definition of strict convergence, we have $\TV{u_n}(\Omega) \to \TV{u}(\Omega)$. Thus, we are in the situation of Part 3 of Lemma~\ref{lem: properties of weak and weak$^*$ convergence}  and conclude that $\TV{u_n}\rightharpoonup \TV{u}$  as well.
\end{proof}

As we have seen earlier, the Sobolev space $W^{1,1}(\Omega)$ is a subspace of $BV(\Omega)$ and for $u\in W^{1,1}(\Omega)$ the $BV$- and Sobolev-norms coincide. Thus,   $W^{1,1}(\Omega)$ is closed in~$BV(\Omega)$. By the  Meyers--Serrin theorem, smooth Sobolev functions are dense in Sobolev spaces and  therefore, they cannot be dense in $BV(\Omega)$ with respect to the norm topology. However, it is possible to assert the following density result in terms of strict convergence. Note that  a strict approximation result for functions in $BV(M)$ on a complete manifold $M$ appears in \cite{MirPalParPre07_HeatSemigroupFunctions}, and a weaker approximation result in terms of an $L^1$-approximation of  $BV$ functions with weak$^*$ convergence of the total variation measures is shown in \cite[Theorem 2.4]{AmbGheMag15_BVFunctionsSets} in the sub-Riemannian setting.

\begin{proposition}[Strict approximation by smooth functions] 
\label{prop: approximation by smooth functions}
    \label{prop-item: Anzellotti-Giaquinta}   For all $u\in BV(\Omega)$ there exists a sequence of functions $u_n \in C^\infty(\Omega)\cap BV(\Omega)$ such that  $u_n$ converges to $u$ strictly. 
\end{proposition}

\begin{proof} 
    Let  $u\in BV(\Omega)$. By  \cite[Thm. 3.1]{AmbGheMag15_BVFunctionsSets} there exists a sequence of locally Lipschitz continuous functions $v_k$ satisfying
    \begin{equation*}
       \sup_k|Dv_k|(\Omega) =  \sup_k\int_\Omega |\nabla v_k|\d\vol<\infty
    \end{equation*}
    such that $(v_k)$ converges to $u$ strictly. By Rademacher's theorem $\mathrm{Lip}_\loc(\Omega)\subset W^{1,1}_\loc(\Omega)$, so we get $v_k\in W^{1,1}_\loc(\Omega)\cap BV(\Omega)\subset W^{1,1}(\Omega)$. Therefore we can use density of
     $C^{\infty}(\Omega)\cap W^{1,1}(\Omega)$ in  $W^{1,1}(\Omega)$ (cf. \cite[Lemma 3.1]{MulSal07_ScatteringTheoryLaplacian} or  more generally, \cite[Cor. 3.6] {GuiGunPal17_L^1ellipticRegularityWhole})  to approximate each $v_k$ by a sequence $(u_{k_n})$
      in $C^{\infty}(\Omega)\cap W^{1,1}(\Omega)$. Since convergence in $W^{1,1}(\Omega)$ implies strict convergence on $\Omega$, we can pick an appropriate diagonal  sequence $(u_n)\subset C^\infty(\Omega)\cap W^{1,1}(\Omega) = C^\infty(\Omega)\cap BV(\Omega) $ that converges to $u$ strictly on $\Omega$.
\end{proof}

As a consequence, we  formulate an equivalent characterization for functions of bounded variation: A function $u$ lies in $BV(\Omega)$  if and only if there exists a sequence $(u_n)\subset C^\infty(\Omega)\cap W^{1,1}(\Omega)$ such that $u_n\to  u$ in $L^1(\Omega)$ and $\limsup_n \int_\Omega |\nabla u_n| \d \vol<\infty$. In that case, $  \TV{u}(\Omega) = \lim_n \int_\Omega |\nabla u_n| \d \vol$. 
We  note moreover that if $u$ is compactly supported or if $\Omega = M$ and $M$ is complete, then the sequence can be chosen in $C_c^{\infty}(\Omega)$ (see \cite{GuiGunPal17_L^1ellipticRegularityWhole} and \cite{Aub76_EspacesSobolevVarietes}, respectively).

\begin{corollary}
% [Weak$^*$-compactness in $BV_\loc$]
\label{Cor: compactness in BV_loc}
    If $(u_n)\subset BV_\loc(\Omega)$ is a sequence such that $\sup_n\norm[BV(U)]{u_n}<\infty$ for all open $U\subset\subset \Omega$, then there exists $u\in BV_\loc(\Omega)$ such that along a subsequence, $u_n \to  u$ in $L^1_\loc(\Omega)$ and $Du_n\to  Du$  weakly$^*$.
\end{corollary}

\begin{proof}
    \textit{Step 1) For each $U\subset\subset \Omega$ one can find  a subsequence that converges to some $v$  in $L^1(U)$ and such that  $\Var{u_n}\weakstarto \Var{v} $ on $U$.}

    \textit{Proof of Step 1). } We fix such a set $U$ and let $U'\subset\subset\Omega$ be an  open set with smooth boundary such that $U\subset\subset U'$. Then $u_n\in BV(U')$ so by Proposition~\ref{prop-item: Anzellotti-Giaquinta},  we can find for each $n$ a function $v_n\in C^\infty(U')\cap BV(U')\subset W^{1,1}(U')$ such that 
    \begin{equation}\label{eqn: L^1 estimate for compactness}
            \norm[1,U']{u_n-v_n}<\tfrac{1}{n} \qand \big| \TV{u_n}(U')-\TV{v_n}(U')|\big|<\frac{1}{n}.
    \end{equation}
    Therefore, since $\norm[BV(U')]{u_n}$ is uniformly bounded, it follows that $ \norm[W^{1,1}(U')]{v_n} =\norm[BV(U')]{v_n}$ is uniformly bounded as well. Since the closure of $U'$ is compact with smooth boundary, the embedding $W^{1,1}(U')\hookrightarrow L^1(U')$ is compact by \cite[Theorem 10.1]{Heb00_NonlinearAnalysisManifolds}  and there exists $v\in L^1(U')$ such that $v_n\to v$ in $L^1(U')$ along a subsequence. Then in view of  \eqref{eqn: L^1 estimate for compactness}, we have $u_n\to v$ in $L^1(U')$ along the same subsequence and by Corollary~\ref{cor: weak$^*$ and weak convergence of (total) variation measure}, $v$ lies in $BV(U')$ and $\Var{u_n}\weakstarto \Var{v}$  on $U'$ and the same convergences are valid on~$U$.
\smallskip

\textit{Step 2) The compactness statement holds  for $BV_\loc(\Omega)$.}

\textit{Proof of Step 2). }  Let $(U_k)_{k\geq1}$ be an exhaustion of $\Omega$ by compactly contained open sets.  Apply Step 1 iteratively to  $U_k$ for $k\geq 1$, to obtain, for each $k$, a subsequence $(u_{k_n})_n$ converging to some $v_k$ in $L^1(U_k)$ and with $\smash{ Du_{k_n}\to Dv_k }$ weakly$^*$ on $U_k$. Since the sets $U_k$ are nested, we have $\smash{ {v_{k+1}}_{|U_k}=v_k }$. Now  for arbitrary measurable $A\subset\subset\Omega$, we can find some $k$ such that $A\subset U_k$, hence the function  $u\in L^1_\loc(\Omega)$ is well-defined by setting $u_{|A}:={v_k}_{|A}$. Finally, we extract a diagonal sequence of $(u_{k_n})$ that converges to  $u\in L^1_\loc(\Omega)$ and satisfies   $Du_{k_n}\to  Du$  weakly$^*$ on~$\Omega$.
\end{proof}

%%%%%%%%%%%%%%%%%%%%%%%%%%%%%%%%%%%%%%%%%%%%%%%%%%%%%%%%%%%%%%%%%%%%%%%%%%%%%%%%%%%%%%%%
%%%%%%%%%%%%%%%%%%%%%%% Sets of finite perimeter %%%%%%%%%%%%%%%%%%%%%%%%%%%%%%%%%%%%%%%
%%%%%%%%%%%%%%%%%%%%%%%%%%%%%%%%%%%%%%%%%%%%%%%%%%%%%%%%%%%%%%%%%%%%%%%%%%%%%%%%%%%%%%%%
\subsection{Sets of finite perimeter}\label{appendix sec: sets of finite perimeter}
%%%%%%%%%%%%%%%%%%%%%%%%%%%%%%%%%%%%%%%%%%%%%%%%%%%%%%%%%%%%%%%%%%%%%%%%%%%%%%%%%%%%%%%%

Let $E$ be a measurable subset of $M$ and let $\Omega\subset M$ be open. The \textit{perimeter of $E$ in $\Omega$} is defined as
\begin{equation*}
    \Per(E; \Omega):= |D{\mathbf{1}_E}|(\Omega),
\end{equation*}
and $E$ is called a set of \emph{(locally) finite perimeter in} $\Omega$ if $\mathbf{1}_E$ has (locally) bounded variation in $\Omega$.  If we say $E$ has (locally) finite perimeter, we always mean that $E$ has (locally) finite perimeter in $M$ and refer to $P(E):=P(E;M)$ as the perimeter of $E$. Sets of finite perimeter are also referred to as \textit{Caccioppoli sets} and we will thus write  
$$\scrC(\Omega) = \set{E\subset\Omega:\, |E|<\infty \text{ and } \Per(E;\Omega)<\infty}$$ 
and define $\scrC_\loc(\Omega)$ as the measurable subsets of  locally finite perimeter. We say that a family $\set{E_n}$ in $\scrC(\Omega)$ is uniformly bounded if $\sup_n |E_n|+\Per(E_n;\Omega)<\infty$. 
If $E\subset M$ has finite perimeter in $\Omega$, then the perimeter of $E$ in $F$ is well-defined and finite for all measurable sets $F\subset \Omega$. A sequence of finite perimeter sets $E_n\subset M$ is said to converge \emph{strictly in} $\Omega$ to $E$ if the indicator functions $\mathbf{1}_{E_n}$ converge strictly to $\mathbf{1}_E$ in $BV(\Omega)$. Recall also that a sequence of measurable sets $E_n\subset M$ converges to $E\subset M$ in  measure if  their indicator functions converge in $L^1(M)$ with respect to the volume measure or equivalently, if the symmetric differences of $E_n$ and $E$ satisfy $\vol(E_n \Delta E)\to0 $. Then, \textit{local convergence in measure} is defined in the obvious way. 

\begin{remark}\label{rmk: compactness of Caccioppoli sets}
    The space of Borel measurable subsets of $\Omega$ is closed under (local) convergence in measure. This follows from the fact that any $L^1$-convergent sequence of indicator functions admits a subsequence that (locally) converges pointwise almost everywhere. Since the pointwise limit must also have values in the closed set $\{0,1\}$, it must be an indicator function, too. In particular, this implies that the compactness property from \ref{Cor: compactness in BV_loc}  can be applied to $\scrC_\loc(\Omega)$ in a similar manner.
\end{remark}

We recall a central theorem of the theory of $BV$ functions (cf. the proof of~\cite[Theorem 4.2]{Mir03_FunctionsBoundedVariation}):
\begin{theorem}[Coarea formula] \label{thm: coarea formula}
    Let $u\in BV(\Omega)$. Then
    \begin{equation*}
        \begin{aligned}
        \TV{u}(B) 
        &= \int_{\R} |D{\mathbf{1}_{\set{u>t}}}|(B) \d t \text{ for all measurable } B\subset \Omega,
        \end{aligned}
    \end{equation*} 
    where $\mathbf{1}_{\set{u>t}}$ is the indicator function of the superlevel  set  $E^t:= \set{x\in \Omega:\ u(x)>t}$ for $t\in\R$. Moreover, one has
    \begin{equation*}
       \TV{u}(B) = \int_{0}^{\infty} |D{\mathbf{1}_{\set{|u|>t}}}|(B) \d t 
        \text{ for all measurable } B\subset \Omega.
    \end{equation*}
\end{theorem}
Note that the  second equality follows from the first one, using either the identity $|D|u||=\TV{u}$, or   using that for $t>0$, the measure  $|D\mathbf{1}_{\set{|u|>t}}|$ is concentrated on the disjoint sets $\set{u = t}$ and $\set{u = -t}$, and decomposes as desired on these sets.

%%%%%%%%%%%%%%%%%%%%%%%%%%%%%%%%%%%%%%%%%%%%%%%%%%%%%%%%%%%%
\section{Differentiation theorems for Radon and covector measures}
\label{sec: differentiation theorems}
%%%%%%%%%%%%%%%%%%%%%%%%%%%%%%%%%%%%%%%%%%%%%%%%%%%%%%%%%%%%
We devote this section to the proof of  a differentiation theorem for arbitrary covector measures   with respect to Radon measures on $M$. Recall that the classical Lebesgue differentiation theorem states that, given  $f\in L^1_\loc(\R^m)$, with respect to the Lebesgue measure $\lambda$ one has 
    \begin{equation*}
        f(x) = \lim_{r\to0} \frac{1}{ \lambda(B_e(x,r))} \int_{B_e(x,r)} f\d\lambda \tforae x\in \R^m.
    \end{equation*}
     The more general Lebesgue--Besicovitch theorem states that, given a positive Radon measure $\mu$ and a signed or  vector-valued Radon measure $\nu$ on $\R^m$,  the limit
    \begin{equation}\label{eqn: classical vector valued besicovitch limit}
        F(x) = \lim_{r\downarrow0} \frac{\nu(\bar{B}_e(x,r))}{\mu(\bar{B}_e(x,r))}
    \end{equation}
    exists for $\mu$-almost every $x\in \R^m$ and one has $\nu = F\mu + \nu^s$, where $\nu^s$ is the singular part of $\nu$ with respect to $\mu$~\cite[Theorem 2.22]{AmbFusPal00_FunctionsBoundedVariation}. 
    
    The generalization of this statement to covector measures on Riemannian manifolds presents itself with two major difficulties. Firstly,  the expression in \eqref{eqn: classical vector valued besicovitch limit} is not well-defined for a Radon measure $\mu$ on $M$ and $\nu\in \calM_\loc (M;T^*M)$. Instead, we will consider the weak$^*$-limit of the sequence given by
    \begin{equation*}
        \frac{\nu\llcorner \bar{B}(x,r)}{\mu(\bar{B}(x,r))} \in \calM(M;T^*M),
    \end{equation*}
    see Theorem~\ref{thm: Besicovitch Theorem for vector measures}.
    Secondly, the proof of the Euclidean statement relies on the Besicovitch covering lemma,  which establishes a uniform bound on the overlap among countable subcollections of any family of closed balls. While this covering lemma   is generally  false on arbitrary Riemannian manifolds,  Federer actually  proves a much more general differentiation theorem on  metric spaces  for every locally Borel regular measure  that  admits a so-called Vitali relation. 
    We review some constructions and results from \cite{Fed96_GeometricMeasureTheory} in order to apply them to the Riemannian setting:
    
    If $\calX$ is a  metric space, a \textit{Vitali relation $\scrV$ for a measure $\mu$} on $\calX$ is a family of pairs
    \begin{equation*}
        \scrV \subset \set{(x,S):\, x\in S\tand S\subset \calX}
    \end{equation*}
   such that
   \begin{enumerate}[label=(\roman*)]
       \item for every $x\in \calX$, the family \textit{$\scrV$ is fine at $x$} in the sense that $\inf \set{\diam(S):\, (x,S)\in \scrV} = 0$,
       \item if $\tilde{\scrV}$ is a subcollection of $\scrV$ and $A\subset \calX$ is such that $\tilde{\scrV}$ is fine at each $x\in A$, 
       % $\inf \set{\diam(S):\, (x,S)\in \tilde{\scrV}} = 0$, 
       then there exists a countable subcollection of pairwise disjoint sets 
       $$\set{S_1,S_2,\dots} \subset \tilde{\scrV}(A)
       :=\set{S:\, (x,S)\in \tilde{\scrV}\tand x\in A}
       $$ 
       such that $\mu(A\setminus\textstyle{\bigcup_{k=1}^\infty} S_k) = 0.$
   \end{enumerate}
   For a given Vitali relation $\scrV$ and a  function $\Psi:\dom{\Psi } \to \overline{\R}$ defined on a subset $\dom{\Psi }$ of the power set of $\calX$,  the sequences defined by
    \begin{equation*}  
    \begin{aligned}
        \overline{a}_r &= \sup \set{\Psi (S):\, (x,S)\in \scrV,\, \diam(S)<r,\, S\in \dom{\Psi }} \in\overline{\R}\\
        \underline{a}_r &= \inf \set{\Psi (S):\, (x,S)\in \scrV,\, \diam(S)<r,\, S\in \dom{\Psi }} \in\overline{\R}
    \end{aligned}
    \end{equation*}
    are monotone increasing and decreasing respectively and admit  limits in $\overline{\R}$ for $r\downarrow 0 $. One writes
    \begin{equation}\label{eqn: liminf ans limsup}
        \limsup_{S\to x}\Psi (S) := \limsup_{r\downarrow 0}  \overline{a}_r = \lim_{r\downarrow0} \overline{a}_r \qand \liminf_{S\to x} \Psi (S) = \liminf_{r\downarrow 0  } \underline{a}_r =  \lim_{r\downarrow0} \underline{a}_r.
    \end{equation}
    If both limits coincide, then one defines
    \begin{equation}\label{eqn: limit with respect to Vitali relation}
        \lim_{S\to x}\Psi (S) :=   \limsup_{S\to x}\Psi (S) = \liminf_{S\to x} \Psi (S).
    \end{equation}
    We collect some remarks concerning these constructions on $M$.
    \begin{remark}\label{rmk: Vitali relations for manifolds}
    \begin{enumerate}[leftmargin = *,   labelindent = 0em,]
        \item If $\calY$ is another metric space and $\varphi:\calX\to \calY$ is a bi-Lipschitz map, then it is straightforward to show that
        $$
        \varphi(\scrV) := \set{(\varphi(x),\varphi(S)): (x,S)\in \scrV}
        $$
        is a Vitali relation for the pushforward measure $\varphi_*\mu$ on $\calY$.
        \item  We will use these constructions in the following setting: For every Radon measure on $M$, the family of closed Riemannian balls
        \begin{equation*}
            \scrV=\set{ (x, \bar{B}(x,r)):\, x\in M,\, r>0}
        \end{equation*}
        forms a Vitali relation. This is due to the fact that the Riemannian distance function is \textit{directionally limited} on relatively compact sets (see \cite[2.8.9]{Fed96_GeometricMeasureTheory} and the example thereafter, as well as \cite[Def. 2.5]{BueLeo16_RecoveringMeasuresApproximate} for a more accessible definition). Along with separability of $M$ and the fact that $M$ can be covered by countably many compact subsets, Theorem 2.8.18 in \cite{Fed96_GeometricMeasureTheory} implies that $\scrV$ satisfies the claimed property. Now suppose $\varphi$ is a local chart for $M$. Then the  limits in the sense of \eqref{eqn: limit with respect to Vitali relation}  with respect to $\varphi({\scrV})$  can be written equivalently as
        \begin{equation*}
        \begin{gathered}
            \liminf_{\varphi(S)\to \varphi(x)} \Psi (\varphi(S)) = \liminf_{r\downarrow 0 } \Psi (\varphi(\bar{B}(x,r)))
            ,\quad
            \limsup_{\varphi(S)\to \varphi(x)} \Psi (\varphi(S)) = \limsup_{r\downarrow 0 } \Psi (\varphi(\bar{B}(x,r)))\\
            \qand \lim_{\varphi(S)\to \varphi(x)} \Psi (\varphi(S)) = \lim_{r\downarrow 0 } \Psi (\varphi(\bar{B}(x,r))),
        \end{gathered}
        \end{equation*}
        and clearly the same holds for $\varphi=id_M$.
        Typically, $\Psi $ will be of the form $\frac{\nu}{\mu}$ where $\nu$ is a Radon measure or a covector measure, and $\mu$ is another Radon measure in $M$.
        \item   Let $\mu$ and $\nu$ be Radon measures  on $M$. Then by \cite[Theorem 2.9.2]{Fed96_GeometricMeasureTheory},
    $$\nu^\mu(A) := \inf\set{\nu(B):\, B \text{ is Borel measurable with } \mu(A\setminus B) = 0} \quad\tfor A\subset M$$
    defines a Radon measure that is absolutely continuous with respect to $\mu$. Moreover, there exists a Borel set $B_\mu$ with $\mu(M\setminus B_\mu) = 0$ such that
    $$\nu^\mu = \nu\llcorner B_\mu \tand  \nu = \nu^\mu + \nu^s \text{ with }\nu^s:= \nu\llcorner(M\setminus B_\mu).$$
    Then the Radon measures
    $\nu^s$  and $\mu$
    are mutually singular, that is, $\nu = \nu^\mu + \nu^s$
    is the Lebesgue decomposition of $\nu$ with respect to $\mu$.
    \end{enumerate}
    \end{remark}
    
Using our observations thus far, the following theorem is a summary of the results from \cite[2.9.5--2.9.10]{Fed96_GeometricMeasureTheory} applied to manifolds and the Vitali relation of closed metric balls for arbitrary Radon measures on $M$, and we stress that the statements are independent of the choice of Vitali relation.

\begin{theorem}[Lebesgue--Besicovitch--Federer Differentiation theorem]\label{thm: Lebesgue--Besicovitch--Federer}
    Let  $\mu$ and $\nu$ be  Radon measures on $M$. 
    \begin{enumerate}[label=(\roman*)]
        \item The limit 
    \begin{equation*}
        f(x) := \lim_{r\downarrow0} \frac{\nu(\bar{B}(x,r))}{\mu(\bar{B}(x,r))} \text{ exists in }[0,\infty) \text{ for } \mu\text{-a.e.} x\in M
    \end{equation*}
     and the function $f:M\to [0,\infty)$ is  $\mu$-measurable. Every $\mu$-measurable set $A\subset M$ is $\nu^\mu$-measurable and one has $\nu^\mu (A) = \int_A f\d\mu.$
    % \begin{equation*}
    %     \nu^\mu (A) = \int_A f\d\mu .
    % \end{equation*}
    \item\label{item: differentiation thm for functions}  If $f:M\to \bar{\R}$ is a locally $\mu$-integrable function, then for $\mu$-almost every $x\in M$,
    \begin{equation*}
        \lim_{r\downarrow 0} \frac{1}{\mu(\bar{B}(x,r)} \int_{\bar{B}(x,r))} f \d\mu = f(x).
    \end{equation*}
    \item For $\mu$-almost every $x\in M$ one has
    \begin{equation*}
         \lim_{r\downarrow0} \frac{\nu^\mu(\bar{B}(x,r))}{\mu(\bar{B}(x,r))} 
         =  \lim_{r\downarrow0} \frac{\nu(\bar{B}(x,r))}{\mu(\bar{B}(x,r))} 
         \qand  
         \lim_{r\downarrow0} \frac{\nu^s(\bar{B}(x,r))}{\mu(\bar{B}(x,r))}  
         =  0.
    \end{equation*}
    \end{enumerate}
\end{theorem}

As a corollary, we formulate a differentiation theorem for covector measures with respect to arbitrary Radon measures on $M$. We say that a covector measure $\nu \in \calM_\loc(M;T^*M)$   is \textit{absolutely continuous} with respect to a Radon measure $\mu$ on $M$ if its restriction  $\nu\llcorner E$ to a   Borel set $E\subset M$ with $\mu(E) = 0$
satisfies $\nu\llcorner E = 0$.  
It follows easily that one has $\nu\ll \mu$  if and only if the total variation measure satisfies $|\nu|\ll \mu$. The property of $\nu$ and $\mu$ being mutually singular is defined analogously and holds if and only if $|\nu|$ and $\mu$ are mutually singular. Thus, $\nu$ admits a (unique) Lebesgue decomposition  $\nu = \nu^\mu + \nu^s$  given by
\begin{equation*}
    \nu^\mu [X]:=\int_M (\Sigma^\nu, X) \d |\nu|^\mu 
        \qand
        \nu^s[X]:= \int_M (\Sigma^\nu, X) \d |\nu|^s \tfor X\in \Gamma_{C_c^\infty}(M;TM).
\end{equation*}

\begin{theorem}
% [Differentiation theorem for covector measures]
\label{thm: Besicovitch Theorem for vector measures}
    Let $\nu\in \calM_\loc(M;T^*M)$ be a  covector measure and $\mu$  a  Radon measure on $M$. Then for $\mu$-almost every $x \in M$   and every $X\in \Gamma_{C_c^\infty}(M;TM)$ there exists $f_X(x)\in \R$ such that 
    \begin{equation*}
      \lim_{r\downarrow 0}\frac{ \nu\llcorner \bar{B}(x,r)}{ \mu (\bar{B}(x,r))}[X] = f_X(x),
    \end{equation*}
    and $f_X:M\to \R$ defines a locally $\mu$-integrable function.
    Moreover, there exists a $\mu$-a.e. uniquely defined locally $\mu$-integrable section $F:M\to TM$ such that for all $X$, we have
    \begin{equation*}
        f_X(x) = (F(x),X(x)),
    \end{equation*}
    whenever the limit exists at $x$.
    Finally, it holds that
    \begin{equation*}
        \nu^\mu[X] = \int_M (F,X)\d\mu,
    \end{equation*}
    and if $\nu \in \calM(M;T^*M)$, then $f_X$ and $F$ are $\mu$-integrable.
\end{theorem}
\begin{proof}
   Let $(\Sigma^\nu, |\nu|)$ be the  polar  decomposition of $\nu$  and let $\nu = \nu^\mu + \nu^s$ be its Lebesgue decomposition   with respect to $\mu$.
   Fix $X\in\Gamma_{C_c^\infty}(M;TM)$.
    The signed measure $\nu^s_X$ defined by contracting $\nu^s$ with $X$
    admits a Jordan decomposition into its positive and negative parts $(\nu^s_X)^\pm$. Since $\nu^s$ and $\mu$ are mutually singular,  we have $(\nu^s_X)^\pm\perp \mu$ as well. Thus, from the third part of Theorem~\ref{thm: Lebesgue--Besicovitch--Federer} applied to  $(\nu^s_X)^\pm$ it follows that    
    \begin{equation}\label{eqn: singular vector measure has zero derivative}
       \lim_{r\downarrow 0} \frac{ (\nu^s\llcorner \bar{B}(x,r)) }{\mu(\bar{B}(x,r))}  [X]
       % = \lim_{r\downarrow 0} \frac{\nu^s_X (\bar{B}(x,r))}{\mu  (\bar{B}(x,r))}
       = \lim_{r\downarrow 0} \frac{(\nu^s_X)^+ (\bar{B}(x,r))}{\mu  (\bar{B}(x,r))} -  \lim_{r\downarrow 0} \frac{(\nu^s_X)^- (\bar{B}(x,r))}{\mu  (\bar{B}(x,r))}
       = 0
    \end{equation} 
    for $\mu$-almost every $x\in M$.
    Next, invoking the Besicovitch--Federer theorem for $|\nu|$ and $\mu$, the limit 
    \begin{equation*}
        f(x) =  \lim_{r\downarrow0} \frac{|\nu|(\bar{B}(x,r))}{\mu(\bar{B}(x,r))}
    \end{equation*}
    exists $\mu$-almost everywhere   with $|\nu|^\mu(A) = \int_A f\d\mu \tfor \mu$-measurable sets $A\subset M$. Since $|\nu|$ is locally finite, $f$ is locally $\mu$-integrable.
    By Borel measurability of $\Sigma^\nu$, the function $$f_X:=f\cdot(\Sigma^\nu, X)$$ 
    is  $\mu$-measurable. For $\mu$-measurable sets $E\subset M$, we use that $|\Sigma^\nu| = 1$ for $|\nu|$-almost every $x$ to estimate
    \begin{equation*}
        \int_E |f_X|\d\mu
        = \int_E |(\Sigma^\nu,X)|f\d\mu
        = \int_E |(\Sigma^\nu, X)|\d|\nu|^\mu
        \leq \int_E |X|\d|\nu|^\mu,
    \end{equation*}
    and since $|\nu|^\mu\leq |\nu|$ is locally finite and $X$ is bounded, the function $f_X$ is locally  $\mu$-integrable.
    Thus, combining the second part of Theorem~\ref{thm: Lebesgue--Besicovitch--Federer} with \eqref{eqn: singular vector measure has zero derivative}, we get
    \begin{align*}
      \lim_{r\downarrow 0 } \frac{ (\nu\llcorner \bar{B}(x,r)) }{\mu(\bar{B}(x,r))}  [X]
      &=\lim_{r\downarrow 0 }\frac{ (\nu^\mu\llcorner \bar{B}(x,r)) }{\mu(\bar{B}(x,r))} [X]
      +  \lim_{r\downarrow 0 }\frac{ (\nu^s\llcorner \bar{B}(x,r)) }{\mu(\bar{B}(x,r))} [X] \\
      &=\lim_{r\downarrow 0 }  \frac{1}{\mu(\bar{B}(x,r))}  \int_{\bar{B}(x,r)} f\cdot(\Sigma^\nu, X) \d\mu + 0\\
      &= f(x)(\Sigma^\nu(x), X(x))\\
      &= f_X(x) \tfor \mu\text{-almost every }x\in M.
      \end{align*}
      To complete the proof, it remains to show that  
      $F(x) := f(x)\Sigma^\nu(x)$  is uniquely determined for $\mu$-almost every $x$ and defines a locally $\mu$-integrable section.
       We know that $f(x)$ is unique $\mu$-almost everywhere and by definition of $\nu$, the vector field $\Sigma^\nu$ is unique $|\nu|$-almost everywhere on $M$.  Suppose $\hat{\Sigma}^\nu$ is another representative of $\Sigma^\nu$ and consider
       $$S:=\set{x\in M: \hat{\Sigma}^\nu(x)\neq \Sigma^\nu(x)}.$$
       Since $\hat{\Sigma}^\nu$ and $ \Sigma^\nu$ are Borel measurable,  the set $S$ is Borel measurable, and therefore $\mu$-measurable.
       % Assume $\mu(S)\neq 0$. 
       Using that $\hat{\Sigma}^\nu = \Sigma^\nu$ holds $|\nu|$-almost everywhere and $|\nu|^\mu\ll |\nu|$, we have 
      \begin{equation*}
          0 = |\nu|(S) = |\nu|^\mu(S) = \int_S f\d\mu.
      \end{equation*}
      So if $\mu(S)\neq 0$,  then $f\geq 0$ implies   $f\equiv0$ on $S$ and this shows that $F(x) \in T_xM$ is well-defined for $\mu$-almost every $x\in M$, and for relatively compact $E$ we have
      \begin{equation*}
          \int_E |F| \d\mu \leq \int_E f\cdot|\Sigma^\nu|\d\mu \leq |\nu|(E)<\infty.
      \end{equation*}
      If $\nu\in \calM (M;T^*M)$, then $\mu$-integrability of $f$, $f_X$ and $F$ follow from finiteness of $|\nu|$.
\end{proof}

Notice that the first part of the statement of Theorem~\ref{thm: Besicovitch Theorem for vector measures} can be formulated equivalently by saying that there exists a unique $\mu$-integrable section $F$ in $TM$ such that for $\mu$-almost every $x\in M$, the family of covector measures 
\begin{equation}\label{eqn: weak* convergence of differentiation limit}
    \frac{\nu\llcorner \bar{B}(x,r)}{\mu(\bar{B}(x,r))} \text{ converges in the weak}^*\text{-sense to }F\delta_x\text{ as }r\downarrow0,
\end{equation}
where the\textit{ covector Dirac measure} $F\delta_x $ is defined by
\begin{equation*}
    F\delta_x [X] = \int_M (F,X) \d\delta_x = (F(x),X(x)).
\end{equation*}
If $\nu\in \calM(M;T^*M)$, then weak$^*$-convergence in \eqref{eqn: weak* convergence of differentiation limit} is equivalent to weak convergence.
We stress that in Euclidean space, this is consistent with the Besicovitch theorem for classical vector-valued measures~\cite{AmbFusPal00_FunctionsBoundedVariation}. Every covector measure on $\R^m$ is uniquely determined by its  polar  decomposition $(|\nu|_e, \Sigma^\nu_e)$  with respect to the Euclidean scalar product $e$, and the same applies to classical vector-valued measures on~$\R^m$. We write
\begin{equation*}
    \nu[X]= \int_{\R^m} (\Sigma^\nu_e,X)_e \d|\nu|_e \tfor X\in C_c^\infty(\R^m;\R^m)
\end{equation*}
 when viewing $\nu$ as a covector measure and 
\begin{equation*}
    \nu(A)=\int_A \Sigma^\nu_e\d|\nu|_e \in \R^m
\end{equation*}
to regard $\nu$ as a vector-valued measure on $\R^m$. Due to these identifications, it is clear that the density vector fields \eqref{eqn: classical vector valued besicovitch limit} produced by the classical vector-valued Lebesgue--Besicovitch theorem  and the Lebesgue--Besicovitch--Federer theorem for covector measures (Theorem~\ref{thm: Besicovitch Theorem for vector measures}) with respect to a Radon measure $\mu$ must coincide $\mu$-almost everywhere. Later, we will be interested in pointwise properties of these functions in order to define a Riemannian version of the reduced boundary of a finite perimeter set, where the following lemma will be of use.
\begin{lemma} \label{lem: equivalence of differentiability points of covector and vector measures on Rm}
    Let $\mu$ be a Radon measure and $\nu$  a vector-valued measure on $\R^m$. We fix a point $x\in \R^m$.
    If  
    \begin{equation}\label{eqn: weak^* convergence of euclidean besicovitch limit}
        \lim_{r\downarrow 0 } \frac{\nu\llcorner {\bar{B}_e(x,r)}}{\mu({\bar{B}_e(x,r)})}[X]   \text{ exists in }\R \tforall X\in C_c^\infty(\R^m;\R^m),
    \end{equation}
      then
    \begin{equation}\label{eqn: euclidean besicovitch limit}
        \lim_{r\downarrow 0} \frac{\nu({\bar{B}_e(x,r)})}{\mu({\bar{B}_e(x,r)})}    \text{ exists in }\R^m.
    \end{equation}
     If the locally $\mu$-integrable vector field  $F:\R^m\to\R^m$ from Theorem~\ref{thm: Besicovitch Theorem for vector measures}  additionally satisfies\footnote{Note that  \eqref{eqn: bounded absolute average} is satisfied for all $x\in \R^m$ if $F$ is locally essentially bounded with respect to $\mu$.}
    \begin{equation}\label{eqn: bounded absolute average}
    \limsup_{r\downarrow 0}\fint_{\bar{B}_e(x,r)} |F|_e \d\mu< \infty,
    \end{equation}
    then \eqref{eqn: weak^* convergence of euclidean besicovitch limit} and \eqref{eqn: euclidean besicovitch limit} are equivalent. 
     Here we write  $\fint_B\d\mu$ for the $\mu$-average over $B$.
\end{lemma}

\begin{proof}
    To prove that \eqref{eqn: weak^* convergence of euclidean besicovitch limit} implies \eqref{eqn: euclidean besicovitch limit}, simply pick the vector field $X$ to be  constant  on some  small neighborhood of $x$. Then \eqref{eqn: weak^* convergence of euclidean besicovitch limit} applied to all such $X$ implies that
    \begin{equation*}
        \frac{\nu (\bar{B}_e(x,r))}{\mu (\bar{B}_e(x,r))} \rightharpoonup F(x) \text{ in } \R^m.
    \end{equation*}
    Since this is equivalent to strong convergence in $\R^m$, we obtain \eqref{eqn: euclidean besicovitch limit}. For the reverse implication, suppose $x$ is such that \eqref{eqn: euclidean besicovitch limit} and \eqref{eqn: bounded absolute average} hold  and let $X\in C_c^\infty(\R^m  ; \R^m)$ be arbitrary. First notice that 
    \begin{equation}\label{eqn: mu-lebesgue points of smooth vf}
        % \lim_{B\to x} \frac{1}{\mu(B)} \int_B X\d\mu- X(x) \leq 
        \lim_{r\downarrow 0}   \fint_{\bar{B}_e(x,r)} |X(y)-X(x)|_e \d\mu(y) \leq \lim_{r\downarrow 0}\sup_{y\in {\bar{B}_e(x,r)}} |X(y)-X(x)|_e  =0
    \end{equation}
    holds for arbitrary $x$ since $X$ is smooth. 
    For $\nu^\mu=F\mu$, we get
    \begin{align}
        &\frac{\nu^\mu \llcorner {\bar{B}_e(x,r)}}{\mu ({\bar{B}_e(x,r)})}[X] - (F(x), X(x))_e  \nonumber\\
        &\quad = \fint_{{\bar{B}_e(x,r)}} (F, X)_e \d\mu     - (F(x), X(x))_e  
         + \fint_{{\bar{B}_e(x,r)}}    (F(y), X(x))_e - (F(y), X(x))_e   \d\mu (y) \nonumber \\
        &\quad=  \fint_{{\bar{B}_e(x,r)}} (F(y), X (y) - X(x))_e \d\mu (y) 
          + \fint_{{\bar{B}_e(x,r)}}  (F(y) , X(x))_e \d\mu (y) -(F(x), X(x))_e \nonumber\\
        &\quad \leq \fint_{{\bar{B}_e(x,r)}} |F|_e\d\mu \,\fint_{{\bar{B}_e(x,r)}} |X (y) - X(x)|_e \d\mu (y) \label{eqn: first term euclidean besicovitch} \\
        &\quad\quad\quad+ \big(\textstyle{\fint_{\bar{B}_e(x,r)}}  F \d\mu \,,  X(x)\big)_e - \big(F(x),X(x)\big)_e .\label{eqn: second term euclidean besicovitch}
    \end{align}
The term  in \eqref{eqn: first term euclidean besicovitch} tends to zero as $r\downarrow 0$ in view of \eqref{eqn: mu-lebesgue points of smooth vf} and boundedness of  $\fint_{{\bar{B}(x,r)}} |F|_e\d\mu$ for sufficiently small $r$.  The last term \eqref{eqn: second term euclidean besicovitch} equals
\begin{equation*}
    \Big(\textstyle{\frac{\nu^\mu ({\bar{B}_e(x,r)})}{\mu ({\bar{B}_e(x,r)})}}  - F(x) \,,  X(x)\Big)_e 
\end{equation*}
and  tends to zero since $\frac{\nu^\mu ({\bar{B}_e(x,r)})}{\mu ({\bar{B}_e(x,r)})}$ converges to $F(x)$ strongly in $\R^m$. We conclude that
\begin{equation*}
   \lim_{r\downarrow 0} \frac{\nu \llcorner {\bar{B}_e(x,r)}}{\mu ({\bar{B}_e(x,r)})}[X] = \lim_{r\downarrow 0} \frac{\nu^\mu \llcorner {\bar{B}_e(x,r)}}{\mu ({\bar{B}_e(x,r)})}[X] = (F(x), X(x))_e  .
\end{equation*}
\end{proof}

%%%%%%%%%%%%%%%%%%%%%%%%%%%%%%%%%%%%%%%%%%%%%%%%%%%%%%%%%%%%
\section{Structure theory for sets of finite perimeter}\label{sec: structure theory}

Central results in the theory of finite perimeter sets in $\R^m$ are the structure theorems by De~Giorgi and Federer~\cite[3.59 and 3.61]{AmbFusPal00_FunctionsBoundedVariation}.
In order to prove their Riemannian versions, we first provide the required geometric  definitions and prove the analogs of some structure results from the theory of rectifiable sets in $\R^m$. For details on the Euclidean case we refer to \cite[Chapter 3]{Sim18_IntroductionGeometricMeasure}, \cite[Chapter 2.9 - 2.12]{AmbFusPal00_FunctionsBoundedVariation} and \cite[Chapter 10]{Mag12_SetsFinitePerimeter}.

\subsection{The reduced boundary}
\label{sec: reduced boundary and De~Giorgi}
%%%%%%%%%%%%%%%%%%%%%%%%%%%%%%%%%%%%%%%%%%%%%%%%%%%%%%%%%%%%

Recall that the reduced boundary $\partial^*E$ of a set of locally finite perimeter   in $\R^m$ is the set of $x\in \sppt |D\mathbf{1}_E|_e$ such that the Euclidean measure-theoretic inner normal vector
\begin{equation*}
    \normal[e]{E} (x):=  \lim_{r\downarrow0} \frac{D\mathbf{1}_{E} (\bar{B}_e(x,r))}{|D\mathbf{1}_E|_e(\bar{B}_e(x,r))} 
\end{equation*}
exists and lies in unit sphere in $\R^m$. 
With the Lebesgue--Besicovitch--Federer theorem for covector measures at hand, we can extend this definition to Riemannian manifolds:
\begin{definition}
    The \emph{reduced boundary} $\partial^*E$ of a set of locally finite perimeter $E\subset M$ is the set of points  $x\in M$ whose \emph{measure-theoretic inner normal vector} $\normal{E}(x)$ defined by
    \begin{equation*}
    \lim_{r\downarrow0} \frac{D\mathbf{1}_{E}\llcorner (\bar{B}(x,r))}{|D\mathbf{1}_E|(\bar{B}(x,r))} [X] = (\normal{E}(x),X(x)) \tforall X\in \Gamma_{C_c^\infty}(M;TM)
    \end{equation*}
exists and has unit norm. 
\end{definition} 
Thus, if $\Sigma^E$ is the  polar  vector field of $\Var{\mathbf{1}_E}$, then for $x\in\partial^*E$ we get $\Sigma^E(x) = \normal{E}(x)$. Since $\Sigma^E$ is bounded by construction,   Lemma~\ref{lem: equivalence of differentiability points of covector and vector measures on Rm}  shows that in the Euclidean setting, our definition of the reduced boundary is equivalent to the classical one above.
In order to apply results from the Euclidean theory of finite perimeter sets, we show that the Lebesgue--Besicovitch--Federer theorem for covector measures can be expressed chartwise: Let $\mu$ be a Radon measure on $M$ and $\nu$ a covector measure on $M$. If $(U,\varphi)$ is a local chart for $M$, then $\varphi_*\mu$ is a Radon measure on $\varphi(U)\subset\R^m$ and we view it as a measure on the full space which is concentrated on $\varphi(U)$ by setting
\begin{equation*}
    \varphi_*\mu (E) := \mu(\varphi\inv(E\cap \varphi(U))) \tfor \text{Borel sets } E\subset \R^m.
\end{equation*}
 Similarly, $\varphi_*\nu$ is a covector measure  on $\R^m$ and therefore a classical vector-valued measure with
\begin{equation*}
    \varphi_*\nu(E) = \int_E \Sigma^{\varphi_*\nu}_e \d |\varphi_*\nu|_e \in \R^m \tfor \text{Borel sets }E\subset \R^m ,
\end{equation*}
where $(\Sigma^{\varphi_*\nu}_e,|\varphi_*\nu|_e)$ is the  polar  decomposition of $\varphi_*\nu$ with respect to the Euclidean metric $e$.
As before we denote the $m$-dimensional Euclidean closed ball of radius $r$ around $w$ by $\bar{B}_e(w,r)$. 

\begin{lemma} \label{lem: differentiable points in charts}
Let $\mu$ be a Radon measure and $\nu$  a covector measure on $M$.
Then  for all $ X\in \Gamma_{C_c^\infty}(M;TM)$ the limit
\begin{equation}\label{eqn: besicovitch limit on M with metric balls}
    \lim_{r\downarrow0} \frac{\nu\llcorner (\bar{B}(x,r))}{\mu(\bar{B}(x,r))} [X] =: (F(x),X(x)) \text{ exists at }x\in M
\end{equation}
if and only if  for each chart $(U,\varphi)$ around $x$  and all $Y\in C^\infty_c(\R^m;\R^m)$ the limit
\begin{equation}\label{eqn: weak^* convergence of  pushforward besicovitch limit in Rm}
    \lim_{r\downarrow 0} \frac{(\varphi_* \nu)(\bar{B}_e(w,r))}{(\varphi_* \mu)(\bar{B}_e(w,r))}[Y] =: (F_\varphi(w), Y(w))_e \text{ exists at } w:= \varphi(x).
\end{equation}
In that case, one has
\begin{equation}\label{eqn: pointwise vector-valued Radon-Nikodym density in charts}
    F_\varphi(w) =( \calA\, \varphi_*)( F(x)),
\end{equation}
where   $\calA = \calA_{e,\tilde{g}}$ is as in \eqref{eqn: calA endomorphism} with $\tilde{g}= (\varphi\inv)^*g$.
\end{lemma}

\begin{proof} We divide the proof into three steps.

    \textit{Step 1)} \textit{For $\varphi_* \mu$-almost every $w\in \varphi(U)\subset\R^m$, it holds that
        $F_\varphi(w) =( \calA\varphi_* F)(w)$.}
    
    \textit{Proof of Step 1).} First notice that $\nu\ll\mu$ holds if and only if $\varphi_*\nu\ll\varphi_*\mu$, hence we get
    \begin{equation}\label{eqn: pushforward of absolutely continuous part}
        \varphi_*(\nu^\mu) = (\varphi_*\nu)^{\varphi_*\mu}.
    \end{equation}
    For $Y\in C_c^\infty(\varphi(U),\R^m)$ we have
    \begin{equation*}
    \begin{gathered}
        (\varphi_*(\nu^\mu))[Y] 
        = \nu^\mu[\varphi\inv_* Y]  
        = \int_{\varphi(U)} (F, (\varphi\inv)_* Y)_g         \circ\varphi\inv \d(\varphi_*\mu)  
        = \int_{\varphi(U)} (\varphi_* F,   Y)_{\tilde{g}}   \d(\varphi_*\mu)   \\
        =\int_{\varphi(U)} (\calA (\varphi_* F),  Y)_e   \d(\varphi_*\mu)   .
    \end{gathered}
    \end{equation*} 
    On the other hand, using \eqref{eqn: pushforward of absolutely continuous part} and the definition of $F_\varphi$, this is equal to
    \begin{equation*}
       (\varphi_*\nu)^{\varphi_*\mu}[Y]= \int_{\varphi(U)} (F_\varphi,Y)_e \d(\varphi_*\mu) \tforall Y \in C_c^\infty(\varphi(U),\R^m).
    \end{equation*}
     By a version of the fundamental lemma of the calculus of variations, this implies that $F_\varphi$ coincides $\mu$-almost everywhere with $\calA (\varphi_* F)$.

   \textit{Step 2)} \textit{ If for $w = \varphi(x)$, the limit \eqref{eqn: weak^* convergence of  pushforward besicovitch limit in Rm} exists for all $Y\in C_c^\infty(\R^m; \R^m)$, then the limit \eqref{eqn: besicovitch limit on M with metric balls} exists at $x$ for all $X\in \Gamma_{C_c^\infty}(M;TM)$.}

    \textit{Proof of Step 2). } Let $X\in \Gamma_{C_c^\infty}(U;TU)$ and write $Y:= \varphi_*X \in C_c^\infty(\R^m,\R^m)$. Then
    \begin{equation} 
    \begin{aligned}
         \lim_{r\downarrow 0} \frac{\nu\llcorner \big(\varphi\inv(\bar{B}_e(w,r))\big)}{\mu\big(\varphi\inv(\bar{B}_e(w,r))\big)}[X] 
        &= \lim_{r\downarrow 0} {\left((\varphi_*\mu)(\bar{B}_e(w,r))\right)\inv} \int_{\varphi\inv(\bar{B}_e(w,r))} (F, X)_g          \d\mu \\
        & =\lim_{r\downarrow 0} {\left((\varphi_*\mu)(\bar{B}_e(w,r))\right)\inv} \int_{\bar{B}_e(w,r)} (F, (\varphi\inv)_* Y)_g         \circ\varphi\inv \d(\varphi_*\mu)   \\
        &= \lim_{r\downarrow 0} \frac{(\varphi_* \nu)(\bar{B}_e(w,r))}{(\varphi_* \mu)(\bar{B}_e(w,r))}\\
         &= (F_\varphi(w),Y(w))_e  \in \R.
    \end{aligned}
    \end{equation}
    By Remark \ref{rmk: Vitali relations for manifolds}, the images of closed Euclidean balls under $\varphi\inv$ form a Vitali relation on $U$. Since the limit is independent of the choice of Vitali relation, we must have
    \begin{equation*}
        (F_\varphi(w), Y(w))_e
        =\lim_{r\downarrow 0} \frac{\nu\llcorner \big(\varphi\inv(\bar{B}_e(w,r))\big)}{\mu\big(\varphi\inv(\bar{B}_e(w,r))\big)}[X] 
        = \lim_{r\downarrow 0} \frac{\nu\llcorner(\bar{B}(x,r))}{\mu(\bar{B}(x,r))}[X] = (F(x),X(x))_g,
    \end{equation*}
    and in particular, $F(x)$ is well-defined. Since $X$ was arbitrary and $\varphi_*$ is a local isomorphism of the tangent bundles, the assertion is valid for arbitrary  $Y\in C_c^\infty(\varphi(U);\R^m)$. This is sufficient to deduce the statement for smooth, compactly supported sections on the full spaces.

    \textit{Step 3)} The converse of the previous claim is valid by analogous arguments. Therefore,  the relation \eqref{eqn: pointwise vector-valued Radon-Nikodym density in charts} is true pointwise for every such $w =\varphi(x)$   and a simple calculation shows that the statement is independent of the choice of chart.
\end{proof} 
As a corollary of this lemma, we obtain  a chartwise characterization of the reduced boundary. If $(U,\varphi)$ is a chart for $M$ and $E\subset M$, we will sometimes use the notation
\begin{equation*}
    \varphi(E):=\varphi(E\cap U).
\end{equation*}

\begin{corollary}\label{cor: chart wise reduced boundary}
    If $E\subset M$ is a set of finite perimeter, then $x\in \partial^*E$ if and only if $\varphi(x)\in \partial^*\varphi(E)\cap\varphi(U)$ for an arbitrary chart $(U,\varphi)$ around $x$. Moreover, $\normal[]{E}(x)$ is a measure-theoretic inner normal vector of $E$ at $x$ if and only if 
    $$\normal[e]{\varphi(E)}(\varphi(x)) = \frac{(\calA\,\varphi_*) (\normal[]{E}(x))}{|(\calA \,\varphi_* )(\normal[]{E}(x))|_e}$$ 
    is a Euclidean measure-theoretic inner normal vector of $\varphi(E)$ at $\varphi(x)$, with  $\calA$ as in Lemma~\ref{lem: differentiable points in charts}.
\end{corollary}

\subsection{Rectifiability and De~Giorgi's structure theorem}

Let $ \calH^k_g$ be the $k$-dimensional Hausdorff measure  on $M$ induced by the Riemannian distance function. That is, for $E\subset M$,
\begin{equation*}
    \calH^k_g(E)= \lim_{r  \downarrow 0} \inf 
    \Bigl\{ \omega_k \sum_j \left(\tfrac{1}{2}\diam_g(B_j)\right)^k :\, B_j\subset M,\, E\subset\bigcup_j B_j,\, \diam_g{B_j}<r \Bigr\}
\end{equation*}
  with $\omega_k$ denoting the $k$-dimensional Lebesgue measure of the Euclidean $k$-dimensional unit ball.
One can show that  $\calH^k_g$ is Borel regular on $M$. If $h$ is another Riemannian metric on $M$, then for Lipschitz functions $f:(M,g)\to (M,h)$ one has 
    \begin{equation}\label{eqn: Lipschitz estimate for Hausdorff measures}
            \calH^k_h(f(E)) \leq \mathrm{Lip}(f)^k \calH^k_g(E)
    \end{equation}
for arbitrary sets $E\subset M$. In particular, the identity map from $(M,g)$ to $ (M,h)$ is locally bi-Lipschitz continuous, and the families of Borel measurable sets of measure zero coincide  for $\calH^k_g$ and $\calH^k_h$. Since both measures are Borel regular, this implies  by Carath\'eorory's theorem that  the families of measurable sets coincide with the completion of the Borel-$\sigma$-algebra on $M$ and are independent of the choice of metric, and we simply say a set is $\calH^{k}$-measurable if it is $\calH^k_g$ measurable. Moreover, the families of sets of locally finite $k$-dimensional Hausdorff measure are independent of the choice of Riemannian metric by \eqref{eqn: Lipschitz estimate for Hausdorff measures}. If  $\varphi:M\to N$ is a diffeomorphism, then $\varphi$ becomes an isometry if we equip   $N$ with the pullback metric $\tilde{g}=(\varphi\inv)^*g$. In particular, the Riemannian distance is preserved under $\varphi$, hence the Hausdorff measures satisfy
    \begin{equation}\label{eqn: pushforward of Hausdorff measures}
        \varphi_*\calH^{k}_g =\calH^k_{(\varphi\inv)^*g} = \calH^k_{\tilde{g}} \qand (\varphi\inv)_*\calH^k_{\tilde{g}}=\calH^k_g.
    \end{equation}
By these considerations, the following definition is independent of the choice of Riemannian metric:
\begin{definition}\label{def: rectifiability}
    A $\calH^k$-measurable set $S \subset M$ is called \emph{countably $\calH^k$-rectifiable} if there exist countably many Lipschitz functions $f_j: \R^k\to M$, $1\leq j<\infty$, such that
    \begin{equation*}
        \calH^k \biggl( S\setminus\bigcup_j f_j(\R^k) \biggr) = 0.
    \end{equation*}
    If in addition, $\calH^k(S)<\infty$, then $S$ is called \textit{$\calH^k$-rectifiable} and if $\calH^k(S\cap K)<\infty$ for compact sets $K\subset M$, then $S$ is called \textit{locally $\calH^k$-rectifiable}.
\end{definition}
 
\begin{remark}\label{rmk: rectifiability in charts}
 Using Borel regularity of Hausdorff measures, further characterizations of rectifiable sets in $\R^m$ can be generalized to Riemannian manifolds: 
 \begin{enumerate}[label =(\roman*), ]
     \item A set $S\subset M$ is countably $\calH^k$-rectifiable if and only if there exists a  disjoint decomposition
     $S = S_0 \cup \bigcup_{j=1}^\infty  f_j(A_j)$
    with Lipschitz functions $f_j: \R^k\to M$, $1\leq j<\infty$, Borel sets $A_j\subset \R^k$ and a Borel set $S_0$ with $\calH^k(S_0) = 0$. 
    \item  A set $S\subset M$ is countably $\calH^k$-rectifiable if and only if it can be written as a  disjoint union $S = {\bigcup_{j\geq 0}}\, S_j$
    of Borel subsets $S_j$ of $C^1$-embedded $k$-submanifolds $N_j\subset M$, $1\leq j<\infty$ and a Borel  measurable $\calH^k$-nullset $S_0$. 
 \end{enumerate} 
 We stress that these decompositions are non-unique. Using the first characterization and the fact that charts are locally bi-Lipschitz maps, it is straightforward to show the following: 
 \begin{enumerate}
     \item[(iii)]  A  set  $S\subset M$ is countably $\calH^k$-rectifiable if and only if for each chart $(U,\varphi)$, the set $\varphi(S)$ is countably $\calH^k$-rectifiable in $\R^m$. 
 \end{enumerate}
\end{remark}
The previous remark will permit the generalization  of a  characterization of locally $\calH^k$-rectifiable sets in terms of tangential properties. 
 By the classical Euclidean theory, $S\subset\R^m$ has locally finite $\calH^k$ measure, then $S$  is locally $\calH^k$-rectifiable if and only if $S$ admits an {approximate tangent space $T_xS$} at $\calH^k$-almost every $x\in S$. That is, there exists a $k$-dimensional subspace $T_xS$ of $\R^m$ such that
    \begin{equation*}
        \frac{( \phi_{x,r})_*(\calH^k_e\llcorner S)}{r^k}  \overset{*}{\rightharpoonup} \calH^k_e\llcorner T_xS \text{ on }\R^m \text{ as } r\downarrow 0,
    \end{equation*}
      where $\phi_{x,r}(y):= \frac{y-x}{r}$, defines \emph{blow-ups of $S$ at $x$} by $S_{x,r}:= \phi_{x,r}(S)= \frac{S-x}{r}$.
      Using \eqref{eqn: Lipschitz estimate for Hausdorff measures}, we  observe  that  $(\phi_{x,r})_*\calH^k_e = \calH^k_e\circ \phi_{x,r}\inv  = r^k\calH^k_e$, and therefore \footnote{Here and throughout this section, we frequently use that for a Radon measure $\mu$ on $M$, a measurable set $E\subset M$ and a measurable map $f:M\to N$ one has $f_*(\mu\llcorner E) = (f_*\mu)\llcorner(f(E))$.}, 
     \begin{equation*}
          \frac{( \phi_{x,r})_*(\calH^k_e\llcorner S)}{r^k}  
         =\calH^k_e\llcorner(\tfrac{S-x}{r}).
     \end{equation*}
We define the approximate tangent spaces on $M$ in terms of local charts:
\begin{definition}\label{def: approximate tangent space in charts}
    Let $S\subset M$ be a set of locally finite $\calH^k$ measure. We say that $S$ admits an \textit{approximate tangent space $T_xS$ at $x$} if  $\varphi(S)$ admits an approximate tangent space $T_{\varphi(x)}\varphi(S)$ at $\varphi(x)$ in $\R^m$ for some chart $\varphi$ around $x$. In that case, we set $T_xS:= (d\varphi\inv)_{\varphi(x)}(T_{\varphi(x)}\varphi(S))$.
\end{definition}
In the following lemma we will show that the above definition is independent of the choice of chart. Then by Remark \ref{rmk: rectifiability in charts}, this implies that a set $S\subset M$ of locally finite $\calH^k$-measure is locally $\calH^k$-rectifiable if and only if $S$ admits an approximate tangent space at $x$.
\begin{lemma}\label{lem: approximate tangent space independent of chart}
      The approximate tangent space is well-defined, that is, independent of the choice of chart.
\end{lemma}
\begin{proof}
     Let $S\subset M$ be a set of locally finite $\calH^k$-measure.
     
    \textit{(i) Preparations.}  
    Without loss of generality, let $\varphi$ and $\psi$ be two charts around $x\in S$ with the same chart domain $U$ and such that $\varphi(x)=0=\psi(x)$. 
    \begin{itemize}[wide, itemsep=6pt]
        \item  Writing  $\phi_r(y):= y/r$, we define  
        $$
        \calL_r:= \phi_r\circ \psi \circ \varphi\inv\circ \phi_{1/r} :\tfrac{\varphi(U)}{r}\to\tfrac{\psi(U)}{r}
        $$
        for $r>0$. Then 
        $\calL_r(\tfrac{\varphi(S)}{r}) = \tfrac{\psi(S)}{r}$  
        and we have
     \begin{equation*}
           \lim_{r\downarrow0} \calL_r(z) = \lim_{r\downarrow0} \frac{(\psi\circ\varphi)\inv(rz)}{r} = d(\psi\circ\varphi\inv)_0 z
       =:\calL_0 (z)
       \tfor z\in \tfrac{\varphi(U)}{r} 
     \end{equation*}
     locally uniformly in $z$ due to  smoothness of $\psi\circ\varphi\inv$. Similarly,   we get $d\calL_r \to d\calL_0 =\calL_0$ locally uniformly  as $r\downarrow0$.
    \item If $\tilde{S}\subset \R^m$ is a locally $\calH^k$-rectifiable set, then the projection of $d\calL_r(z)$ to the approximate tangent space of $\tilde{S}$ at $z$ is well-defined at $\calH^k$-almost every $z\in \tilde{S}\cap\varphi(U)$ and we shall denote its determinant by $\det ({d\calL_r}_{|\tilde{S}})(z)$.
    By the  area formula for locally $\calH^k$-rectifiable sets in $\R^m$ (see \cite[Ch.3 \S 2]{Sim18_IntroductionGeometricMeasure}),  we have for arbitrary  $f\in C_c^\infty(\R^m)$,
    \begin{equation}\label{eqn: transformation of restricted measure under diffeo}
        \calH^k_e\llcorner (\calL_r(\tilde{S}))(f) 
        = \int_{\tilde{S}} f\circ\calL_r |\det ({d\calL_r}_{|\tilde{S}})|\d\calH^k_e ,
    \end{equation}
    and   one always has $|\det ({d\calL_r}_{|\tilde{S}})(z)|\leq  k\cdot \norm {d\calL_r(z)}$.
    \item  Using \eqref{eqn: pushforward of Hausdorff measures}, we have for arbitrary $\calH^k$-measurable $E$,
    \begin{equation}\label{eqn: transformation of restricted measure under linear map}
        \big(\calH^k_e\llcorner(\calL_r(E))\big) (f) = \left((\calL_r)_* (\calH^k_{L_r^*\,e}\llcorner E \right) (f) = (\calH^k_{L_r^*\,e}\llcorner E)(f\circ\calL_r).
    \end{equation}
    In particular, for $r=0$ the map $\calL_0 = d(\psi\circ\varphi\inv)_0:\R^m\to\R^m$ is linear, so  $\calL_0^*\,e$  is a scalar product on $\R^m$ whose induced distance is globally equivalent to the Euclidean distance. Therefore, the corresponding Hausdorff measures are equivalent with respective bounded and nonvanishing smooth density functions that are defined globally.
    \end{itemize}

    \textit{(ii)}
    \textit{We assume that $\calH^k_e\llcorner(\tfrac{\varphi(S)}{r}) \overset{*}{\rightharpoonup} \calH^k_e\llcorner T_0\varphi(S)$
    and prove that this implies $$\calH^k_e\llcorner(\tfrac{\psi(S)}{r}) \overset{*}{\rightharpoonup} \calH^k_e \llcorner d(\psi\circ\varphi\inv)_0 (T_0\varphi(S)).$$}

    To this end, fix $f\in C_c^\infty(\R^m)$ and let  $K\subset \R^m$  be a compact set such that  
    $$ K_f:=\bigcup_{0\leq r\leq1}\sppt{(f\circ \calL_r)}
    $$
    is compactly contained in the interior $\mathring{K}$ of $K$. 
     Now write $f_r:=f\circ \calL_r\in C_c^\infty(\mathring{K})$ and use \eqref{eqn: transformation of restricted measure under diffeo}  and \eqref{eqn: transformation of restricted measure under linear map} for the $\calH^k$-rectifiable set $ \varphi(S)/r$ to expand
   \begin{align}
       &\calH^k_e\llcorner (\calL_r(\tfrac{\varphi(S)}{r}))(f) -  \calH^k_e\llcorner(\calL_0(T_0\varphi(S))(f)\nonumber\\ 
       &\qquad\qquad\quad= \int_{\frac{\varphi(S)}{r}\cap K} f_r |\det ({d\calL_r}_{|\frac{\varphi(S)}{r}})|\d\calH^k_e 
       - \int_{\frac{\varphi(S)}{r}\cap K} f_0 |\det ({d\calL_0}_{|\frac{\varphi(S)}{r}})|\d\calH^k_e \label{eqn: approximate tangent space expansion1} \\
      &\qquad\qquad\qquad+ \big(\calH^k_{L_0^*\,e}\llcorner (\tfrac{\varphi(S)}{r})\big)(f_0)
       - \big(\calH^k_{L_0^*\,e}\llcorner (T_0\varphi(S))\big)(f_0) .\label{eqn: approximate tangent space expansion2}
    \end{align}

    In order to prove convergence of the difference in \eqref{eqn: approximate tangent space expansion1}, first note that by uniform convergence of $d\calL_r$ to $d\calL_0$, we have
    \begin{equation*}
        \sup_{0\leq r\leq 1}\sup_{z\in K } \norm {d\calL_r(z)} < C_1.
    \end{equation*}
     If $\xi \in C_c^\infty(\mathring{K})$ is such that $0\leq \xi\leq 1$ and $\xi\equiv 1$ on $K_f$, then we can estimate   \eqref{eqn: approximate tangent space expansion1} above by
    \begin{equation*}
      \sup_{z\in K }   k\cdot\norm {d\calL_r(z)} \int_{\frac{\varphi(S)}{r}\cap K}|f_r (z) -f_0(z)| \xi\d\calH^k_e \leq  k\cdot C_1
           \sup_{z\in K} |f_r (z) -f_0(z)| \int_{\frac{\varphi(S)}{r}} \xi\d\calH^k_e.
    \end{equation*}
    The term $\sup_{z\in K}|f_r(z) - f_0(z)|$ converges to zero  since $\calL_r\to \calL_0$ locally uniformly and the factor term converges to $(\calH^k_e\llcorner (T_0\varphi(S)))(\xi)<\infty$ by assumption. Therefore, \eqref{eqn: approximate tangent space expansion1} vanishes for $r\downarrow0$.

    For convergence of \eqref{eqn: approximate tangent space expansion2}, it is sufficient to note that the Hausdorff measures induced by the  distance functions corresponding to the scalar products $e$ and $L_0^*e$ on $\R^m$ are globally equivalent. This implies that \eqref{eqn: approximate tangent space expansion2} converges to zero.

     Thus, we have shown that $\psi(S)$ has an approximate tangent space at $\psi(x)=0$ given by $T_0\psi(S) = \calL_0(T_0\varphi(S))$.
    We conclude that the existence of an approximate tangent space is independent of the choice of chart. Recalling that $\calL_0= d(\psi\circ\varphi\inv)_0$ and applying the chain rule yields
    \begin{equation*}
        d\psi\inv_{\psi(x)}(T_0\psi(S)) = d\varphi\inv_{\varphi(x)}(T_0\varphi(S)),
    \end{equation*}
    so $T_xS$ is determined uniquely.    
\end{proof}
\begin{remark}\label{rmk: approximate tangent spaces equal classical tangent spaces}
    If $S\subset M$ is locally $\calH^k$-rectifiable and admits an approximate tangent space $T_xS$ at $x\in S$, then
    $T_xS = T_xN_j,$
    where $N_j$ is a $C^1$-embedded $k$-submanifold such that $S=\bigcup_j S_j$ and $S_j\subset N_j$ in the sense of Remark \ref{rmk: rectifiability in charts}. We stress that this is a nontrivial statement since the decomposition is non-unique. Indeed, this can be derived from the corresponding Euclidean statement: Given a chart $\varphi$, any such decomposition $S=\bigcup S_j$ with $S_j\subset N_j$ naturally induces a decomposition $\varphi(S) =\bigcup \varphi(S_j)$ with $ \varphi(S_j)\subset \varphi(N_j)$. Then the approximate tangent space of $S$ at $x$ is uniquely defined by
    $$
    T_xS =(d\varphi\inv)_{\varphi(x)}(T_{\varphi(x)}\varphi(S)) =  (d\varphi\inv)_{\varphi(x)}(T_{\varphi(x)}\varphi(N_j)) = T_xN_j.
    $$
\end{remark}

Our next aim is to prove a transformation formula for Hausdorff measures. Let $h$ be another Riemannian metric on $M$ and let $\calA=\calA_{g,h}$ as in~\eqref{eqn: calA endomorphism}.
If $S$ is a  locally $\calH^k$-rectifiable  set in $M$, then we denote by 
$\calA_{|S}(x)$ the restriction of $\calA$ to the approximate tangent space $T_xS$ of $S$ at $x$. Let $\varepsilon_1,\dots,\varepsilon_m$ be a basis for $T_xM$ such that $\varepsilon_1,\dots,\varepsilon_k$ span~$T_xS$. If  $A$ is the matrix representation of $\calA$ at $x$ with respect to $(\varepsilon_{j} )_{j\leq m}$,  then 
\begin{equation*}
    \det(\calA_{|S})(x):= \det([A_{i,j}]_{i,j\leq k})
\end{equation*}
is well-defined for $\calH^k$-almost every $x\in S$. We prove an area formula for locally $\calH^k$-rectifiable sets.

\begin{lemma}\label{lem: transformation of Hausdorff measures w.r.t metric}
     Assume in the situation of Remark \ref{rmk: approximate tangent spaces equal classical tangent spaces} that $S$ is locally $\calH^k$-rectifiable in $M$. Then for all Borel subsets $B\subset M$, the Hausdorff measures with respect to $h$ and $g$ transform via
    \begin{equation}\label{eqn: change of variables for k-dim Hausdorff measures}
        \calH^k_h\llcorner S(B) = \int_{S\cap B} \sqrt{\det (\calA_{|S})} \d \calH^k_g.
    \end{equation}
\end{lemma}
\begin{proof}
    First assume that $S$ is a $C^1$-embedded  $k$-submanifold of $M$. Then  the restriction of $\calH^k_g$ to $S$ equals the Riemannian volume of the induced  metric $g_{|S}$, that is,
    % \begin{equation*}
        $\calH^k_g \llcorner S = \vol_{g_{|S}},$
    % \end{equation*}
    cf. \cite[Section 1.5.2]{Rit23_IsoperimetricInequalitiesRiemannian}. Since the induced metrics satisfy 
        $g_{|S}(\calA_{|S}\,\cdot,\cdot) = h_{|S}(\cdot,\cdot),$
     the classical change of variables formula on manifolds implies
     \begin{equation*}
        \calH^k_h \llcorner S = \vol_{h_{|S}} =  \sqrt{\det (\calA_{|S})} \vol_{g_{|S}} =\sqrt{\det (\calA_{|S})} \calH^k_g \llcorner S.
    \end{equation*}
    Next, let $S = \bigcup_j S_j$ be a disjoint decomposition into Borel subsets $S_j$ of $C^1$ submanifolds $N_j$ for $j\geq 1$ and $\calH^k_g(S_0) = 0$. Then by Remark \ref{rmk: approximate tangent spaces equal classical tangent spaces}, if $x\in N_j$ we have $T_xN_j = T_xS$ hence
    \begin{equation*}
        \begin{gathered}
        \calH^k_h(S\cap B) 
        =  \sum_{j\geq 1}\calH^k_h (S_j\cap B) 
        =  \sum_{j\geq 1}\calH^k_h ((S\cap B)\cap N_j) \\
        =\sum_{j\geq 1} \int_{S_j\cap B} \sqrt{\det (\calA_{|N_j})} \d\calH^k_g  
        = \int_{S\cap B} \sqrt{\det (\calA_{|S})} \d\calH^k_g.
        \end{gathered}
    \end{equation*}
\end{proof}
Since the approximate tangent space for locally $\calH^k$-rectifiable  $S\subset M$ is well-defined $\calH^k$-almost everywhere on $S$ as a $k$-subspace of $T_xM$, we can naturally define an \textit{approximate normal space $N_xS$} as the orthogonal complement of $T_xS$ in $T_xM$ with respect to $g$. 
\begin{corollary}\label{cor: transformation Hausdorff mesures of codimension one}
    Let $S$ be locally $\calH^{m-1}$-rectifiable and $\mathbf{n}^S_g(x)$  a unit vector spanning the approximate normal space $N_xS$ whenever it exists. Then 
    \begin{equation}\label{eqn: change of variables for m-1 Hausdorff measure}
        \calH^{m-1}_h\llcorner S(B) = \int_{S\cap B} \sqrt{\det(\calA)\, g(\calA\inv \mathbf{n}^S_g,\mathbf{n}^S_g)} \d \calH^{m-1}_g,
\end{equation}
where we use the same conventions as above.
\end{corollary}
\begin{proof}
In view of Lemma~\ref{lem: transformation of Hausdorff measures w.r.t metric} it suffices to show that for locally $\calH^{m-1}$-rectifiable $S$ one has
\begin{equation*}
    \det (\calA_{|S}) =  \det(\calA)\, g(\calA\inv \mathbf{n}^S_g,\mathbf{n}^S_g)
\end{equation*}
$\calH^{m-1}$-almost everywhere on $S$. Fix $x\in S$ such that $T_xS$ exists, let $\varepsilon_1,\dots,\varepsilon_{m-1}$ be an  orthonormal basis for $T_xS$ with respect to $g$ and let $\varepsilon_m=\mathbf{n}^S_g$ be the normal vector to $S$ at $x$. Then $A_{i,j} := g(\calA(x)\varepsilon_i,\varepsilon_j)$  is the  representation of $\calA$ at $x$ with respect to $(\varepsilon_i)_{1\leq i\leq m}$, i.e. $A_{i,j}=\varepsilon_i\cdot A\varepsilon_j$. Let $\mathbf{a}_{i,j}$ be the $(i,j)$-th minor of $A$ and  denote by $\mathrm{adj}(A)_{i,j}=(-1)^{i+j} \mathbf{a}_{j,i}$ the adjugate matrix of $A$. Recall that by Cramer's rule one has
$    \mathrm{adj}(A) = \det( A)\, A\inv.
$
Using $\varepsilon_m=\mathbf{n}^S_g$, we  compute
    \begin{equation*}
        \det (\calA_{|S}) (x)
        = \mathbf{a}_{m,m}
        =  \mathbf{n}^S_g \cdot \mathrm{adj}(A)\, \mathbf{n}^S_g
        =\det( A)\, \mathbf{n}^S_g \cdot A\inv \mathbf{n}^S_g
        = \det(\calA)(x)\, g(\calA\inv(x) \mathbf{n}^S_g,\mathbf{n}^S_g)
    \end{equation*}
and the right- and left-hand side of this identity are independent of the choice of basis  for $T_xS$. 
\end{proof}

For fixed $x\in M$, let $R_x\in(0,\infty]$ be the injectivity radius at $x$. Recall that the exponential map $\mathrm{exp}_x$ is a diffeomorphism from the open ball around $0\in T_xM$ of radius $R_x$ to $B(x,R_x)\subset M$. Induced charts are referred to as normal charts in the sequel.

\begin{definition}
% [Blow-up on $M$]
\label{def: blow-up on M}
     For $x\in M$  and $r< R_x$, we define the \textit{blow-up map at $x$} by
    \begin{equation*}
        \phi^M_{x,r} :B(x,R_x) \to T_xM;
        \quad 
        \phi^M_{x,r}(y):
        = \frac{\exp_x\inv (y) }{r}
    \end{equation*}
    and call  $S_{x,r} :=  \phi^M_{x,r}(S):= \phi^M_{x,r}(S\cap B(x,R_x)) \subset T_xM$ the \emph{blow-ups of $S\subset M$ at $x$}.
\end{definition}
This allows a Riemannian characterization of approximate tangent spaces:
\begin{lemma}\label{lem: Riemannian approximate tangent space with exponential map}
     A set  $S\subset M$ of locally finite $\calH^k$-measure admits an approximate tangent space $T_xS$  at $x$ if and only if  
    \begin{equation*}
         \frac{( \phi^M_{x,r})_*(\calH^k_g\llcorner S)}{r^k} \overset{*}{\rightharpoonup}  \calH^k_{g_x}\llcorner T_xS \quad\text{ as }r\downarrow 0,
    \end{equation*}
    where $\calH^k_{g_x}$ is the $k$-dimensional Hausdorff measure on the fiber $(T_xM, g_x)$. In that case, $T_xS$ coincides with the approximate tangent space from Definition~\ref{def: approximate tangent space in charts}.
\end{lemma}
\begin{proof}
    The statement is the special case of Lemma~\ref{lem: Riemannian approximate tangent space chart wise} applied to normal coordinates around~$x$.
\end{proof}

\begin{lemma}\label{lem: Riemannian approximate tangent space chart wise}
    A   set $S$ of locally finite $\calH^k$ measure has an approximate tangent space $T_xS$ at $x\in S$ if and only if for each chart $(U,\varphi)$ around $x$ one has 
    \begin{equation}\label{eqn: Riemannian approximate tangent space chart wise}
        \frac{(\phi_{\varphi(x),r}\circ\varphi)_* (\calH^k_g \llcorner S )}{r^k} \overset{*}{\rightharpoonup } (d\varphi_x)_* (\calH^k_{g_x} \llcorner (T_xS)) \text{ on } \R^m \text{ as } r\downarrow 0.
    \end{equation}
\end{lemma}

\begin{proof}
    Again we assume without loss of generality that $(U,\varphi)$ is a chart around $x$ with $\varphi(x)=0$, and let $\tilde{g}= (\varphi\inv)^*g$.
    Using that
    $\varphi_*\calH^k_g=\calH^k_{\tilde{g}}$ and $(d\varphi_x)_* \calH^k_{g_x} = \calH^k_{\tilde{g}_0}$, it follows that \eqref{eqn: Riemannian approximate tangent space chart wise} is equivalent to
    \begin{equation}\label{eqn: Riemannian approximate tangent space chart wise via blow ups}
        \calH^k_{\tilde{g}}\llcorner\big(\tfrac{\varphi(S)}{r}\big) \overset{*}{\rightharpoonup} \calH^k_{\tilde{g}_0}\llcorner\big(T_0\varphi(S)).
    \end{equation}
    Moreover, by linearity of $d\varphi_x:T_xM\to\R^m$, one can argue as in the proof of Lemma~\ref{lem: approximate tangent space independent of chart} to find that by global equivalence of $\calH^k_e$ and $\calH^k_{\tilde{g}_0}$
    it suffices to prove that \eqref{eqn: Riemannian approximate tangent space chart wise via blow ups} holds if and only if
    \begin{equation}\label{eqn: euclidean approximate tangent space w.r.t pushforward metric}
    \calH^k_{\tilde{g_0}}\llcorner\big(\tfrac{\varphi(S)}{r}\big) \overset{*}{\rightharpoonup} \calH^k_{\tilde{g_0}}\llcorner\big(T_0\varphi(S)\big).
    \end{equation}

     We assume  \eqref{eqn: euclidean approximate tangent space w.r.t pushforward metric} and show that 
    \begin{equation}
        \calH^k_{\tilde{g}} \llcorner \big(\tfrac{\varphi ( S)}{r} \big) - \calH^k_{\tilde{g}_0} \llcorner \big(\tfrac{\varphi ( S)}{r} \big)  \overset{*}{\rightharpoonup} 0
    \end{equation}

    Let $f\in C_c^\infty(\R^m)$ be arbitrary and let $\calA = \calA_{e,\tilde{g}}$ as in \eqref{eqn: calA endomorphism}.  We use Lemma~\ref{lem: transformation of Hausdorff measures w.r.t metric} to estimate
    \begin{equation*}
    \begin{aligned}
        \Big(\calH^k_{\tilde{g}} \llcorner \big(\tfrac{\varphi ( S)}{r} \big)- \calH^k_{\tilde{g}_0} \llcorner \big(\tfrac{\varphi ( S)}{r} \big)\Big )(f) 
        &=r^{-k}\Big(\calH^k_{\tilde{g}} \llcorner \big(\varphi ( S) \big)- \calH^k_{\tilde{g}_0} \llcorner \big(\varphi ( S) \big)\Big )(f\circ \phi_r) \\
        &= r^{-k}\int_{ \varphi
        (S)} f(\tfrac{z}{r}) \big(|\det\calA_{| \varphi(S)} (z)  
        -  |\det\calA_{| \varphi(S)} (0)| \big) \d\calH^k_e(z)\\
        &\leq \sup_{\tfrac{z}{r}\in  \sppt(f) } \big||\det\calA_{| \varphi(S)}(z)| - |\det\calA_{|\varphi(S)}(0)|\big| \int_{\frac{\varphi
        (S)}{r}}| f| \d\calH^k_e.
        \end{aligned}
    \end{equation*}
    The last factor is bounded  since $\calH^k_e\llcorner(\tfrac{\varphi(S)}{r})\overset{*}{\rightharpoonup}{\calH^k_e}\llcorner(T_0\varphi(S))$ by assumption. Moreover, using that  $z\mapsto\det|\calA_{|\varphi(S)}(z)|$ is locally Lipschitz, there exists some constant $C>0$ such that the whole expression in the last line is bounded above by 
    $C\cdot\sup\{|z|:\tfrac{z}{r}\in \sppt(f)\}$ and vanishes as $r$ tends to zero. This implies  \eqref{eqn: Riemannian approximate tangent space chart wise via blow ups}.
    To prove the other direction, assume that \eqref{eqn: Riemannian approximate tangent space chart wise via blow ups} holds and  write 
    $$ \calH^k_{\tilde{g_0}}\llcorner\big(\tfrac{\varphi(S)}{r}\big) (f)
    = \int_{\varphi(S)} f(\tfrac{z}{r}) |\det((\calA\inv(z)\calA(0))_{|\varphi(S)})| \d\calH^k_e.
    $$
    Then  the same line of arguments as before shows that  \eqref{eqn: euclidean approximate tangent space w.r.t pushforward metric} holds true.
\end{proof}

\begin{lemma}\label{lem: convergence of blow-ups in charts}
    Let $H\subset T_xM$ be a measurable set such that $\vol_{g_x}(\partial H)=0$ and let $\varphi$ be a local chart around $x$. The blow-ups $E_{x,r}\subset T_xM$ of a measurable set $E\subset M$ at $x\in E$ converge locally in $\vol_{g_x}$ to $H$ as $r\downarrow0$ if and only if the blow-ups of $\varphi(E)\subset \R^m$ at $\varphi(x)$ converge locally in Lebesgue measure $\lambda$ to $d\varphi_x (H)$.
\end{lemma}
\begin{proof}
    We first  assume that $\varphi$ is a normal chart centered in $x$ and let  $\tilde{g} =( \varphi\inv)^*g$.  Then convergence of $\phi^M_{x,r}(E)\subset T_xM$ to $H$ locally in $\vol_{g_x}$ is clearly equivalent to convergence of  $\varphi(E)\subset \R^m$ to $d\varphi_0(H)$ with respect to the induced volume measure $\vol_{\tilde{g}_0}$  due to the canonical identification of $(T_xM,g_x)$ and $(\R^m, \tilde{g}_0)$. Since  $\vol_{\tilde{g}_0}$ has a constant density with respect to the Lebesgue measure, this shows equivalence of both statements in the special case where $\varphi$ is a normal chart. It remains to show that convergence of the blow-ups $\varphi(E)$ at $\varphi(x)$ in $\R^m$ is independent of the choice of chart. 
    To this end, consider two charts $\varphi$ and $\psi$ as in part $(i)$ of the proof of Lemma~\ref{lem: approximate tangent space independent of chart} as well as the maps $\calL_r$ and $\calL_0$ and the corresponding constructions.
 Then we have 
$$\tfrac{\psi(E)}{r} = \calL_r(\tfrac{\varphi(E)}{r})\qand 
\mathbf{1}_{{\psi(E)}/{r}} = \mathbf{1}_{{\varphi(E)}/{r}}\circ \calL_r\inv.$$
We assume without loss of generality that $\varphi(E)/r$ locally converges in measure to $d\varphi_x(H)$ and show that $\psi(E)/r$ locally converges in measure to $\calL_0\inv(d\varphi_x(H))=d\psi_x(H)$. Fix a relatively compact set $V\subset \R^m$. Then
\begin{equation}\label{eqn: triangel inequality for convergence of blow ups in measure}
    \int_V |\mathbf{1}_{{\psi(E)}/{r}} - \mathbf{1}_{d\psi_x(H)}|\d\lambda 
    \leq \int_V |(\mathbf{1}_{{\varphi(E)}/{r}} - \mathbf{1}_{d\varphi_x(H)})\circ\calL_r\inv |\d\lambda 
    +  \int_V |\mathbf{1}_{d\varphi_x(H)} \circ \calL_r\inv- \mathbf{1}_{d\varphi_x(H)}\circ\calL_0\inv|\d\lambda 
\end{equation}
and we prove convergence of the first and second term on the right-hand side separately. Since  $\calL_r\to\calL_0$ uniformly, and similarly $\calL_r\inv\to \calL_0\inv$ uniformly, we can find a compact set $K(V)\subset\R^m$ containing all sets $\calL_r\inv(V)$ for $0\leq r \leq 1$.
Using the transformation formula, the first term on the right-hand side of \eqref{eqn: triangel inequality for convergence of blow ups in measure} becomes
\begin{equation*}
    \int_{\calL_r\inv(V)} |(\mathbf{1}_{{\varphi(E)}/{r}} - \mathbf{1}_{d\varphi_x(H)}) \det (d\calL_r)| \d\lambda
    \leq \int_{K} |\mathbf{1}_{{\varphi(E)}/{r}} - \mathbf{1}_{d\varphi_x(H)}| \d\lambda \sup_{z\in K(V)}|\det (d\calL_r(z))| \to 0.
\end{equation*}
For the last term of \eqref{eqn: triangel inequality for convergence of blow ups in measure}, note that the discontinuity set $\partial (d\varphi_x(H))$ of $\mathbf{1}_{d\varphi_x(H)}$ satisfies $\lambda(\calL_0 (\partial( d\varphi_x(H)))) = 0$ since $\calL_0$ is a linear bijection. Thus, using that $\calL_r\inv$ converges to $\calL\inv$ locally in measure, we can apply the continuous mapping theorem to conclude that $ \mathbf{1}_{d\varphi_x(H)}\circ \calL_r\inv \to \mathbf{1}_{d\varphi_x(H)}\circ \calL_0\inv$ locally in measure and this is equivalent to convergence locally in $L^1$. Therefore, the last term of \eqref{eqn: triangel inequality for convergence of blow ups in measure} tends to zero as $r\downarrow 0$ and the proof is complete.
\end{proof}

We use these constructions to prove the announced Riemannian version of the classical structure theorem due to De~Giorgi. 

\begin{theorem}[De~Giorgi]\label{thm: De~Giorgi}
    Let $E\subset M$ have (locally) finite perimeter.
    \begin{enumerate}[label=(\roman*)]
        \item The reduced boundary $\partial^*E$ is (locally) $\calH^{m-1}$-rectifiable.
        \item For  every $x\in \partial^*E$ the approximate tangent space $T_x\partial^*E$ exists and is given by the  orthogonal complement of $\normal[]{E}(x)$ in $T_xM$.
        \item The perimeter measure of $E$ satisfies ${\TV{\mathbf{1}_E}} = \calH^{m-1}\llcorner(\partial^*E)$.
        \item If $x\in \partial^*E$, then the blow-ups $E_{x,r}:=\phi^M_{r}(E)\subset T_xM$ locally converge as $r\downarrow 0$ in measure to the half space $H^E(x)= \set{\xi \in T_xM: (\xi, \normal{E}(x)) \geq 0}$.
    \end{enumerate}   
\end{theorem}
\begin{proof}
We cover $M$ by countably many  charts and let  $\varphi:U\to \tilde{U}=\varphi(U) \subset \R^m$ be such a chart with inverse $\varphi\inv:\tilde{U}\to U$. Set $\tilde{g}:=(\varphi\inv)^*g$ 
    and $\calA = \calA_{e,\tilde{g}}$ as in \eqref{eqn: calA endomorphism}.

 \begin{enumerate}[wide, itemsep=6pt, label = (\roman*)]
     \item Combine Corollary~\ref{cor: chart wise reduced boundary} with Remark \ref{rmk: rectifiability in charts} and apply the corresponding statement for $\R^m$ from \cite[Cor.~16.1]{Mag12_SetsFinitePerimeter}.
     \item  By Definition~\ref{def: approximate tangent space in charts} and Lemma~\ref{lem: approximate tangent space independent of chart}, the Riemannian approximate tangent space of $\partial^*E$ at $x$ exists if and only if the Euclidean approximate tangent space of $\varphi(\partial^*E)$ exists at $\varphi(x)$ and one has 
     \begin{equation}\label{eqn: pushforward of tangent spaces}
          T_{\varphi(x)}\varphi(\partial^*E) = d\varphi(T_x\partial^*E).
     \end{equation}
      Moreover, in view of Corollary~\ref{cor: chart wise reduced boundary},
     the  Euclidean normal vector $\normal[e]{\varphi(E)}(\varphi(x))$  and the Riemannian normal vector $\normal{E}(x)$  satisfy 
     $$ \normal[e]{\varphi(E)}(\varphi(x)) = c\cdot\,{\calA\, d\varphi (\normal{E}(x))} \quad\text{ with } \quad c= |\calA\, d\varphi (\normal{E}(x))|_e\inv.
     $$ 
     Now for arbitrary $\xi\in T_xM$, one has
    \begin{equation}\label{eqn: scalar product of normal and other tangent vectors}
        \big(\xi,\normal{E}(x)\big)_g = \big(d\varphi (\xi), \calA\, d\varphi  (\normal{E}(x))\big)_e = c\cdot\, \big( d\varphi (\xi), \normal[e]{\varphi(E)}(\varphi(x)\big)_e.
    \end{equation}
     Together with \eqref{eqn: pushforward of tangent spaces} and   the  Euclidean statement in \cite[Cor.~16.1]{Mag12_SetsFinitePerimeter}, this implies   that the approximate tangent space at $x$ is given as the orthogonal complement of the measure-theoretic inner normal.
    \item  Let $B\subset U$ be a Borel measurable subset of $U$. 
% With $\calA$ as above, the  induced volume measures transform via $\vol_h = \sqrt{\det(\calA)}\,\vol_e$.
    Applying the pushforward relations for covector measures in Remark \ref{rmk: pushfoward of vector measures}, we have
    \begin{equation}\label{eqn: De~Giorgi first eqn}
            \TV[g]{\mathbf{1}_E}(B) 
        = \TV[\tilde{g}]{(\mathbf{1}_{E}\circ\varphi\inv)}(\varphi(B))
        = \TV[\tilde{g}]{\mathbf{1}_{\varphi(E)}} (\varphi(B)).
    \end{equation}
    Since the polar  vector field of the Euclidean variation measure of $\mathbf{1}_{\varphi(E)}$ coincides with the Euclidean generalized normal vector $\normal[e]{\varphi(E)}$ on $\partial^*(\varphi(E))\cap\varphi(U)$, we can apply \eqref{eqn: g to h transformation of unweighted variation} to get 
    \begin{equation}\label{eqn: De~Giorgi second eqn}
       \TV[\tilde{g}]{\mathbf{1}_{\varphi(E)}} 
        = \sqrt{ \det(\calA)\, \big(\calA\inv \normal[e]{\varphi(E)},\normal[e]{\varphi(E)}\big)_e}\,  \TV[e] {\mathbf{1}_{\varphi (E)}}.
    \end{equation}
    By De~Giorgi's theorem in Euclidean space we have that
    \begin{equation} \label{eqn: De~Giorgi third eqn}
         \TV[e] {\mathbf{1}_{\varphi (E) }} (\varphi(B)) = \calH^{m-1}_e \big(\partial^*(\varphi(E)) \cap \varphi(B)\big)
    \end{equation}
    and  $\normal[e]{\varphi(E)}$ is orthogonal to the approximate tangent space of $\partial^*\varphi(E)\cap\varphi(U)$.
     Therefore,  Corollary~\ref{cor: transformation Hausdorff mesures of codimension one}  yields
    \begin{equation}\label{eqn: De~Giorgi last eqn}
     \sqrt{\det(\calA)\, \big(\calA\inv \normal[e]{\varphi(E)},\normal[e]{\varphi(E)}\big)_e}\, \, \calH^{m-1}_e \big(\partial^*(\varphi(E)) \cap \varphi(B)\big)
         = \calH^{m-1}_{\tilde{g}}\big(\partial^* (\varphi(E)) \cap \varphi(B)\big).
    \end{equation}
    Combining \eqref{eqn: De~Giorgi first eqn}--\eqref{eqn: De~Giorgi last eqn} and using $\calH^k_{\tilde{g}}=\calH^k_g\circ \varphi\inv$ on $\varphi(U)$,  and  $\varphi\inv (\partial^* (\varphi(E)) \cap \varphi(B)) = \partial^*E\cap B$,   we   conclude that $\TV[g]{\mathbf{1}_E}(B) 
        = \calH^{m-1}_g (\partial^*E\cap B).$
    \item This follows from combining Lemma~\ref{lem: convergence of blow-ups in charts} with \eqref{eqn: scalar product of normal and other tangent vectors} and  using the Euclidean statement from \cite[Theorem~15.5]{Mag12_SetsFinitePerimeter}. 
 \end{enumerate}
\end{proof}

%%%%%%%%%%%%%%%%%%%%%%%%%%%%%%%%%%%%%%%%%%%%%%%%%%%%%%%%%%%%
\subsection{The measure-theoretic boundary and Federer's Theorem}
\label{sec: measure-theoretic boundary and Federer's Theorem}
%%%%%%%%%%%%%%%%%%%%%%%%%%%%%%%%%%%%%%%%%%%%%%%%%%%%%%%%%%%%
Let $E$ be a  measurable subset of $M$. For $r>0$ define
\begin{equation*}
    \Theta_E(x,r) :=\frac{\vol(E\cap B(x,r))}{\vol(B(x,r))}\in [0,1]
    \qand
    \Theta_E(x) :=\lim_{r\downarrow 0} \Theta_E(x,r) .
\end{equation*}
Then $\Theta_E(x) $ is called the \emph{density of $E$ at $x\in M$} and one has  
\begin{equation}\label{eqn: density of complement}
    \Theta_E(x) = 1-\Theta_{E^c}(x).
\end{equation}
For $t\in [0,1]$, we denote by $E^{(t)}$ the set of points $x\in M$ where $E$ has density~$t$.
Clearly, $E$ has density $1$ at all  points in the topological interior of $E$ and density $0$ at all  points in the topological exterior of $E$.
Moreover, by Theorem~\ref{thm: Lebesgue--Besicovitch--Federer}~\ref{item: differentiation thm for functions} for $f= \mathbf{1}_E$ one has  $\Theta_E(x) = 1$ for  almost every $x\in E$ and $\Theta_E(x) = 0$ for almost every $x\in M\setminus E$ with respect to the volume measure. Thus,  $E^{(1)}$ and $E^{(0)}$ are called the \textit{essential interior} and \textit{essential exterior}, respectively.
The \emph{essential boundary}  of $E$ is defined as the set
\begin{equation*}
    \partial^{\ess}E:= \set{ x\in M :\, \Theta_E(x) \in (0,1) } \subset \partial E.
\end{equation*}
Sometimes these sets are referred to as \textit{measure-theoretic} interior, exterior, and boundary. If $U$ is an open neighborhood of $x$, then $\Theta_E(x) = \Theta_{E\cap U}(x)$. Moreover, the set $\partial^{\ess}E$ is independent of the Riemannian metric $g$   on $M$. To see this, let $h$ be another Riemannian metric and recall that  locally, $\vol_h = \varrho \vol_g$ with a smooth, positive density function  $\varrho$ that is bounded away from zero. This implies
    \begin{equation*}
        \frac{\vol_h(E\cap B(x,r) )}{\vol_h(B(x,r) )}  \lesssim \frac{\vol_g(E\cap B(x,r) )}{\vol_g(B(x,r) )} \lesssim \frac{\vol_h(E\cap B(x,r) )}{\vol_h(B(x,r) )},
    \end{equation*}
    and we note that the same holds if we replace the volume measure with any absolutely continuous measure on~$M$.
    Then with \eqref{eqn: density of complement}, the density of $E$ with respect to $\vol_h$ at a point $x$ equals zero (or one) if and only if the density with respect to $\vol_g$ equals zero (or one). This means that $\smash{E^{(0)}}$ and $\smash{E^{(1)}}$ are purely topological data since they are independent of the choice of Riemannian metric. We shall see that $E^{(t)}$ for $t\in(0,1)$, on the other hand, does depend on $g$ and is therefore Riemannian. First we prove invariance of the essential boundary under coordinate charts.

\begin{lemma}\label{lem: essential boundary in charts}
    A point $x\in M$ belongs to $\partial^{\ess}E$ if and only if $\varphi(x)$ belongs to $\partial^{\ess}(\varphi(E)) \cap \varphi(U)$ in $\R^m$ for every chart $(U,\varphi)$ of $M$ with $x\in U$.
\end{lemma}

\begin{proof}
    Let $x\in\partial^{\ess}E$ and let $(U,\varphi)$ be a chart of $M$ containing $x$.
    Let $ r_0>0$ be such that $B_e(\varphi(x), r_0)\subset \varphi(U)$ and suppose that $r_0$ is sufficiently small such that for some $L=L(r_0)>1$ one has
    \begin{equation*}
        L\inv |\varphi(y)-\varphi(z)| \leq \dist(y,z) \leq L\, |\varphi(y)-\varphi(z)| \tforall y,z \in \varphi\inv (B_e(\varphi(x), r_0)),
    \end{equation*}
    and $B(x,Lr_0)\subset U$ (see \cite[Lemma 3.24]{Gri12_HeatKernelAnalysis} for a construction).
    Then for
    % $\alpha = \frac{s_0}{r_0}$ and 
    all $r< r_0$  one has 
    % $B_e(\varphi(x) ,r) \subset B_e(\varphi(x),\alpha r)\subset \varphi(U)$ and 
    \begin{equation*}
        \varphi\inv \left(B_e (\varphi(x), r)\right) \subset B (x, L r)
    \qand
         \varphi(B(x, L\inv r)) \subset B_e (\varphi(x), r ) ,
    \end{equation*}
    hence
    \begin{equation}\label{eqn: nesting of metric balls}
       B (x,  L\inv r))\subset \varphi\inv ( B_e (\varphi(x), r )) \subset B (x,  Lr)).
    \end{equation}
    By Remark \ref{rmk: local doubling condition and Poincare inequality   in arbitrary Riemannian manifolds} there exists a local  doubling  constant $C_v = C_v(r_0)$,  that is,  for all $r<r_0$ we have
    \begin{equation}\label{eqn: volume growth constant for nested balls}
        \vol(B(x, L r)) \leq L^{2m} C_v \,\vol(B(x,L\inv r)).
    \end{equation}
    Since $r$ is sufficiently small such that $\varphi\inv( B_e(\varphi(x),r))\subset U$ and since $\varphi$ is bijective, we have
        \begin{gather*}
        \Theta_{\varphi(E)}(\varphi(x),r) 
         = \frac{\lambda \Bigl( \varphi\bigl(E \bigr) \cap B_e(\varphi(x),r) \Bigr)}{\lambda \Bigl( B_e(\varphi(x),r) \Bigr)}
         = \frac{
        \lambda \Bigl( \varphi \bigl(E\cap \varphi\inv( B_e(\varphi(x),r) )  \bigr) \Bigr)}
        {\lambda \Bigl( \varphi \bigl( \varphi\inv( B_e(\varphi(x),r) )\bigr)\Bigr)
        }\\
         = \frac{ \displaystyle
        \int_{ E\cap \varphi\inv( B_e(\varphi(x),r))} {(\sqrt{\det g})}\inv \d\vol }
        {\displaystyle
        \int_{ \varphi\inv( B_e(\varphi(x),r) )}  {(\sqrt{\det g})}\inv \d\vol }
        \leq C_\varrho
        \frac{ \vol \Bigl(E\cap \varphi\inv( B_e(\varphi(x),r)) \Bigr)}
        {\vol \Bigl(\varphi\inv( B_e(\varphi(x),r) ) \Bigr)}
    \end{gather*}
with 
\begin{equation*}
    C_\varrho =C_\varrho(r_0)=\frac{\displaystyle \sup \set{ \sqrt{\det g} (y):\, y\in \varphi\inv( B_e(\varphi(x),r_0) } }{\displaystyle \inf \set{ \sqrt{\det g} (y):\, y\in \varphi\inv( B_e(\varphi(x),r_0)} }>0.
\end{equation*}
Now we use \eqref{eqn: nesting of metric balls} and \eqref{eqn: volume growth constant for nested balls} to estimate
    \begin{gather*}
         C_\varrho 
        \frac{ \vol \Bigl(E\cap \varphi\inv( B_e(\varphi(x),r)) \Bigr)}
        {\vol \Bigl(\varphi\inv( B_e(\varphi(x),r) ) \Bigr)}
        \leq  C_\varrho 
        \frac{ \vol \Bigl( E\cap B(x,Lr) \Bigr)}
        {\vol \Bigl( B(x,L\inv r)\Bigr)}   \\
        \leq L^{2m} C_v C_\varrho 
        \frac{ \vol \Bigl(E\cap  B(x,Lr) \Bigr)}
        {\vol \Bigl(  B(x,Lr) \Bigr)}
        =  L^{2m} C_v C_\varrho\, \Theta_E(x,Lr).
    \end{gather*}
    Taking the limit $r\to 0$, it follows that $\Theta_{\varphi(E)}(\varphi(x)) \geq \mathrm{const.}\, \Theta_E(x)$ where $\Theta_E(x)$ is the density of $E$ at $x$ with respect to the Riemannian volume measure. Therefore, $\Theta_E(x)\neq 0 $ implies $\Theta_{\varphi(E)}(\varphi
    (x))\neq 0$. Using that $\Theta_{E}(x) = 1$ if and only if $\Theta_{M\setminus E}(x) = 0$ it follows similarly that $\Theta_E (x)\neq 1 $ implies $\Theta_{\varphi(E)}(\varphi
    (x))\neq 1$. Recalling that the essential boundary is independent of the choice of volume measure, we conclude: if $x\in \partial^{\ess}E$, then $\varphi(x)\in \partial^{\ess}(\varphi(E))\cap\varphi(U)$. The proof for the reverse implication works analogously.
\end{proof}
In the proof of Lemma~\ref{lem: essential boundary in charts} we showed that for a given  point $x\in M$, a set  $E\subset M$  has density $t=0$ (respectively, $t=1)$ at $x$ if and only if for some/every chart $\varphi$ around $x$ one has that $\varphi(E)$
 has density $t=0$ (respectively, $t=1$) at $\varphi(x)$. This is generally false for $t\in (0,1)$ as the following example shows for $t=1/2$.

\begin{example}\label{ex: density not preserved under linear transformation}
    Consider the density of the set
$E= \set{|x_2| > |x_1|}\subset \R^2$
at $x = 0$. Then for arbitrary $r>0$,  
\begin{equation*}
   % \lim_{r\downarrow 0} 
   \frac{\lambda(E\cap B_r(0))}{\lambda(B_r(0))} 
   % = \frac{\ \lambda(B_r(0))}{\lambda(B_r(0))} 
   = \frac{1}{2} 
\end{equation*}
is the  Euclidean density of $E$ at $x=0$.
Let $A$ be the diagonal matrix with entries $A_{11}=1$ and $A_{22}=2$. As Figure \ref{fig:density under linear transformation} illustrates, we get for every $r>0$
\begin{equation*}
   % \lim_{r\downarrow 0} 
   \frac{\lambda(A(E)\cap B_r(0))}{\lambda(B_r(0))} 
   < \frac{r^2}{2\pi r^2}  = \frac{1}{2\pi} .
\end{equation*}
\end{example}

\begin{figure}[H]
    \centering
     \caption{The density of $E$ at the origin is not preserved under the linear transformation $A$.}
    \begin{multicols}{2}
        \begin{subfigure}[b]{\linewidth}
            \centering
          \begin{tikzpicture}[scale=.8,>=stealth]
    % --- Parameters ---
    \pgfmathsetmacro{\rnum}{1.5}       % r
    
    \newlength{\xradius} 
    \setlength{\xradius}{\rnum cm}   % y-semi-axis
    
    \pgfmathsetmacro{\Axnum}{1.8*\rnum}
    \pgfmathsetmacro{\Aynum}{1.8*\rnum}
    \newlength{\ellAx}  
    \setlength{\ellAx}{\Axnum cm}
    \newlength{\ellAy}  
    \setlength{\ellAy}{\Aynum cm}
    
    % --- Ball of radius r ---
    \fill[myorange!35] (0,0) circle [radius=\xradius];   % fill the circle
    
    % --- Light green E (|x2| >= |x1|) ---
    \fill[myblue!35] 
        (-\ellAx,\ellAx) -- (-\ellAx,\ellAx) -- (0,0) -- (\ellAx,\ellAx) -- (\ellAx,\ellAx) -- cycle;
    \fill[myblue!35] 
        (-\ellAx,-\ellAx) -- (-\ellAx,-\ellAx) -- (0,0) -- (\ellAx,-\ellAx) -- (\ellAx,-\ellAx) -- cycle;

    % --- E ∩ circle (myorange) ---
    \begin{scope}
        \clip (0,0) circle [radius=\xradius];
        \fill[myorange,opacity=0.8] (-\ellAx,\ellAx) -- (0,0) -- (\ellAx,\ellAx) -- cycle;
        \fill[myorange,opacity=0.8] (-\ellAx,-\ellAx) -- (0,0) -- (\ellAx,-\ellAx) -- cycle;
    \end{scope}
    
    % --- Axes ---
    \draw[->] (-\ellAx,0) -- (\ellAx,0) node[right] {$x_1$};
    \draw[->] (0,-\ellAx) -- (0,\ellAx) node[above] {$x_2$} node[pos = 0.9, right, color = myblue] {$E$};
    
    % --- Intersection points of circle ---
    \node[ left] at (0,\xradius) {$r$};
    \end{tikzpicture}

             \caption{The volume of  $E\cap B_r(0)$ equals half the volume of $B_r(0)$.}
            \label{fig:subfig1}
        \end{subfigure}

        \begin{subfigure}[b]{\linewidth}
        \centering
        
        \begin{tikzpicture}[scale=.8,>=stealth]
        
        % --- Parameters ---
        \pgfmathsetmacro{\rnum}{1.5}       % r
        \pgfmathsetmacro{\twornum}{3}
        % \newlength{\xradius} 
        \setlength{\xradius}{\rnum cm}   % y-semi-axis
        \newlength{\yradius} \setlength{\yradius}{\twornum cm}
        \pgfmathsetmacro{\Axnum}{1.8*\rnum}
        \pgfmathsetmacro{\Aynum}{\rnum}
        % \newlength{\ellAx}  
        \setlength{\ellAx}{\Axnum cm}
        
        % --- Ball of radius r ---
        \fill[myorange!35] (0,0) circle [radius=\xradius];   % fill the circle
        
        % --- Light green E (|x2| >= |x1|) ---
        \fill[myblue!35] 
            (-0.5*\ellAx,\ellAx) -- (-0.5*\ellAx,\ellAx) -- (0,0) -- (0.5*\ellAx,\ellAx) -- (0.5*\ellAx,\ellAx) -- cycle;
        \fill[myblue!35] 
            (-0.5*\ellAx,-\ellAx) -- (-0.5*\ellAx,-\ellAx) -- (0,0) -- (0.5*\ellAx,-\ellAx) -- (0.5*\ellAx,-\ellAx) -- cycle;
        
        % --- Dark green E ∩ {|x2| <= r} ---
        \fill[myblue!60] (-0.5*\xradius,\xradius) -- (0,0) -- (0.5*\xradius,\xradius) -- cycle;
        \fill[myblue!60] (-0.5*\xradius,-\xradius) -- (0,0) -- (0.5*\xradius,-\xradius) -- cycle;
        
        % --- E ∩ circle (myorange) ---
        \begin{scope}
            \clip (0,0) circle [radius=\xradius];
            \fill[myorange,opacity=0.8] (-0.5*\ellAx,\ellAx) -- (0,0) -- (0.5*\ellAx,\ellAx) -- cycle;
            \fill[myorange,opacity=0.8] (-0.5*\ellAx,-\ellAx) -- (0,0) -- (0.5*\ellAx,-\ellAx) -- cycle;
        \end{scope}
        
        % --- Axes ---
        \draw[->] (-\ellAx,0) -- (\ellAx,0) node[right] {$x_1$};
        \draw[->] (0,-\ellAx) -- (0,\ellAx) node[above] {$x_2$} node[pos = 0.9, right, color = myblue] {$A(E)$};
        
        % --- Intersection points of circle ---
        % \node[below ] at (\xradius,0) {$r$};
        \node[ left] at (0,\xradius) {$r$};
        \end{tikzpicture}
        
            \caption{The volume of $A(E)\cap  B_r(0)$ is smaller than $r^2$, the volume  of the dark colored area.}
            \label{fig:subfig2}
        \end{subfigure}
    \end{multicols}
    \label{fig:density under linear transformation}
\end{figure}
We remark that nevertheless, applying the transformation formula and exploiting smoothness of the Jacobian determinant, one finds    that
\begin{equation*}
   \lim_{r\downarrow 0} \frac{\lambda(f(E\cap B_r(x))}{\lambda(f(B_r(x)))} = \lim_{r\downarrow 0} \frac{\lambda(E\cap B_r(x)}{\lambda(  B_r(x))}
\end{equation*}
for every diffeomorphism $f:M\to \R^m$.
The Riemannian nature of the density of $E$ at $x$ is reflected in the following result on the density of a point in  normal coordinates. The proof relies on the following estimate for the doubling constant from~\cite[inequality~(2)]{Sal92_UniformlyEllipticOperators}:
If $r$ is sufficiently small such that $B(x,2r)$ is  relatively compact  and if $\mathrm{Ric}_{|B(x,2r)}\geq -K$ for some  $K \geq 0$, then
\begin{equation}\label{eqn: doubling estimate}
    \vol  \big(B(x,r)\big) \leq \vol  \big(B(x,s)\big) \left( \frac{r}{s} \right)^{m} \exp\big( (m-1)K r \big)
  \quad \text{for all } x \in M \text{ and } 0<s < r .
\end{equation}

\begin{corollary}\label{cor: density of of sets in normal coordinates}
    The density of $E\subset M$ at $x\in M$ coincides with the density of its image under a normal chart centered in $x$.
\end{corollary}
\begin{proof}
    Consider the constants $L(r_0), C_\varrho$ and $C_v$ from the proof of Lemma~\ref{lem: essential boundary in charts}. From the construction in \cite[Lemma 3.24]{Gri12_HeatKernelAnalysis} it follows that we can choose 
    \begin{align*}
        L(r_0) &= \sup_{y\in \varphi\inv(B_e(\varphi(x),r_0))}\max\set{ \norm[]{[g_{ij}](y)},1/ \norm[]{[g_{ij}]\inv(y)} } \\
        &= \sup_{y\in \varphi\inv(B_e(\varphi(x),r_0))}\max\set{ \lambda_{\max}(y), 1/\lambda_{\min}(y) },
    \end{align*}
    where $\lambda_{\max}(y), \lambda_{\min}(y)$ are the largest and smallest eigenvalues of the matrix $[g_{ij}](y)$.
    By \eqref{eqn: doubling estimate}, for sufficiently small $r_0$ and a local lower Ricci curvature bound $K$ we have
    \begin{equation*}
        C_v(r_0) = \left(\tfrac{1}{L(r_0)} \right)^{2m}\exp((m-1)K r_0).
    \end{equation*}
    Now suppose $\varphi$ is a normal chart in $x$ and  $L(r_0)r_0$ is smaller than the injectivity radius at $x$. Then 
    $$g_{ij}(y)=\delta_{ij}+\calO(\dist(x,y)) \tfor y\in B(x,L(r_0)r_0)$$
    (see e.g. \cite{BerGetVer04_HeatKernelsDirac}).
     Therefore, we get
    \begin{equation*}
        L(r_0)\to 1,\quad C_\varrho(r_0)\to 1 \qand C_v(r_0)\to 1 \quad\text{ as } r_0\to 0.
    \end{equation*}
    Thus, since for arbitrary $r< r_0$  we have
    \begin{equation*}
        \Theta_{\varphi(E)}(\varphi(x),r) 
        \leq  {L(r_0)}^{2m}\,C_v(r_0) \,C_\varrho(r_0)\, \Theta_E (x,L(r_0) r) ,
    \end{equation*}
    taking the limit $r_0\to0$ implies that $\Theta_{\varphi(E)}(\varphi(x)) \leq \Theta_E (x)  $.
   The other direction works analogously.
\end{proof}

With the machinery developed in the previous sections, the formulation of Federer's theorem~\cite[Theorem 3.61]{AmbFusPal00_FunctionsBoundedVariation}  translates verbatim to the Riemannian setting:
\begin{theorem}[Federer]\label{thm: Federer}
If $E$   has finite perimeter in $\Omega$, then 
$$
(\partial^* E \cap \Omega) \subset E^{(1/2)}\subset \partial^{\ess}E,
$$
and
$$
\mathcal{H}^{m-1}\big( \Omega \setminus (E^{(0)} \cup \partial^*E \cup E^{(1)}) \big) = 0.
$$
In particular, $E$ has density either $0$, $\tfrac{1}{2}$, or $1$ at 
$\mathcal{H}^{m-1}$-almost every $x \in \Omega$, and 
$\mathcal{H}^{m-1}$-almost every $x \in \partial^{\ess} E \cap \Omega$ belongs to $\partial^*E$.
\end{theorem}

\begin{proof}
    Let $x\in \partial^*E\cap\Omega$ and let $(U,\varphi)$ be a normal chart centered in $x$. Then by Corollary~\ref{cor: chart wise reduced boundary} we have $\varphi(x)\in \partial^*(\varphi(E))\cap \varphi(\Omega)\subset\R^m$, hence Federer's theorem in Euclidean space implies that $\varphi(x)\in \varphi(E)^{(1/2)}$. Since by Corollary~\ref{cor: density of of sets in normal coordinates} the density of a set at a point is invariant under a normal chart centered in that point, it follows that $x\in E^{(1/2)}$ and the first inclusion is proved. The second inclusion is an immediate consequence of the definition of the essential boundary. Finally, using the Euclidean result, for an arbitrary chart $(U,\varphi)$ we have 
    \begin{equation*}
        \mathcal{H}^{m-1}_{e}\big(\varphi( \Omega) \setminus (\varphi(E)^{(0)} \cup \partial^*(\varphi(E) )\cup \varphi(E)^{(1)}) \big) = 0
    \end{equation*}
    and since nullsets are invariant under Hausdorff measures with respect to a change of the Riemannian metric, using the same arguments as before and the pushforward formula for Hausdorff measures, we get
    \begin{equation*}
        0
        = \mathcal{H}^{m-1}_{\tilde{g}}\big(\varphi( \Omega) \setminus (\varphi(E)^{(0)} \cup \partial^*(\varphi(E) )\cup \varphi(E)^{(1)}) \big) 
        = \mathcal{H}^{m-1}_g\big( (\Omega\cap U )\setminus (E^{(0)} \cup \partial^*E \cup E^{(1)}) \big) 
    \end{equation*}    
    and the full statement follows from a straightforward covering argument.
\end{proof}

As we have observed earlier, the density of a measurable set $E$ at $x$ equals the density of its image  $\varphi(E)$ under a chart $\varphi$ at $\varphi(x)$ for  almost every $x\in M$. Using Federer's theorem and our previous results regarding pointwise properties of the reduced and essential boundary, we can refine this statement as follows:
\begin{corollary} 
    If $E$ has finite perimeter in the chart domain of an arbitrary local chart $(U,\varphi)$   around $x$, then
    \begin{equation*}
        \calH^{m-1}_g \big(E^{(1/2)} \Delta (\varphi\inv(\varphi(E)^{(1/2)}))\big) = 0.
    \end{equation*}
    In other words, $\varphi(E)$ has density $1/2$ at $\varphi(x)$  at $\calH^{m-1}_g$-almost every $x\in E^{(1/2)}$.  
\end{corollary}

We compare this result to the setting of a complete metric measure space $(\calX,d,\mu)$ satisfying a doubling property and a $1$-Poincar\'e inequality (PI space). Since general PI spaces do not have a fixed dimension, one uses the following construction:
    Let $B_r $ denote a metric ball  of radius $r $ in $\calX$  and let $\calS = \calS_{\mu,d}$ be the \textit{generalized codimension-$1$ spherical Hausdorff} measure given  by
    \begin{equation*}
     \calS_{\mu,d}(A) := \lim_{r  \to 0} \inf \biggl\{   \sum_j \frac{\mu(B_{r _j})}{r _j} :\, A\subset \bigcup_j B_{r _j},\, r _j\leq r   \biggr\} .
     \end{equation*} 
     The following result is due to~\cite[Theorems 5.3]{Amb02_FinePropertiesSets} and~\cite[Theorem 4.6]{AmbMirPal04_SpecialFunctionsBounded}:
    If $E\subset M$ is a set of finite perimeter, then there exist $0<c_1<c_2<\infty$ and a Borel measurable density function $\theta:\partial^{\ess}E\to [c_1,c_2]$ such that
    \begin{equation}\label{eqn: federer on PI spaces}
        \TV{\mathbf{1}_E}(A) = \int_{\partial^{\ess}E\cap A} \theta\d \calS_{\mu,g} \tfor A\in \scrB{(M)}.
    \end{equation}
    For the Riemannian distance function  and the volume measure  on $M$, the measure $\calS$ is equivalent to the classical Hausdorff measure $\calH^{m-1}$. While the existence  of the density function in \eqref{eqn: federer on PI spaces} is obtained in a nonconstructive way, Theorem~\ref{thm: Federer} refines this result to the explicit density $\theta\equiv 1$ with respect to $\calH^{m-1}$ and we obtain
    \begin{equation*}
        \TV{\mathbf{1}_E} = \calH^{m-1}\llcorner\partial^{\ess} E
    \end{equation*}
    in analogy to the Euclidean case. 
    Notice, moreover, that $\Omega$ has finite perimeter if and only if  $\calH^{m-1}(\partial^{\ess} \Omega)<\infty$. This result, too, is due to Federer in the Euclidean case and was proved for metric measure  spaces in \cite{Lah20_FederersCharacterizationSets}. 

%%%%%%%%%%%%%%%%%%%%%%%%%%%%%%%%%%%%%%%%%%%%%%%%%%%%%%%%%%%
\section{Theory for  boundary value problems and a mixed boundary capillarity problem} 
\label{sec: BVP and capillarity}
%%%%%%%%%%%%%%%%%%%%%%%%%%%%%%%%%%%%%%%%%%%%%%%%%%%%%%%%%%%%

\subsection{Extensions, boundary traces and the Gauss--Green formula}\label{sec: gauss green}
 We begin by collecting the essential ingredients for boundary value problems from the metric measure space literature and adapt them to the Riemannian setting with the classical codimension-$1$ Hausdorff measures on~$M$.
We recall the definitions of extension domains from~\cite{CapKoiRaj24_SobolevBVPerimeter}.
\begin{definition} \label{def: s-BV extension domain}
A domain $\Omega\subset M$ is called a \emph{BV-extension domain} if there exists a   map  $\calE:BV(\Omega) \to BV(M)$ with $\calE u|_{\Omega} = u$ and a constant $C>0$ such that $\norm[BV(M)]{\calE u}\leq C\norm[BV(\Omega)]{u}$ for    $u \in BV(\Omega)$. If in addition, $\calE$ can be chosen  such that  the total variation satisfies $|\Var{(\calE u)}|(\partial \Omega) = 0$, then $\Omega$ is called \textit{strong BV-extension domain}.
\end{definition}

\begin{remark}\label{rmk: properties of extension domains}
    \begin{enumerate}[leftmargin = *,   labelindent = 0em,]
        \item By \cite[Theorem 3.1]{Lah15_ExtensionsTracesFunctions} a sufficient condition for a bounded domain $\Omega$ in a complete PI space to be a strong $BV$-extension domain is that $\Omega$ be \textit{uniform} in the following sense: there exists a constant $C\geq 1$ such that  any two points $x,y\in \Omega$ can be joined by a rectifiable curve $\gamma$ whose length $\ell(\gamma)$ satisfies
    \begin{equation*}
        \ell(\gamma)\leq C \dist(x,y) \qand \dist(\gamma(t),\partial\Omega)\geq C\inv \min\set{t,\ell(\gamma)-t}, \quad t\in[0,\ell(\gamma)].
    \end{equation*}
    By classical results, this is satisfied by bounded Lipschitz domains in $\R^m$ and carries over to bounded Lipschitz domains in $M$.  
    More recent results  show that in complete PI spaces,  an open and bounded domain $\Omega$ is a strong $BV$-extension domain if and only if it is a $W^{1,1}$-extension domain (cf. \cite[Theorems 1.5, 1.6]{CapKoiRaj24_SobolevBVPerimeter} and the prior Remark that the boundary of Sobolev extension domains in PI spaces have zero measure). 
    \item \textit{Compactness in $BV(\Omega)$.}  Assume that $\Omega\subset M$ is a  BV-extension domain. If $(u_n)\subset BV(\Omega)$ is uniformly bounded in $BV(\Omega)$, then there exists $u\in BV(\Omega)$ such that along a subsequence, $u_n \to  u$ in $L^1(\Omega)$ and $Du_n\to  Du$ weakly$^*$. Indeed, if $\calE$ denotes the extension operator, then $\norm[BV(M)]{\calE u_n} \leq \norm[BV(\Omega)]{u_n}$ is uniformly bounded. Now one can apply Step 1 from the proof of the local compactness result  Corollary~\ref{Cor: compactness in BV_loc} to the set $\Omega\subset\subset M$.
    \item \textit{Smooth approximation up to the boundary.}
    If $\Omega$ is a strong BV-extension domain, then every $u\in BV(\Omega)$ can be approximated strictly by a sequence in $C^{\infty}(\bar{\Omega})$: Let $u_n\in C^\infty(M)\cap W^{1,1}(M)$ be a strict approximation of the extended function $\calE u\in BV(M)$. Since we have $\TV{(\calE u)}(\partial\Omega) = 0$, the Portmanteau theorem~\cite[4.10]{Els18_MassUndIntegrationstheorie} implies that $\TV{u_n}(\Omega)$ converges to $\TV{(\calE u)}(\Omega) = \TV{u}(\Omega)$, and clearly, $(u_n)$  converges to $u$ in $L^1(\Omega)$. Hence the sequence lies in $C^\infty(\bar{\Omega})$ and converges to $u$ strictly in~$BV(\Omega)$.
    \end{enumerate}
\end{remark}

\begin{definition} \label{def: BV trace operator}
   A domain $\Omega\subset M$ is said to admit a    \textit{($BV$-)trace operator} if for every $u\in BV(\Omega)$ there exists a measurable function $\trace[\Omega](u):\partial\Omega\to\R$ such that for $\calH^{m-1}$-almost every $x\in \partial\Omega$ one has
   \begin{equation*}
        \lim_{r\downarrow 0} \frac{1}{r^m}\int_{\Omega\cap B(x,r)} |u(y)-\trace[\Omega] (u)(x)| \d\vol(y) =  0.
    \end{equation*}
\end{definition}

\begin{theorem}
% [Bounded trace operator]
\label{thm: Bounded trace operator}
A  strong $BV$-extension domain $\Omega$ 
admits a $BV$-trace operator. If in addition, $\Omega$ admits a constant $C>0$ such that for $\calH^{m-1}$-almost every $x\in \partial^{*}\Omega$ one has
\begin{equation}\label{eqn: condition for L^1 trace bound}
    \calH^{m-1}(\partial^{*}\Omega\cap B(x,r)) \leq C  \frac{\vol(B(x,r))}{r} \tforall r\in (0,2\diam(\Omega)),
\end{equation}
then the trace operator $u\mapsto\trace[\Omega](u)$ is a bounded linear operator from $BV(\Omega)$ to $L^1(\partial^*\Omega;\calH^{m-1})$. Similarly, $\trace[\Omega]: BV(\Omega)\to L^1(\partial\Omega;\calH^{m-1})$ is bounded if condition \eqref{eqn: condition for L^1 trace bound} is satisfied for the full boundary $\partial\Omega$.

\end{theorem}
\begin{proof}
    By \cite[Proposition 4.2]{CapKoiLucRaj25_ClosedBVextension$W^11$extension}, a strong $BV$-extension domain satisfies a \textit{strong measure density condition} in the sense that there exists a constant $c>0$ such that
\begin{equation*}
    \Theta_\Omega(x,r)=\frac{\vol(\Omega\cap B(x,r))}{ \vol(B(x,r))} \geq c \tforall x \in \overline{M\setminus \Omega^{(0)}},\, r\in(0,\diam(\Omega)).
\end{equation*}
 Since $\Omega$ is open, we have $\smash{ \Omega \subset \Omega^{(1)}\subset M\setminus\Omega^{(0)} }$, hence~$\smash{ \overline{\Omega}\subset \overline{M\setminus \Omega^{(0)}} }$. Using this condition,
the existence statement is proved in \cite[Theorem 5.7]{HakKinLahLeh16_RelaxationIntegralRepresentation}. The $L^1$-estimate is proved for $\partial^{\ess}\Omega$
in \cite[Theorem 4.13]{Lah15_ExtensionsTracesFunctions}  and by Theorem~\ref{thm: Federer} we may replace the essential boundary $\partial^{\ess}\Omega$ with $\partial^*\Omega$.
The $L^1$-estimate for  $\partial\Omega$ can be found in \cite[Theorem 5.5]{LahSha18_TraceTheoremsFunctions}.
\end{proof}

\begin{remark} \label{rmk: properties of for bounded trace}
    \begin{enumerate}[leftmargin = *,   labelindent = 0em,]
        \item  Condition \eqref{eqn: condition for L^1 trace bound} is necessary and sufficient for the $L^1$-bound but is \textit{not} satisfied by arbitrary uniform domains $\Omega$, even if $\calH^{m-1}(\partial\Omega)<\infty$ (see \cite[Example 5.3]{LahSha18_TraceTheoremsFunctions}). In particular, it is not satisfied for arbitrary strong $BV$-extension domains. However, the condition does hold for   bounded Lipschitz domains where one has $|D\mathbf{1}_\Omega| = \calH^{m-1}\llcorner(\partial\Omega)<\infty$ (see the proof of \cite[Theorem 3.87]{AmbFusPal00_FunctionsBoundedVariation} for the Euclidean case). If the domain is relatively compact, then either of the two $L^1$-bounds implies that $\calH^{m-1}(\partial^{*}\Omega)<\infty$, hence $\Omega$ must have  finite perimeter. For further necessary and sufficient conditions for the existence and $L^1$-boundedness of trace operators see~\cite{Lah15_ExtensionsTracesFunctions,LahSha18_TraceTheoremsFunctions}. 
        \item If $\Omega\subset M$ admits a bounded linear trace operator $\trace[\Omega]: BV(\Omega) \to L^1(\partial\Omega; \calH^{m-1})$, then the trace operator is continuous with respect to the weaker topology of strict convergence on~$BV(\Omega)$. That is, if a sequence $u_n \to u$ strictly in $BV(\Omega)$, then $\trace[\Omega](u_n) \to \trace[\Omega](u)$ in $L^1(\partial\Omega;\calH^{m-1})$ by 
        \cite[Proposition 7.1]{KorLahLiSha19_NotionsDirichletProblem}. The proof of this theorem applies verbatim if one replaces $\partial\Omega$ by $\partial^{*}\Omega$ and the analogous statement remains true if the trace operator is only continuous from $BV(\Omega)$ to $L^1(\partial^{*}\Omega;\calH^{m-1})$.
        \item For an arbitrary domain $\Omega\subset M$ and $u\in BV(\Omega)$ we define the \textit{zero extension of $u$ to $M$} as the function $\bar{u}\in L^1(M)$   given by $\bar{u}_{|\Omega}=u$ and $\bar{u}_{|M\setminus\Omega}=0$. Then we have
        \begin{equation*}
            BV(\bar{\Omega}):= \set{ u\in BV(\Omega):\, \bar{u}\in BV(M)} = \set{ u\in BV(M): {u}_{|M\setminus\Omega}=0}.
        \end{equation*}
        If $\Omega$ is a bounded $BV$-extension domain with a trace operator that admits an $L^1$-bound in the sense of Theorem~\ref{thm: Bounded trace operator}, then one has $\bar{u}\in BV(M)$ for all $u\in BV(\Omega)$ (cf. \cite{Lah15_ExtensionsTracesFunctions,LahSha18_TraceTheoremsFunctions}), hence $BV(\Omega)= BV(\bar{\Omega})$.  
    \end{enumerate}
\end{remark}

\begin{corollary}[Gauss--Green]\label{cor: Gauss--Green}
    Let $\Omega$ be a finite perimeter strong $BV$-extension domain that admits a bounded linear trace operator  $\trace[\Omega]:BV(\Omega)\to L^1(\partial^{*}\Omega;\calH^{m-1})$. Then for $u\in BV(\Omega)$ and all $X\in \Gamma_{C_c^\infty}(M;TM)$ one has
    \begin{equation*}
       - \int_\Omega u\ddiv(X)\d \vol = \int_\Omega X\cdot\d Du + \int_{\partial^{*}\Omega} \trace[\Omega] (u) (X,\normal{\Omega}) \d\calH^{m-1}.
    \end{equation*}
\end{corollary}

\begin{proof}
Since $\Omega$ is a strong $BV$-extension domain, there exist $u_n\in C^\infty(\bar{\Omega})$ converging to $u$ strictly in $BV(\Omega)$ by Remark \ref{rmk: properties of extension domains}. Since $\Omega$ has finite perimeter, we obtain
\begin{equation}\label{eqn: gauss thm in BV}
    -\int_{\Omega} \ddiv(\phi X)\d\vol = \int_{\partial\Omega} (\phi X, \normal{\Omega})\d\TV{\mathbf{1}_\Omega}
\end{equation}
for arbitrary $\phi\in C^\infty(M)$ and $X\in \Gamma_{C_c^\infty}(M;TM)$. 
Using the chain rule for $\ddiv(u_n X)$ and applying \eqref{eqn: gauss thm in BV} with $\phi=u_n$ yields
\begin{equation*}
\begin{aligned}
     - \int_\Omega u_n\ddiv(X)\d \vol 
    &=   \int_\Omega du_n(X)\d\vol -\int_\Omega \ddiv(u_nX)\d\vol \\
    &=  \int_\Omega X\cdot\d Du_n + \int_{\partial\Omega} u_n (X,\normal{\Omega}) \d\TV{\mathbf{1}_\Omega}.
\end{aligned}
\end{equation*}
By strict convergence, one has $u_n \to u$ in $L^1(\Omega)$ and $Du_n\to Du$  weakly (cf.\ Corollary~\ref{cor: weak$^*$ and weak convergence of (total) variation measure}). Moreover, continuity of the trace operator with respect to strict convergence (Remark \ref{rmk: properties of for bounded trace}) implies that ${u_n}_{|\partial\Omega} = \trace[\Omega](u_n) \to \trace[\Omega] (u)$ in $L^1(\partial^{*} \Omega;\calH^{m-1})$.  Thus, taking limits on both sides and using $\TV{\mathbf{1}_\Omega} = \calH^{m-1}\llcorner\partial^{*}\Omega$ on $M$, we obtain the claimed formula.
\end{proof}

Since we will repeatedly rely on the validity of Corollary~\ref{cor: Gauss--Green}, we introduce the following convention, which will later be assumed from time to time:

\begin{assumption} \label{ass: gauss green domain}
   The domain $\Omega\subset M$ has finite perimeter, admits a strong $BV$-extension   operator $\calE:BV(\Omega)\to BV(M)$ and   a bounded linear trace operator $\trace[\Omega]: BV(\Omega)\to  L^1(\partial^{*}\Omega;\calH^{m-1})$.
\end{assumption}

\subsection{Strict interior approximation of finite perimeter sets for mixed boundary problems}\label{sec: mixed bvp and strict approximation}

Throughout this section let $\Omega\subset M$ be a domain  and let $\gamma\subset\partial \Omega$ be a  measurable subset with $\calH^{m-1}(\gamma^*)\neq 0$  for  $\gamma^*:=\gamma\cap\partial^*\Omega$. The set $\gamma$  will play the role of a designated Dirichlet boundary.  On the level of finite perimeter sets, $\calH^{m-1}(\partial^* E\cap\gamma)=0$ reflects a Dirichlet boundary condition. 
If $\Omega$ admits a trace operator, then we define by
\begin{equation*}
    BV_{\gamma}(\Omega):= \set{ u\in BV(\Omega):\ \trace[\Omega] (u) = 0,\, \calH^{m-1}\text{-a.e. on }\gamma}
\end{equation*}
the space of \textit{$BV$ functions with zero boundary conditions on $\gamma$}. As before, $\bar{u}$ denotes the zero extension of $u\in BV(\Omega)$.
Note that under Assumption~\ref{ass: gauss green domain},   one has
\begin{equation*}
     {P(E;\Omega\cup \gamma)} 
     % = {P(E;\Omega\cup \gamma^\ast)} 
     = {P(E;\Omega)+\calH^{m-1}(\partial^*E\cap\gamma)}
\end{equation*}
and
\begin{equation*}
    {|D\bar{u}|(\Omega\cup \gamma)} 
    % = {|D\bar{u}|(\Omega\cup \gamma^\ast)} 
    =  {|Du|(\Omega)+ \int_{\gamma^* }\trace[\Omega](u)\d\calH^{m-1}}.
\end{equation*}

\begin{lemma}\label{lem: characterization of mixed boundary TV} 
Assume \ref{ass: gauss green domain} and that $\gamma\subset\partial\Omega$ is relatively open. 
Then for all $u\in BV(\Omega)$ one has
    $$
        \TV{\bar{u}}(\Omega\cup\gamma) = \sup \set{\int_\Omega u \ddiv(X) \d \vol:\, X\in \Gamma_{C_c^\infty}(\Omega\cup\gamma;TM), \, |X|\leq 1  },
    $$
where $\Gamma_{C_c^\infty}(\Omega\cup\gamma;TM)$ is defined as in Appendix \ref{appendix}.
\end{lemma}

\begin{proof}
    Since $\gamma\subset\partial\Omega$ is relatively open, there exists an open set $U\subset M$ such that $\bar{\Omega}\cap U = \Omega\cup \gamma$.  Now, since $\TV{\bar{u}}$ is concentrated on $\bar{\Omega}$, it follows that 
    \begin{equation} \label{eqn: tv sup for locally compact sets}
        \TV{\bar{u}}(\Omega\cup\gamma) 
    = \TV{\bar{u}}(U)
     = \sup \set{ \Var{\bar{u}}[X]:\, X\in \Gamma_{C_c}(U,TU),\, |X|\leq1}
    .
    \end{equation}
    Note that $\bar{\Omega}\cap U = \Omega\cup \gamma$ is a locally compact Hausdorff space. By Lemma \ref{lem: smooth and cts functions on LCH spaces} the space $C_c^\infty(\Omega\cup\gamma)$ is dense in $C_c(\Omega\cup\gamma)$ with respect to uniform convergence and every function in $ C_c(\Omega\cup\gamma)$ can be extended to a function in $C_c(U)\subset C_0(U)$. In particular, using a partition of unity argument, one has $\Gamma_{C_c^\infty}(\Omega\cup\gamma;TM) \subset \Gamma_{C_c}(U,TU)$.  For every $X\in \Gamma_{C_c^\infty}(\Omega\cup\gamma;TM)$, the Gauss--Green formula above shows that
    \begin{equation*}
        -\int_\Omega u \ddiv(X) \d \vol =\Var{\bar{u}}[X],
    \end{equation*}
    which implies that 
    \begin{equation*}
           \sup \set{\int_\Omega u \ddiv(X) \d \vol:\, X\in \Gamma_{C_c^\infty}(\Omega\cup\gamma;TM), \, |X|\leq 1  }
       %     \\
       % \leq \sup \set{\Var{\bar{u}}[X]:\, X  \in\Gamma_{C_c^\infty}(U,TU):\, |X|\leq1}
        = \TV{\bar{u}}(\Omega\cup\gamma)
       .
    \end{equation*}

    For the other inequality, take $X\in \Gamma_{C_c^\infty}(U,TU)$. Then the support of the restriction $X_{|\bar{\Omega}\cap U}$ is compact since it is given by $\sppt(X)\cap\bar{\Omega}$, and therefore, 
    $$X_{|\bar{\Omega}\cap U}\in \Gamma_{C_c^\infty}(\bar{\Omega}\cap U,TM) = \Gamma_{C_c^\infty}(\Omega\cup\gamma,TM).$$
    Finally, we use the Gauss--Green formula once more to deduce that $\Var{\bar{u}}[X] = \Var{\bar{u}}(X_{|\Omega\cup\gamma})$, and conclude using \eqref{eqn: tv sup for locally compact sets} that
    \begin{align*}
        \TV{\bar{u}}(\Omega\cup\gamma) &= \sup\set{\Var{\bar{u}}[X]:\, X\in \Gamma_{C_c^\infty}(U,TU):\, |X|\leq1} \\
        &\leq \sup\set{\Var{\bar{u}}[X]:\, X\in \Gamma_{C_c^\infty}(\Omega\cup\gamma;TM), \, |X|\leq 1 }.
    \end{align*}
\end{proof}

\begin{remark}\label{rmk: compactness and lsc for mixed boundaries}
Let $\gamma$ be relatively open in $\partial\Omega$.

\begin{enumerate}[leftmargin = *, 
 labelindent = 0em, 
% itemsep=6pt
]
    \item   \textit{Compactness in $BV(\Omega\cup\gamma)$.} Let $\Omega$ be a  $BV$-extension domain.  If $(u_n)\subset BV(\Omega)$ is a sequence such that $\norm[1,\Omega]{u_n} + \TV{\bar{u}_n}(\Omega\cup\gamma) $ is uniformly bounded, then there exists $u\in BV(\Omega)$ such that along some subsequence, $u_n\to  u$ in $L^1(\Omega)$ and $\Var{\bar{u}_n} \llcorner (\Omega\cup\gamma) \to  \Var{\bar{u}} \llcorner (\Omega\cup\gamma) $ weakly$^*$ in the sense of generalized vector measures on $\Omega\cup\gamma$. In other words, $D\bar{u}_n[X]\to D\bar{u}[X]$ for all vector fields $X$ with coefficients in $C_0(\Omega\cup\gamma )$.
    This can be seen as follows: Since $\gamma\subset \partial\Omega$ is relatively open, there exists an open subset $U\subset M$ such that $U\cap\bar{\Omega} = \Omega \cup \gamma$ and since $\Omega$ is a strong $BV$-extension domain,  we can pick $U$ satisfying this property as well.  Then $\norm[1,U]{\bar{u}_n}+\TV{\bar{u}_n}(U) = \norm[1,\Omega]{\bar{u}_n}+\TV{\bar{u}_n}(\Omega\cup\gamma)$, so the  compactness result from Remark~\ref{rmk: properties of extension domains} for $\bar{u}_n\subset BV(M)$ on $U$ yields existence of a subsequence converging to some $\bar{u}$ in $L^1(U)$ and with $D\bar{u}_n\to D\bar{u}$ weakly$^*$ on $U$. Using the same density and extension argument as in the proof of Lemma~\ref{lem: characterization of mixed boundary TV}, it follows that $D\bar{u}_n[X]\to D\bar{u}[X]$ for all $X\in \Gamma_{C_c^\infty}(\Omega\cup\gamma;TM)$.
    \item \textit{Lower semicontinuity.}  Assume \ref{ass: gauss green domain}. Then Lemma~\ref{lem: characterization of mixed boundary TV} implies the following: If $u_n\to  u$ in $L^1(\Omega)$, or if $Du_n[X]\to Du[X]$ holds for all $C_0(\Omega\cup\gamma)$ vector fields $X$, then $\liminf_n |\Var{\bar{u}_n}|(U)\geq |\Var{\bar{u}}|(U) $ for all relatively open subsets $U\subset \Omega\cup\gamma$. 

    \item If  $\tilde{\gamma}\subset\partial\Omega$ is a measurable subset with $\calH^{m-1}(\tilde{\gamma}\Delta\gamma)=0$, then one has 
    $$|D\bar{u}|(\Omega\cup \gamma) 
    % = |Du|(\Omega) + \int_{\gamma\cap\partial^*\Omega}\trace[\Omega](u)\d\calH^{m-1} =  |Du|(\Omega) + \int_{\tilde{\gamma}\cap\partial^*\Omega}\trace[\Omega](u)\d\calH^{m-1} 
    = |D\bar{u}|(\Omega\cup \tilde{\gamma}),$$
    hence  the compactness and lower semicontinuity properties remain valid on $\Omega\cup\tilde{\gamma}$.
\end{enumerate}
\end{remark}
 
For future reference we record:
\begin{assumption}\label{ass: approximation domain}
     The set $\Omega\subset M$ is a  relatively compact set of finite perimeter with
    \begin{equation}\label{eqn: Omega boundary regularity}
    \calH^{m-1}\big(\Omega^{(1)}\cap \partial \Omega \big) = 0.
    \end{equation}
\end{assumption}

\begin{remark}\label{rmk: strict approximation from inside for full boundary}

\begin{enumerate}[leftmargin = *, 
 labelindent = 0em, 
% itemsep=6pt
]
    \item It was shown in \cite{GuiHuLi23_SmoothInteriorApproximation} that a relatively compact set of finite perimeter $\Omega$ in $\R^m$ satisfies \eqref{eqn: Omega boundary regularity}
    if and only if $\Omega$ can be approximated strictly  from \textit{inside} by smoothly bounded open sets, meaning that there exist sets $\smash{ \hat{\Omega}_n\subset\subset \Omega }$ with $C^\infty$-boundary such that $\smash{ \hat{\Omega}_n\to  \Omega }$  and $\smash{ \Per(\hat{\Omega}_n)\to  \Per(\Omega) }$ as~$n\to\infty$. 
    Employing the localization techniques from the previous section, the same is true on Riemannian manifolds.
    \item Recalling the definitions from Section~\ref{sec: measure-theoretic boundary and Federer's Theorem}, notice that \eqref{eqn: Omega boundary regularity} is equivalent to 
    $$
    \limsup_{r\downarrow 0} \Theta_\Omega(x,r) <1\quad\text{for $\calH^{m-1}$-almost every $x\in \partial\Omega$.}
    $$  
    Similarly, $\smash{ \calH^{m-1}(\Omega^{(0)}\cap \partial \Omega) = 0 }$ is equivalent to  
    $$
    \liminf_{r\downarrow 0} \Theta_\Omega(x,r) >0\quad\text{ for $\calH^{m-1}$-almost every $x\in \partial\Omega$.}
    $$
    In particular, this implies that if both \ref{ass: gauss green domain} and \ref{ass: approximation domain} are satisfied, then the strong measure density condition  on strong extension domains yields 
    \begin{equation*}
        \calH^{m-1}((\Omega^{(0)}\cup\Omega^{(1)})\cap \partial \Omega)=0. 
    \end{equation*}
\end{enumerate}
\end{remark}

The final approximation result will rely on the following property:

\begin{lemma}\label{lem: uniqueness of normal vector for subsets}  
    Let $E$ and $F $ be sets of locally finite perimeter in  $M$.
      If $x\in \partial^*E\cap\partial^*F$ and there exists some $\varepsilon>0$ such that essentially (that is,  up to nullsets) one has
      \begin{equation}\label{eqn: normal vector inclusion}
          (E\cap B_\varepsilon(x))\subset (F \cap B_\varepsilon(x)),
      \end{equation}
      then  $\normal{E}(x) = \normal{F}(x)$.
      \end{lemma}
\begin{proof}
    Assume that \eqref{eqn: normal vector inclusion} is satisfied, and without loss of generality, assume that $\varepsilon$ is small enough such that the exponential map is well-defined on the ball of radius $\varepsilon$ around $0$ in $T_xM$. 
    For $r>0$,  let $E_{x,r} $ be the blow-up of $E$ at $x\in M$ from Definition~\ref{def: blow-up on M}. If $x\in \partial^*E$, then by Theorem~\ref{thm: De~Giorgi} one has 
    $$E_{x,r} \overset{loc.}{\to} H^E(x) = \set{\xi\in T_xM:\ (\xi,\normal{E}(x))\geq 0 },$$
    and an analogous condition holds for $F_{x,r}$.

     Now let $K\subset T_xM$ be a compact set and let $R_K>0$ be the   radius of some ball centered at zero which contains $K$.  We set $r_K:= {\varepsilon}/{R_K}$. Using that $\exp_x$ is a radial isometry, by definition of the blow-up one gets that $E_{x,r}\cap K$ is essentially contained  in $F_{x,r}\cap K$ for all $r<r_K$. Then from $E_{x,r}\cap K \to H^E(x)\cap K$ and $F_{x,r}\cap K \to  H^F(x) \cap K$ it follows that $H^E(x)\cap K$ is essentially contained in   $  H^F(x) \cap K$. Since $K$ was arbitrary, this implies that the half spaces $H^E(x)$ and $H^F(x)$  coincide and therefore, $\normal{E}(x) = \normal{F}(x)$.
\end{proof}

We specify the assumptions on the Dirichlet boundary.

 \begin{assumption} \label{ass: gamma dirichlet boundary}
There exists a relatively open set $\hat{\gamma}\subset\partial\Omega$ with $\calH^{m-1}(\hat{\gamma}\Delta\gamma)=0$ as well as a set $A\subset \Omega$ with finite perimeter such that  
\begin{equation}\label{eqn: assumption for gamma}
\partial^*A\cap\partial^*\Omega = \gamma\cap\partial^*\Omega=: \gamma^*.
\end{equation}    
 \end{assumption}

\begin{example}\label{ex: decomposition of Lipschitz domains}
        If $\Omega$ is a   relatively compact Lipschitz domain and $\gamma\subset\partial\Omega$ is the closure of a relatively open  subset of $\partial\Omega$, then one can show that there exists a bounded Lipschitz domain $A\subset\Omega$ such  that $\partial A\cap \partial \Omega = \gamma$.
\end{example}

    \begin{theorem}\label{thm: strict Dirichlet approximation of finite perimeter sets}
    Assume \ref{ass: gauss green domain}, \ref{ass: approximation domain} and \ref{ass: gamma dirichlet boundary}
    and let $E\subset\Omega$ be a set of finite perimeter in $\Omega$. Then there exists a sequence of finite perimeter sets $E_n\subset\Omega$ with $\Per(E_n;\gamma)=0$ and $\Per(E_n; \partial\Omega\setminus\gamma) = \Per(E; \partial\Omega\setminus\gamma)$ for all $n\in\N$ such that
        \begin{equation}\label{eqn: Dirichlet approximation of sets}
            E_n\to  E, \qand \Per(E_n; \Omega) \to  \Per(E;\Omega\cup\gamma ).
        \end{equation}
        In other words, $E_n$ converges to $E$ strictly on $\Omega\cup\gamma$.
    \end{theorem}
    \begin{proof}
        Firstly,  notice that for every subset $E\subset \Omega$ with $\Per(E; \Omega)<\infty$ the trace theorem   implies that
        \begin{equation}\label{eqn: boundary decomposition for perimeter}
        \begin{aligned}
            \Per(E) &= \Per(E; \Omega) + \Per(E;\gamma) + \Per( E ; \partial\Omega\setminus\gamma)\\
            &= |D\mathbf{1}_E|(\Omega) + \calH^{m-1}(\partial^*E\cap\gamma) + \calH^{m-1}(\partial^*E\cap(\partial^*\Omega\setminus\gamma)).
            \end{aligned}
        \end{equation}
          In addition, for arbitrary sets $E,F \subset M$ of (locally) finite perimeter, it holds that both $E\cup F $ and $E\cap F $ are of (locally) finite perimeter and for every Borel measurable set  $U\subset M$ one has
        \begin{equation}\label{eqn: perimeter union and intersection inequality}
            \Per(E\cup F ;U) + \Per(E\cap F ; U) \leq \Per(E;U)+ \Per(F ;U).
        \end{equation}
        This follows from \cite[Theorem 16.3]{Mag12_SetsFinitePerimeter} in the Euclidean case and with the structure theory from the previous section at hand, the proof of this theorem applies verbatim to the Riemannian case. We divide the proof into two steps.\\
        
        \textit{Step 1: There exists a sequence of finite perimeter sets $\hat{E}_n\subset\Omega$ with $\Per(\hat{E}_n;\partial\Omega) = 0$ such that $\hat{E}_n\to E$ and $\Per(\hat{E}_n;\Omega) \to \Per(E;\bar{\Omega}) = \Per(E)$. }\\
        \textit{Proof of Step 1:} Let  $\hat{\Omega}_n$ be as in Remark~\ref{rmk: strict approximation from inside for full boundary} and set $\hat{E}_n:= E \cap\hat{\Omega}_n$. Then it is clear that $\hat{E}_n\to  E$. Thus, by lower semicontinuity of the perimeter  and since $\hat{E}_n\subset\subset\Omega$, it follows that
        $\liminf \Per(\hat{E}_n;\Omega) = \liminf \Per(\hat{E}_n) \geq \Per(E)$.
         It remains to show that $P(\hat{E}_n;\bar{\Omega}) =P(\hat{E}_n) $ is upper semicontinuous: Using \eqref{eqn: perimeter union and intersection inequality} we compute
        \begin{equation*}
            \begin{gathered}
            \limsup \Per(\hat{E}_n)  
            % = \limsup\Per(E\cap\hat{\Omega}_n) 
            \leq \limsup (-\Per(E\cup \hat{\Omega}_n)) + \Per(E) + \lim \Per(\hat{\Omega}_n)\\
            = - \liminf \Per(E\cup \hat{\Omega}_n)  + \Per(E) + \Per(\Omega)
            \leq  \Per(E),
            \end{gathered}
        \end{equation*}
        where the last inequality follows from  $E\cup\hat{\Omega}_n \to \Omega$ and lower semicontinuity of $\Per(E\cup\hat{\Omega}_n)$.\\

         \textit{Step 2: There exists a sequence $(E_n)$ satisfying the statement of the Lemma.}\\
        \textit{Proof of Step 2:} Let $A\subset\Omega$ be as in \ref{ass: gamma dirichlet boundary} and set $E_n:=( E\cap(\Omega\setminus A))\cup \hat{E}_n$ where $\hat{E}_n=E\cap\hat{\Omega}_n$ is as in Step 1. We first prove the following  identities:
        \begin{align}
            \Per(E\cap(\Omega\setminus A);\partial\Omega) = \Per(E;\partial\Omega\setminus\gamma) \quad& \text{and}\quad \Per(E\cap(\Omega\setminus A);\gamma) = 0\label{eqn: first perimeter identity}\\
            \Per(E_n;\gamma) &= 0 \label{eqn: second perimeter identity}\\
            \Per(E_n; \partial\Omega\setminus\gamma) &=  \Per(E; \partial\Omega\setminus\gamma) \label{eqn: third perimeter identity}
        \end{align}
        for all $n\in \N$.
        Starting with \eqref{eqn: first perimeter identity}, we use \cite[Theorem 16.3]{Mag12_SetsFinitePerimeter} to get
        \begin{equation}\label{eqn: perimeter of intersection full}\begin{split}
            \Per(E\cap(\Omega\setminus A);\partial\Omega) 
            &= \Per(E;(\Omega\setminus A)\one\cap \partial\Omega))
            + \Per(\Omega\setminus A; E\one\cap \partial \Omega) \\
            &\quad + \calH^{m-1}\big(\bigl\{\normal{E} = \normal{\Omega\setminus A}\bigr\} \cap \partial\Omega\big).
        \end{split}
        \end{equation}
        By De~Giorgi's Theorem~\ref{thm: De~Giorgi}, the first two terms on the right-hand side equal
        \begin{equation*}
            \calH^{m-1}(\partial^*E\cap (\Omega\setminus A)\one\cap \partial\Omega))+ \calH^{m-1}(\partial^*(\Omega\setminus A) \cap E\one\cap \partial \Omega)
        \end{equation*}
        and using that $((\Omega\setminus A)\one\cap \partial\Omega)\subset  (\Omega^{(1)}\cap\partial \Omega)$ and $( E\one\cap \partial \Omega)\subset (\Omega^{(1)}\cap\partial \Omega)$, the equality \eqref{eqn: Omega boundary regularity} implies that both terms are equal to zero. 
         For the third term, observe first that for every $x\in \set{\normal{E} = \normal{\Omega\setminus A}} \cap \partial\Omega$  by definition one has $x\in \partial^*(\Omega\setminus A)\cap \partial^*\Omega$ and since $\Omega\setminus A\subset\Omega$,  Lemma~\ref{lem: uniqueness of normal vector for subsets} implies that $\normal{\Omega\setminus A} (x) = \normal{\Omega} (x)$. Hence we get
        \begin{align*}
            \bigl\{\normal{E} = \normal{\Omega\setminus A}\bigr\} \cap \partial\Omega = \bigl\{\normal{E} = \normal{\Omega} \bigr\} \cap \partial^*(\Omega \setminus A) = \partial^*E\cap \partial^* \Omega  \cap \partial^* (\Omega \setminus A),
        \end{align*}
        where we applied Lemma~\ref{lem: uniqueness of normal vector for subsets} to $E\subset\Omega$ to get the second equality.
        Since with Theorem \ref{thm: Federer} we get that 
        $$\calH^{m-1}(\partial^* \Omega  \cap \partial^* (\Omega \setminus A)) = \calH^{m-1}((\partial\Omega\setminus\gamma)\cap \partial^*\Omega),$$  the first equality in \eqref{eqn: first perimeter identity} follows via
        \begin{equation*}
            \calH^{m-1}\big(\bigl\{\normal{E} = \normal{\Omega\setminus A}\bigr\} \cap \partial\Omega\ \big)
            % = \calH^{m-1}(\partial^*E\cap \partial^* \Omega  \cap \partial^* (\Omega \setminus A))
            =\calH^{m-1}(\partial^*E\cap \partial^* \Omega  \cap (\partial\Omega \setminus \gamma))
            =\Per(E;\partial\Omega\setminus\gamma).
        \end{equation*}
        Using the decomposition \eqref{eqn: perimeter of intersection full} for $\gamma$ instead of the full boundary $\partial\Omega$ and repeating the arguments thereafter, one finds that   $\Per(E\cap(\Omega\setminus A);\gamma) = 0$.
        We move on to \eqref{eqn: second perimeter identity}. Applying \eqref{eqn: perimeter union and intersection inequality} to $E_n$, we get
        \begin{equation*}
                \Per(E_n;\gamma) 
                % = \Per((E\cap (\Omega\setminus A)) \cup \hat{E}_n;\gamma)\\
                \leq  \Per(E\cap(\Omega\setminus A);\gamma) + \Per(\hat{E}_n;\gamma) -\Per(\hat{E}_n\cap(\Omega\setminus A);\gamma).
            \end{equation*}
        The first term equals zero by \eqref{eqn: first perimeter identity}. The remaining two terms vanish since $\hat{E_n}$ is compactly contained in $\Omega$ and thus, has zero perimeter on $\partial\Omega$.
        For the third identity \eqref{eqn: third perimeter identity}, we write $F:= E\cap(\Omega\setminus A)$ and invoke \cite[Theorem 16.3]{Mag12_SetsFinitePerimeter} once more to get
        \begin{equation*}
            \begin{aligned}
                \Per(E_n;\partial\Omega\setminus\gamma) 
                = \Per(\hat{E}_n; F^{(0)} \cap (\partial\Omega\setminus\gamma))+ \Per(F; \hat{E}_n^{(0)}\cap (\partial\Omega\setminus\gamma)) + \calH^{m-1}(\{\normal{\hat{E}_n} = \normal{F}\}\cap \partial\Omega\setminus\gamma)
            \end{aligned}
        \end{equation*}
        The first term vanishes since $\hat{E}_n\subset\subset \Omega$ has no perimeter on $\partial\Omega\setminus\gamma$. For the same reason one has  \smash{$\{\normal{\hat{E}_n} = \normal{F}\}=\emptyset$} and the third term vanishes as well. Finally, $\hat{E}_n\subset\subset\Omega$ implies that $(\partial\Omega\setminus\gamma)\subset\hat{E}_n^{(0)} $ and thus, the second term is equal to $\Per(E\cap(\Omega\setminus A);\partial\Omega\setminus\gamma)$. Using \eqref{eqn: first perimeter identity}, we obtain the claimed identity  \eqref{eqn: third perimeter identity}.

        We are left to prove the convergence \eqref{eqn: Dirichlet approximation of sets}. Since we clearly have $E_n\to  E$, it follows from lower semicontinuity of  relative perimeters in $\Omega\cup \gamma$ (see Remark \ref{rmk: compactness and lsc for mixed boundaries}) that  $\liminf \Per(E_n;\Omega) = \liminf \Per(E_n;\Omega\cup\gamma) \geq \Per(E;\Omega\cup\gamma)$.
        Using \eqref{eqn: perimeter union and intersection inequality}  we get
        \begin{equation*}
            \begin{gathered}
                \Per(E_n;\Omega) 
                % = \Per((E\cap (\Omega\setminus A)) \cup \hat{E}_n;\Omega)\\
                \leq  \Per(E\cap(\Omega\setminus A);\Omega) + \Per(\hat{E}_n;\Omega) -\Per((E\cap(\Omega\setminus A))\cap \hat{\Omega}_n;\Omega) .
            \end{gathered}
        \end{equation*}
         Taking  $\limsup$ on both sides as before  and using \eqref{eqn: boundary decomposition for perimeter}  as well as Step 1 yields 
        \begin{equation*}\begin{aligned} 
            \limsup \Per(E_n;\Omega) 
            &\leq  \Per(E\cap(\Omega\setminus A);\Omega) + \Per(E) -\Per(E\cap(\Omega\setminus A)) \\
            &= \Per(E) - \Per(E\cap(\Omega\setminus A);\partial\Omega)\\
            &=\Per(E;\Omega\cup\gamma),
        \end{aligned}
        \end{equation*}
        where the last identity is equivalent to \eqref{eqn: first perimeter identity}.
    \end{proof}

We briefly comment on related results to Theorem~\ref{thm: strict Dirichlet approximation of finite perimeter sets} and  applications in variational problems, complementing the discussion from the introduction. In general, strict (smooth) interior approximation results for finite perimeter sets in $\Omega$ can be viewed as the measure theoretic analog of strict density of (smooth) compactly supported functions in $BV(\Omega)$, but the techniques available for set approximations are fundamentally different. In particular, standard mollification procedures and smooth cut-off functions are not available. 

In this spirit, the approximation result for Euclidean $BV$-functions on Lipschitz domains by a sequence of functions in $BV(\R^m)$ with compact support given in \cite[Lemma 1]{IonLac05_GeneralizedCheegerSets}  is related to Theorem~\ref{thm: strict Dirichlet approximation of finite perimeter sets}.  However, this is not an \textit{interior} approximation (with respect to $\gamma$). Moreover, one should note that the approximating sets $E_n\subset\Omega$ in our setting cannot simultaneously be smoothly bounded and satisfy \eqref{eqn: Dirichlet approximation of sets} unless $\partial E\cap\gamma^*$ is smooth or $\calH^{m-1}(\partial E\cap \gamma^*)=0$.

\subsection{A capillarity problem in a container with inhomogeneous boundary}\label{sec: mixed boundary capillarity}
Capillarity problems model equilibrium states of liquids and the formation of droplets on surfaces. We refer to \cite[Chapter 19]{Mag12_SetsFinitePerimeter} for an introduction to the variational framework. If $\Omega$ is an open set in $\R^m$, the free energy of a liquid covering a subset $E$ of $\Omega$  is modeled by
\begin{equation*}
    \calF_\beta(E) := \Per(E;\Omega) + \beta \Per(E;\partial\Omega ) + \int_E v(x)\d x.
\end{equation*}
The parameter $\beta \in[-1,1]$ is an adhesion parameter between the container walls $\partial\Omega$ and the liquid, and the last term is a potential energy term. If $v\in C^0(\Omega)$ and $\hat{E}\in \scrC(\Omega)$ with $c:=|\hat{E}|<\infty$ minimizes $\calF_\beta$ among all finite perimeter sets $E$ in $\Omega$ with volume $ |E|=c$, then there exists some $\lambda\in \R$ such that $\hat{E}$ has distributional mean curvature $-v+\lambda$ in $\Omega$. If additionally, $\Omega$ has a $C^1$-boundary and $\Omega\cap \partial\hat{E}$ is a $C^2$-hypersurface with boundary, then $\hat{E}$ satisfies \textit{Young's law}:
\begin{equation*}
    \normal{\hat{E}} \cdot\normal{\Omega} = -\beta \text{ on } \partial(\Omega\cap \partial \hat{E}).
\end{equation*}
In other words, the contact angle $\theta$ between the liquid and the container walls satisfies $\beta =\cos(\theta)$. The term $\beta \Per(E;\partial\Omega )$ in $\calF_\beta$ is also referred to as the \textit{wetting energy} and $\partial\Omega$ can be understood as fully wetting if $|\beta| = 1$ and as hydrophobic if $\beta = 0$. Various aspects of capillarity problems and the related prescribed mean curvature problems have been studied extensively in recent years, see \cite{PasPoz24_QuantitativeIsoperimetricInequalities, LeeParPyo25_CapillaryStableMinimal, DeFusMor24_RegularityCapillarityDroplets, FusJulMorPra25_IsoperimetricInequalityCapillary, Sch25_IsoperimetricConditionsLower}, just to mention a few.

The authors in \cite{LiZhoZhu25_MinmaxTheoryCapillary} and \cite{HonSat23_CapillarySurfacesStability} study regularity and stability of capillary surfaces in $3$-manifolds. Motivated by this, we  propose a variational formulation for a capillary problem with \textit{mixed boundary  in a Riemannian manifold} and use the results from Section \ref{sec: mixed bvp and strict approximation} to prove existence of minimizers and use the approximation from  Theorem~\ref{thm: strict Dirichlet approximation of finite perimeter sets} to construct recovery sequences in a Gamma-convergence problem for capillarity functionals with varying adhesion parameters $\beta$.

\begin{lemma}\label{lem: existence for mixed boundary capillarity}
Let $\Omega$ satisfy \ref{ass: gauss green domain} and let $\gamma\subset\partial\Omega$ be $\calH^{m-1}$-equivalent to a relatively open subset of $\partial\Omega$.  For $\beta \in[0,1]$, $v\in L^1(\Omega)$ and measurable sets $E\subset\Omega$, let
\begin{equation}
    \calF_\beta^\gamma(E) = \Per(E;\Omega) + \Per(E;\gamma) + \beta \Per(E;\partial\Omega\setminus\gamma) + \int_E v\d\vol
\end{equation}
    Then the functional $\calF_\beta^\gamma$ is lower semicontinuous with respect to convergence in volume measure on $\Omega$. Moreover, for every $c\in (0,|\Omega|)$, $\calF^\gamma_\beta$ admits a minimizer in  $\set{E\in \scrC(\Omega):\ |E|=c}$.
\end{lemma}
\begin{proof}
    We rewrite
    \begin{equation*}
        \calF_\beta^\gamma(E) = (1-\beta)\Per(E;\Omega\cup\gamma) + \beta\Per(E) + \int_E v\d\vol.
    \end{equation*}
    By virtue of  the lower semicontinuity statement for mixed boundaries in Remark \ref{rmk: compactness and lsc for mixed boundaries} and the classical lower semicontinuity property of perimeters on open sets, the first two summands are lower semicontinuous and nonnegative functionals. Using that  the third summand is continuous with respect to $E$, the sum of all three functionals is lower semicontinuous, too. For arbitrary subsets  $E\in \Omega$, the value $\calF_\beta^\gamma(E)$ is bounded below by $-\norm[\Omega,1]{v}$. Therefore, for every $c\in(0,|\Omega|)$,  
    $$m_\beta(c):=\inf\set{ \calF_\beta^\gamma(E):\, E\subset\Omega,\,|E|=c}$$ 
    takes a value in $\R$ and we may pick an infimizing sequence of sets $E_n\subset \Omega$ with $|E|=c$ such that $\calF_\beta^\gamma(E_n)\to m_\beta(c)$. By definition of $\calF_\beta^\gamma$ and nonnegativity of perimeters, it follows that for all $n$,
    \begin{equation*}
        \Per(E_n;\Omega) \leq m_\beta(c) +  \norm[\Omega,1]{v}<\infty,
    \end{equation*}
     hence the sequence is uniformly bounded in $\scrC(\Omega)$ and since $\Omega$ is a bounded extension domain, we may argue as for the compactness property in Remark \ref{rmk: compactness and lsc for mixed boundaries} combined with Remark \ref{rmk: compactness of Caccioppoli sets} to find that $(E_n)$ converges in measure to some $E\in \scrC(\Omega)$ along a subsequence and clearly, $|E|=c$.
     Finally, invoking lower semicontinuity of $\calF_\beta^\gamma$, we conclude that $\calF_\beta^\gamma(E) =m_\beta(c)$. 
\end{proof}

We want to examine the behavior of the minimizer under continuous changes of the prescribed contact angle $\theta$ via the  parameter $\beta$. To this end, we prove the following convergence result where we will crucially depend on the approximation Theorem~\ref{thm: strict Dirichlet approximation of finite perimeter sets} for the construction of a recovery sequence.
\begin{theorem}
 Suppose that the assumptions of Lemma~\ref{lem: existence for mixed boundary capillarity} hold. Assume furthermore that \ref{ass: approximation domain} and \ref{ass: gamma dirichlet boundary} are satisfied. 
    If $(\beta_n)\to \beta$ in $[0,1]$, then the functionals $\calF^\gamma_{\beta_n}$ Gamma-converge to $\calF_\beta^\gamma$ on the space of measurable subsets of $\Omega$.
\end{theorem}

\begin{proof}
    Let $E\subset\Omega$ be arbitrary. We first prove the lower semicontinuity estimate and assume that $(E_n)$ is a sequence of measurable subsets of $\Omega$ that converges to $E$ in volume measure. Using the same arguments as in the proof of Lemma~\ref{lem: existence for mixed boundary capillarity}, we immediately get
    \begin{equation*}
    \begin{aligned}
       \liminf_n \calF^\gamma_{\beta_n}(E_n) 
       &\geq \lim_n (1-\beta_n) \liminf_n \Per(E_n;\Omega\cup\gamma) + \lim_n\beta_n\liminf_n\Per(E_n)+ \lim_n\int_{E_n} v\d\vol\\
       &\geq  (1-\beta)\Per(E;\Omega\cup\gamma) + \beta \Per(E) + \int_E v\d\vol \\
       &= \calF_\beta^\gamma(E).
       \end{aligned}
    \end{equation*}
    We conclude the proof by constructing a recovery sequence. Indeed, for arbitrary fixed $E$, let $(E_n)$ be the sequence from  Theorem~\ref{thm: strict Dirichlet approximation of finite perimeter sets}. Then  the construction directly yields
    \begin{equation*}
       \lim_n \calF^\gamma_{\beta_n}(E_n) 
       = \lim_n    \left( \Per(E_n;\Omega\cup\gamma) + \beta_n \Per(E_n;\partial\Omega\setminus\gamma) + \int_{E_n} v\d\vol \right)
        = \calF_\beta^\gamma(E).
    \end{equation*}
\end{proof}

\appendix
\section{}\label{appendix}

%%%%%%%%%%%%%%%%%%%%%%%%%%%%%%%%%%%%%%%%%%%%%%%%%%%%%%%%%%%%
% \section{Proofs for some auxiliary results} \label{appendix: proofs for auxuliary results}
%%%%%%%%%%%%%%%%%%%%%%%%%%%%%%%%%%%%%%%%%%%%%%%%%%%%%%%%%%%%
For an arbitrary subset $A\subset M$ we write $C(A)$ for the space of continuous functions on $A$ with respect to the subspace topology. Let  $C_c(A)$ be the space of continuous functions with compact support in $A$ and $C_b(A)$ the space of uniformly bounded continuous functions on $A$. We denote the space of continuous functions \textit{vanishing at infinity} by $C_0(A)$
 and clearly one has $C_c(A)\subset C_0(A)\subset C_b(A)$. If $A$ is locally compact, then $C_0(A)$ is the closure of $C_c(A)$ with respect to uniform convergence.
 Given an open set $\Omega\subset M$ and $k\in \N\cup\set{\infty}$, we denote the space of $k$-times continuously differentiable functions on $\Omega$  by  $C^k(\Omega)$ and for arbitrary sets $A\subset M$ we let  
$$
    C^k(A) = \big\{u:A\to \R :\, \text{ there exist } U\subset M \text{ open and }  \hat{u}\in C^k(U) \text{ s.t. }  A\subset U \tand\hat{u}_{|A} = u \big\}
$$
and define $C^k_c(A)\subset C_0^k(A)\subset C^k_b(A)$ as above.
Similarly we define
$\Gamma_{C^\infty}(A;TM)$ to be the space of sections $X:A\to TM$ for which there exists an open subset $U\subset M$ containing $A$ such that $X$   extends to a section in $\Gamma_{C^\infty}(U;TM) = \Gamma_{C^\infty}(TU) $. 

\begin{lemma}\label{lem: smooth and cts functions on LCH spaces} Let $\calX$ be a locally compact Hausdorff space and let $A\subset \calX$ be locally compact in the subspace topology.
    \begin{enumerate}[leftmargin = *,   labelindent = 0em,]
        \item  Every  $u\in C_c(A)$ can be extended to a function $\tilde{u}\in C_c(\calX)$.
        \item Suppose $\calX=M$. 
        Then $C_c^\infty(A)$ is dense in $C_c(A)$ (and $C_0(A)$) with respect to uniform convergence.
        \item There exists a countable set $\calD\subset C_c^\infty(A)$ which is dense in $C_0(A)$.
    \end{enumerate}
\end{lemma}
\begin{proof}
    1. Let $K:=\sppt(u)$. Then $K$ is a compact subset of the LCH space $A$. By \cite[Proposition 4.31]{Fol99_RealAnalysisModern} there exists a relatively open, relatively compact set $U_A$ in $A$ such that
    $$
    K\subset \subset U_A \subset \mathrm{cl}_A(U_A) \subset \subset A \subset \calX.
    $$
    Here $\mathrm{cl}_A$ denotes the closure of a set in the subspace topology of $A$, but since $\mathrm{cl}_A(U_A)$ is a compact subspace of $A$, it is a compact subspace of $\calX$ and therefore, $\mathrm{cl}_A(U_A) = \mathrm{cl}(U_A)$.  
    By definition of the subspace topology, $U_A$ is open in $A$ if and only if there exists an open set $U\subset \calX$ s.t. $U_A = U\cap A$.

    We apply Tietze's extension theorem for locally compact Hausdorff spaces \cite[Theorem 4.34]{Fol99_RealAnalysisModern} to $u\in C(\mathrm{cl}(U_A))$ where we consider $\mathrm{cl}(U_A)\subset\subset \calX$. That is, we let $\hat{u}\in C(\calX)$ s.t. $\hat{u}_{|\mathrm{cl}(U_A)} = u$.
    Moreover, by Urysohn's lemma \cite[Lemma 4.32]{Fol99_RealAnalysisModern}, since $K\subset\subset U$ and $U$ is open in $\calX$, there exists $f\in C(\calX)$ such that $f(x)=1$ for $x\in K$ and $f(x) = 0$ for all $x\in \calX\setminus U$. We define $\tilde{u} = \hat{u}f\in C(\calX)$ and write $A = (A\setminus U) \cup (A\cap U)$. If $x\in A\setminus U$, then $f(x) = 0$, hence $\tilde{u}(x) = 0$ and since $K = \sppt(u) \subset U$, $u(x) = 0$. Now assume $x\in A\cap U = U_A$, then $\hat{u}(x) = u(x)$ and if $x\notin K$, then $\tilde{u}(x) = \hat{u}(x)f(x) = 0 = u(x)$. Finally, if
    $x\in K$, then $f(x) = 1$, hence $\tilde{u}(x) = \hat{u}(x)f(x) = u(x)$. 

    2. Fix $u\in C_c(A)$ as above and let $\tilde{u}\in C_c(M)$ be the continuous extension to $M$ with $U$ and $U_A$ as above. Then $\sppt(\tilde{u})$ is compactly contained in $U$ with $U\subset M$ open, hence by using Friedrichs mollifiers there is a sequence of functions $\tilde{u}_n\in C_c^\infty(M)$ with $\sppt(\tilde{u}_n)\subset\subset U$ converging uniformly to $\tilde{u}$ on $M$. Therefore, by setting $u_n :{=\tilde{u}_n}_{|A}$, we obtain a sequence in $C^\infty(A)$ approximating $u$ uniformly on $A$ and one has 
    $$
    \sppt(u_n) 
    =    \sppt(\tilde{u}_n)\cap A
    = (\sppt(\tilde{u}_n)\cap \mathrm{cl}(U_A) )
    \subset A.
    $$
    Hence, $\sppt(u_n)$ is compact as the intersection of two compact sets and we conclude that $u_n\in C_c^\infty(A)$. Density in $C_0(A)$ follows density of $C_c(A)$ in $C_0$.

    3. Since $A$ is locally compact, $C_0(A)$ is a separable Banach space (cf. \cite{Cho12_NotesSeparability$C^ast$Algebras}). Now the claim follows from density of $C_c^\infty(A)$ in $C_0(A)$.
    \end{proof}

\section*{Declarations}
\textbf{Competing interests:} On behalf of all authors, the corresponding author states that there is no conflict of interest. 

\textbf{Data availability:} Data sharing is not applicable to this article as no datasets were generated or analyzed.

\bibliography{references}
\bibliographystyle{abbrv}

\end{document}